\documentclass{amsart}
\usepackage{amsmath}
\usepackage{amsthm}
\usepackage{amssymb}
\usepackage{graphics}
\usepackage{bm}
\usepackage{setspace}

\usepackage{amsfonts}
\usepackage{eucal}
\usepackage{latexsym}
\usepackage{mathdots}
\usepackage{mathrsfs}
\usepackage{enumitem}

\usepackage{color}

\newcommand{\be}{\begin{equation}}
\newcommand{\ee}{\end{equation}}
\newcommand{\ba}{\begin{eqnarray}}
\newcommand{\ea}{\end{eqnarray}}
\newcommand{\bi}{\begin{itemize}}
\newcommand{\ei}{\end{itemize}}
\newcommand{\bn}{\begin{enumerate}}
\newcommand{\en}{\end{enumerate}}
\newcommand{\bbm}{\begin{bmatrix}}
\newcommand{\ebm}{\end{bmatrix}}
\newcommand{\bp}{\begin{proof}}
\newcommand{\ep}{\end{proof}}
\newcommand{\nn}{\nonumber}

\newcommand{\mr}{\ensuremath{\mathrm}}
\newcommand{\scr}{\ensuremath{\mathscr}}
\newcommand{\mbf}{\ensuremath{\mathbf}}
\newcommand{\mc}{\ensuremath{\mathcal}}

\newcommand{\ov}{\ensuremath{\overline}}

\newcommand{\wt}{\ensuremath{\widetilde}}
\renewcommand{\bm}{\ensuremath{\mathbb }}

\newcommand{\Ga}{\ensuremath{\Gamma}}
\newcommand{\ga}{\ensuremath{\gamma}}

\newcommand{\La}{\ensuremath{\Lambda }}
\newcommand{\la}{\ensuremath{\lambda }}

\newcommand{\ka}{\ensuremath{\kappa }}

\def\C{\mathbb{C}}

\def\D{\mathbb{D}}
\def\T{\mathbb{T}}

\def\N{\mathbb{N}}
\def\B{\mathbb{B}}

\newcommand{\n}{\ensuremath{\mathbf{n} }}

\newcommand{\J}{\ensuremath{\mathcal{J} }}

\newcommand{\K}{\ensuremath{\mathcal{K} }}

\newcommand{\F}{\ensuremath{\mathbb{F} }}

\newcommand{\ip}[2]{\ensuremath{\langle {#1} , {#2} \rangle}}

\renewcommand{\dim}[1]{\ensuremath{\mathrm{dim} \left( {#1} \right) }}
\newcommand{\ran}[1]{\ensuremath{\mathrm{Ran} \left( {#1} \right) }}
\newcommand{\ra}[1]{\ensuremath{\mathcal{R} \left( {#1} \right) }}
\renewcommand{\ker}[1]{\ensuremath{\mathrm{Ker} ({#1}) }}

\newcommand{\ang}[1]{\ensuremath{\langle{#1}\rangle }}

\numberwithin{equation}{section}
\numberwithin{subsection}{section}

\newtheorem{thm}[subsection]{Theorem}

\newtheorem{lemma}[subsection]{Lemma}
\newtheorem{prop}[subsection]{Proposition}
\newtheorem{cor}[subsection]{Corollary}
\newtheorem{thm*}{Theorem}

\theoremstyle{definition}
\newtheorem{defn}[subsection]{Definition}
\newtheorem{remark}[subsection]{Remark}
\newtheorem{eg}[subsection]{Example}

\title[Gleason model]{A Gleason solution model for row contractions}

\author[R.T.W. Martin]{Robert T.W. Martin}
\address{Department of Mathematics and Applied Mathematics\\
University of Cape Town}
\email{rtwmartin@gmail.com}

\author[A. Ramanantoanina]{Andriamanankasina Ramanantoanina}
\address{Department of Mathematics and Applied Mathematics\\
University of Cape Town}
\email{andriamanak@gmail.com}

\dedicatory{Dedicated to, and inspired by research of, Joseph A. Ball.}

\thanks{The first author acknowledges support of NRF CPRR Grants 90551 and 105837.}

\begin{document}
\small
\maketitle

\begin{abstract}

    In the de Branges-Rovnyak functional model for contractions on Hilbert space, any completely non-coisometric (CNC) contraction is
represented as the adjoint of the restriction of the backward shift to a de Branges-Rovnyak space, $\scr{H} (b)$, associated to a contractive analytic
operator-valued function, $b$, on the open unit disk.

We extend this model to a large class of CNC contractions of several copies of a Hilbert space into itself (including all CNC row contractions with commuting component operators). Namely, we completely characterize the set of all CNC row contractions, $T$, which are unitarily equivalent to an extremal Gleason solution for a de Branges-Rovnyak space, $\scr{H} (b_T)$, contractively contained in a vector-valued Drury-Arveson space of analytic functions on the open unit ball in several complex dimensions. Here, a Gleason solution is the appropriate several-variable analogue of the adjoint of the restricted backward shift and the characteristic function, $b_T$, belongs to the several-variable Schur class of contractive multipliers between vector-valued Drury-Arveson spaces. The characteristic function, $b_T$, is a unitary invariant, and we further characterize a natural sub-class of CNC row contractions for which it is a complete unitary invariant.

\end{abstract}

\section{Introduction}

The de Branges-Rovnyak and Sz.-Nagy-Foia\c{s} functional models are two widely-used and powerful approaches to the representation theory of contractions on Hilbert space \cite{dBss,dBmodel,NF}.  These two constructions provide equivalent models for completely non-unitary (CNU) contractions \cite{Ball1987,Nik1986,Sarason-dB,Ball2011dBR}.  In this paper we focus on the de Branges-Rovnyak model and its several-variable extension to the setting of contractions from several copies of a Hilbert space into itself.

A linear contraction, $T: \mc{H} \rightarrow \J$, between Hilbert spaces, $\mc{H}, \J$, is called \emph{completely non-coisometric} (CNC) if it has no co-isometric restriction to a non-trivial subspace. In the de Branges-Rovnyak model for CNC contractions on a (single) Hilbert space, the model operator acts on the (unique) reproducing kernel Hilbert space (RKHS), $\mc{H} (k ^b)$, associated to an operator-valued contractive analytic function $b$ on the open unit disk, $\D$, in the complex plane \cite{Ball1987,Ball2011dBR}. Here the positive, sesqui-analytic de Branges-Rovnyak kernel $k^b$ is
$$ k ^b (z,w) := \frac{ I - b(z) b(w) ^*}{1-zw^*}.$$ In the above $z^* := \ov{z}$ denotes complex conjugation. The RKHS
$\scr{H} (b) := \mc{H} (k ^b)$ is called the de Branges-Rovnyak space of $b$, and in the case where $b \equiv 0$ we recover the Szeg\"{o} kernel for the classical Hardy space of analytic functions in the unit disk.

Recall that the \emph{shift}, $S :H^2 (\D ) \rightarrow H^2 (\D )$ is the canonical isometry of multiplication by $z$ and its adjoint, the \emph{backward shift}, $S^*$, acts as the difference quotient: $$ (S^* h) (z) = \frac{h(z) - h(0)}{z}; \quad \quad h \in H^2 (\D).$$  The shift plays a central role in the classical theory of Hardy spaces \cite{Hoff,Nik-shift,GR-model}. If $T$ is any CNC contraction, there is a (essentially unique) contractive, operator-valued analytic function, $b_T$, on the unit disk, $b_T (z) \in \scr{L} (\mc{H} , \K)$ (\emph{i.e.} a member of the operator-valued \emph{Schur class}), so that $T$ is unitarily equivalent to $X$ where $X ^* := S^* | _{\scr{H} (b _T)}$ is the restriction of the backward shift
of the vector-valued Hardy space $H^2 (\D ) \otimes \K$ to the de Branges-Rovnyak space $\scr{H} (b _T)$ \cite{Ball2011dBR,Ball1987}. (Any de Branges-Rovnyak space $\scr{H} (b)$ associated to a contractive, operator-valued analytic function, $b$, on the disk is always contractively contained in vector-valued Hardy space and is always co-invariant for the shift \cite{Sarason-dB}.)  This provides a natural model for CNC contractions as adjoints of restrictions of the backward shift to de Branges-Rovnyak reproducing kernel Hilbert spaces, and this is the model we extend to several variables in this paper.

A canonical several-variable extension of the Hardy space of the disk is the Drury-Arveson space, $H^2 _d$, the unique RKHS of analytic functions on the open unit ball $\B ^d := (\C ^d ) _1$ corresponding to the several-variable Szeg\"{o} kernel. (If $Y$ is a Banach space, let $(Y) _1$ denote the open unit ball and let $[Y]_1$ denote the closed unit ball.)  The Schur classes of contractive, operator-valued functions on the disk are promoted to the multi-variable Schur classes, $\scr{S} _d (\J , \K)$, of contractive, operator-valued multipliers between vector-valued Drury-Arveson spaces $H^2 _d \otimes \J$ and $H^2 _d \otimes \K$ (see Subsection \ref{RKHSsection}), and the appropriate analogue of the adjoint of the restricted backward shift in this several-variable setting is a contractive solution to the \emph{Gleason problem} in $\scr{H} (b)$: A contraction $X = (X_1 , ... , X_d ) : \scr{H} (b) \otimes \C ^d \rightarrow \scr{H} (b)$ is called a \emph{Gleason solution} (or a solution to the Gleason problem) for $\scr{H} (b)$ if its adjoint acts as the multi-variable difference quotient:
$$ (zX^* f) (z) := z_1 (X_1 ^* f) (z) + ... + z_d (X_d ^* f) (z) = f(z) - f(0); \quad \quad \forall f \in \scr{H} (b).$$  The concepts of \emph{contractive} and extremal contractive (or more simply, \emph{extremal}) Gleason solution are further norm/positivity constraints and we will review these basic definitions in the upcoming Subsection \ref{Gleasonsub}, see Definition \ref{Gledef}.

Given any contraction between Hilbert spaces, $T  : \mc{H} \rightarrow \K$, recall that the defect operator, $D _T$ of $T$ is defined as
$$D _T := \sqrt{I - T^* T}.$$ We say that a \emph{row contraction}, \emph{i.e.} a contraction from several copies of a Hilbert space into itself, $T = (T_1, ..., T_d) : \mc{H} \otimes \C ^d \rightarrow \mc{H}$, obeys the \emph{commutative CNC condition}, and we write: $T$ is CCNC if
$$ \mc{H} = \bigvee _{z \in \B ^d } (I -T z ^* ) ^{-1} \ran{D_{T^*}}. $$ Here, and throughout, $\bigvee$ denotes closed linear span. We will prove that any CCNC row contraction $T$ is automatically CNC (Corollary \ref{CCNCcor}), and that any \emph{$d$-contraction} (a row contraction with $d$ mutually commuting component operators) is CNC if and only if it is CCNC. One of the main results of this paper is the extension of the de Branges-Rovnyak model for CNC contractions to the class of all CCNC row contractions:
\begin{thm*}{ (Theorem \ref{main1})} \label{main1intro}
    A row contraction $T :\mc{H} \otimes \C ^d \rightarrow \mc{H}$ is CCNC (obeys the commutative CNC condition) if and only if it is unitarily
equivalent to an extremal (contractive) Gleason solution in a multi-variable de Branges-Rovnyak space $\scr{H} (b)$ for a Schur-class multiplier $b \in \scr{S} _d (\J , \K)$.

If $T$ is unitarily equivalent to an extremal Gleason solution $X^b$ for $\scr{H} (b)$, then the characteristic function $b := b_T$ is a unitary invariant for $T$:
If $T_1, T_2$ are unitarily equivalent CCNC row contractions, then $b_{T_1}$ coincides weakly with $b_{T_2}$. One can choose $b_T \in \scr{S} _d (\ran{D_T} , \ran{D_{T^*}} )$.
\end{thm*}
\begin{remark}
In the above, two row contractions $T ^{(k)} : \mc{H} _k \otimes \C ^d \rightarrow \mc{H} _k$, $k=1,2$ are said to be unitarily equivalent, denoted $T ^{(1)} \simeq T ^{(2)}$,
if there is an onto isometry $U : \mc{H} _1 \rightarrow \mc{H} _2$ so that
$$ U T ^{(1)} = T ^{(2)} (U \otimes I_d ). $$ Equivalently, $U T ^{(1)} _j U^* = T^{(2)} _j; \quad 1 \leq j \leq d$. Two operator-valued Schur class functions $b_k \in \scr{S} _d (\mc{H} _k , \J _k )$ are said to \emph{coincide weakly} if there is an onto isometry $U : \J _1 \rightarrow \J _2$ so that
$$ U b_1 (z) b_1 (w) ^* U^* = b_2 (z) b_2 (w) ^* ; \quad \quad \forall \ z,w \in \B ^d. $$ See Definition \ref{Livsicdef} and Definition \ref{fullchardef} for the definitions of the Schur-class characteristic functions of a CNC row partial isometry, and an arbitrary CNC row contraction, respectively.
\end{remark}
\begin{remark} \label{BBFremark}
As we will show in Subsection \ref{BBFsect}, the above theorem is an easy consequence of the colligation and transfer-function realization theory of Ball-Bolotnikov-Fang \cite{Ball2007trans,Ball2010}.
\end{remark}

In \cite{Jur2014AC,JM}, the concept of a \emph{quasi-extreme multiplier} for Drury-Arveson space was introduced. This concept is a several-variable extension
of a `Szeg\"{o} approximation property', the salient idea being that this property is equivalent to being an extreme point of the Schur class in the
classical, single-variable, scalar-valued setting (see Section \ref{QErowsect} and Definition \ref{QEdef}). Recently, it has been shown that quasi-extreme implies extreme in the scalar-valued, several-variable setting as well \cite{JMqe}. If $b$ is a quasi-extreme Schur multiplier, then $\scr{H} (b)$ has a unique contractive (and extremal) Gleason solution (see Theorem \ref{equivqe}). We say that a CCNC row contraction $T$ is \emph{quasi-extreme} (QE) if its characteristic function $b_T$ is a quasi-extreme multiplier.
Our second main result characterizes the class of all QE row contractions:
\begin{thm*}{ (Theorem \ref{main2}, Theorem \ref{QEconditionthm})} \label{main2intro}
    A row contraction $T : \mc{H} \otimes \C ^d \rightarrow \mc{H}$ is QE if and only if it is CCNC and obeys the QE condition:
$$ \ker{T} ^\perp \subseteq \bigvee _{z \in \B ^d} z ^* (I -Tz^*)^{-1} \ran{D_{T^*}}. $$
This happens if and only if $T$ is unitarily equivalent to the (unique) extremal Gleason solution in a de Branges-Rovnyak space $\scr{H} (b)$
for a quasi-extreme $b \in \scr{S} _d (\J , \K)$. The QE characteristic function $b_T :=b$ of a QE row contraction $T$ is a complete unitary invariant: two QE row contractions $T_1, T_2$ are unitarily equivalent if and only if their characteristic functions coincide weakly.
\end{thm*}
In summary, given the following strict hierarchy of classes of row contractions on Hilbert space,
$$ CNU \supsetneq CNC \supsetneq CCNC \supsetneq QE, $$ this paper constructs a commutative de Branges-Rovnyak model for the CCNC and QE classes.

Previous work on the representation theory and functional models for row contractions include \cite{Popmodel,Popchar,Ball2005,Ball2008,Ball2011dBR,BES2005char,BES2006cnc}. The theory of Popescu \cite{Popmodel,Popchar} for CNC row contractions constructs a Sz.-Nagy-Foia\c{s}-type model by studying the structure of the space of the minimal row isometric dilation of the row contraction, and defines a non-commutative `characteristic function' (\emph{i.e.} an element of the free Schur class, see \cite{Ball2006Fock}) which is a complete unitary invariant. This theory is extended to arbitrary CNU row contractions by Ball-Vinnikov in \cite{Ball2005}.

The papers \cite{Ball2011dBR,BES2005char,BES2006cnc} of Ball-Bolotnikov and Bhattacharyya-Eschmeier-Sarkar study the model theory of $d-$contractions (row contractions with $d$ mutually commuting component operators).  In \cite{BES2005char,BES2006cnc}, the classical Sz.-Nagy-Foia\c{s} model and the definition of the Sz.-Nagy-Foia\c{s} characteristic function is extended to CNC $d$-contractions, and \cite{Ball2011dBR} recovers this theory as a special case of the general model developed in \cite{Ball2005} for arbitrary CNU row contractions. In this model theory for $d$-contractions, the characteristic function is an element of the (operator-valued, several-variable) Schur classes, and this characteristic function is a complete unitary invariant for CNC $d$-contractions \cite[Theorem 3.6]{BES2006cnc}, \cite[Theorem 5.2]{Ball2011dBR}. In particular, \cite[Theorem 5.7]{Ball2011dBR}, proves that $T$ is a CNC $d$-contraction if and only if $T$ is unitarily equivalent to a contractive Gleason solution, $X^T$ acting on $\scr{H} (b_T)$, where $b_T$ is the characteristic function of $T$. In this case, since $T$ has commuting component operators, the results \cite[Theorem 3.5, Theorem 3.6]{Ball2008} imply that $\scr{H} (b^T) \subset H^2 _d \otimes \K$ is co-invariant for the (vector-valued) Arveson $d$-shift, $S\otimes I_\K$, that $(X^T) ^* = (S^* \otimes I_\K ) | _{\scr{H} (b^T) }$, and that $X^T$ is the unique contractive Gleason solution for $\scr{H} (b^T)$ which has commuting component operators (\emph{i.e.} which is a $d$-contraction).

More generally, the seminal work of Ball-Bolotnikov-Fang on transfer-function realization theory for the multi-variable Schur class provides a commutative de Branges-Rovnyak model for arbitrary CCNC row contractions \cite{Ball2007trans,Ball2008,Ball2010,Ball2011dBR} (see Subsection \ref{BBFsect}).  This series of papers clearly demonstrates that the appropriate multi-variable generalization of the notion of restriction of the backward shift to a de Branges-Rovynak space is (the adjoint of) a contractive Gleason solution, and that contractive Gleason solutions should play the role of the model operator in a several-variable de Branges-Rovnyak model.  Moreover, as we will see in the upcoming Subsection \ref{BBFsect}, Theorem \ref{main1intro} follows directly from the transfer-function realization theory of \cite{Ball2007trans}, and this provides a (commutative) multi-variable de Branges-Rovnyak model for arbitrary CCNC row contractions.

\subsection{Overview}

We will provide an alternate approach to Theorem \ref{main1intro} which is inspired by classical work of Kre\u{\i}n and Liv\v{s}ic \cite{Krein1944,Krein1944one,Krein1949,Krein,Livsicone,Livsic}. M.S. Liv\v{s}ic developed a functional model and the notion of a characteristic function for arbitrary CNU partial isometries with equal defect indices in \cite{Livsic,Livsicone}. Similarly, in a series of papers, M.G. Kre\u{\i}n constructed a general functional model for unbounded symmetric operators with equal defect indices \cite{Krein1944,Krein1944one,Krein1949} (the reference \cite{Krein} provides an overview of this theory). Recall here, that the classes of partial isometries and symmetric linear transformations on Hilbert space are equivalent under a natural bijection, the \emph{Cayley Transform}, a fractional linear transformation of the closed complex upper half-plane onto the closed unit disk (minus the point $\{ 1 \}$) \cite[Chapter VII]{Glazman}. In \cite{AMR,GMR,Martin-ext}, several key aspects of these approaches were combined and slightly generalized with the introduction of the concept of a \emph{model map} for a symmetric linear transformation or partial isometry with equal defect indices \cite[Section 2.1]{AMR}, \cite[Section 3]{GMR}, \cite[Definition 5.1]{Martin-ext}.

In Section \ref{modelmapsect} we extend this approach to several-variables and apply it to develop the representation theory of CCNC row partial isometries.   Along the way we will obtain several results of independent interest.  Corollary \ref{Gmultbyz} proves that any extremal Gleason solution for a multivariable de Branges-Rovnyak space $\scr{H} (b)$ with $b(0) =0$ acts as multiplication by the independent variables on its initial space (the projection onto the initial space prevents this from being a $d$-contraction in general). Our definition, Definition \ref{Livsicdef}, of the (operator-valued) Schur-class characteristic function, $b_V$, of any CCNC row partial isometry, $V$, is a direct multi-variable analogue of Liv\v{s}ic's original definition from \cite{Livsic}. In Section \ref{CCNCmodelsect} we study operator-M\"{o}bius transformations (Frostman shifts) of Schur-class functions, and we show that if $T$ is any CCNC row contraction with canonical decomposition $T=V-C$ where $V=T P _{\ker{D_T}}$ is a CCNC row partial isometry, that the characteristic function $b_V$, of $V$, is the Frostman shift of $b_T$ which vanishes at $0$.

Our Liv\v{s}ic-type characteristic function, $b_T$, is equivalent to the Sz.-Nagy-Foia\c{s}-type characteristic function of $T$ as constructed for CNC $d$-contractions in \cite{BES2005char,BES2006cnc,Ball2011dBR} (see Proposition \ref{NFprop}). Theorem \ref{main1intro} shows that the map $T \mapsto b_T$ is a surjection of CCNC row contractions onto (weak equivalence classes of) multi-variable operator-valued Schur-class functions, and that $T$ is in the inverse image of $b \in \scr{S} _d (\J , \K)$ if and only if it is unitarily equivalent to some extremal Gleason solution for $\scr{H} (b)$. The characteristic function is known to be a complete unitary invariant for the class of commutative CCNC row contractions (CNC $d$-contractions) \cite{BES2005char,BES2006cnc,Ball2011dBR}. The characteristic function is still a unitary invariant for the general class of all CCNC row contractions (see Theorem \ref{main1}), but it is not a complete unitary invariant (see Subsection \ref{Gleg}). Theorem \ref{main2} shows that the characteristic function is also a complete unitary invariant for the smaller sub-class of QE row contractions (which are again generally non-commuting). 

\section{Preliminaries and Background}

\subsection{Vector-valued RKHS} \label{RKHSsection}

    We will be using the theory of vector-valued reproducing kernel Hilbert spaces throughout this paper, as presented in \emph{e.g.} \cite{Paulsen-rkhs}.
Recall the following basic facts from RKHS theory:

Given a set $X$, and an auxiliary Hilbert space $\mc{H}$, an $\mc{H}$-valued RKHS $\K$ on $X$ is a Hilbert space of $\mc{H}$-valued functions on $X$ so that for any $x \in X$ the linear point evaluation maps
$K_x ^* \in \scr{L} ( \K , \mc{H} )$ defined by $$ K_x ^* F = F(x) \in \mc{H}; \quad \quad  F \in \K, $$ are bounded. We write $K_x := (K _x ^* ) ^* \in \scr{L} (\mc{H} , \K )$ for the Hilbert space adjoint. The
operator-valued function $ K : X \times X \rightarrow \scr{L} (\mc{H} )$:
$$ K (x,y) :=  K_x ^* K_y \in \scr{L} (\mc{H} ); \quad \quad x,y \in X, $$ is called the \emph{reproducing kernel} of $\K$. One usually writes $\K = \mc{H} (K )$. The reproducing kernel $K$ of any vector-valued RKHS on $X$ is
a \emph{positive kernel function} on $X$: A function $K : X \times X \rightarrow \scr{L} (\mc{H})$ is an operator-valued positive kernel function on $X$ if for any finite set of distinct points $\{ x_k \} _{k=1} ^N \subset X$, the matrix
$$ [ K (x_i , x_j ) ]  \in \scr{L} (\mc{H} ) \otimes \C ^{N\times N}, $$ is positive semi-definite. The vector-valued extension of the theory of RKHS developed by Aronszajn and Moore (see \emph{e.g.} \cite{Paulsen-rkhs}) shows that there is a bijection between positive $\scr{L} (\mc{H} )$-valued kernel functions on $X\times X$ and RKHS of $\mc{H}$-valued functions on $X$. Namely, given any positive kernel $K$ on $X$ there is a unique RKHS $\K$ on $X$ so that $K$ is its reproducing kernel, $\K = \mc{H} (K)$.

Any two RKHS $\mc{H} (k)$, $\mc{H} (K)$ with $\scr{L} (\J)$ and $\scr{L} (\K)$-valued positive kernel functions $k,K$ on some set $X$, respectively, can be naturally equipped with a \emph{multiplier space}:
$$ \mr{Mult} (\mc{H} (k) , \mc{H} (K) ) := \{ F :  X \rightarrow \scr{L} (\J , \K) | \ F h \in \mc{H} (K) \ \forall h \in \mc{H} (k) \}. $$ Here, and throughout $\mc{H}, \J , \K$ denote separable (or finite-dimensional) Hilbert spaces. That is, $\mr{Mult} (\mc{H} (k) , \mc{H} (K) )$ is the space of all operator-valued functions which multiply elements of $\mc{H} (k)$ into $\mc{H} (K)$. Viewing multipliers, $F$, (elements of the multiplier space) as linear maps, $M _F$, from $\mc{H} (k)$ into $\mc{H} (K)$, standard functional analysis arguments show that any multiplier is a bounded linear map, $\mr{Mult} (\mc{H} (k) , \mc{H} (K) ) \subset \scr{L} (\mc{H} (k), \mc{H} (K) )$, and $\mr{Mult} ( \mc{H} (k) , \mc{H} (K) ) $ is closed in the weak operator topology. In the particular case where $\mc{H} (k) = \mc{H} (K)$, $\mr{Mult} (\mc{H} (K) ) := \mr{Mult} (\mc{H} (K) , \mc{H} (K )$ is a unital WOT-closed algebra of bounded linear operators on $\mc{H} (K)$, the \emph{multiplier algebra}.

We work in the setting of vector-valued Drury-Arveson space $H^2 _d \otimes \mc{H}$, where $\mc{H}$ is finite dimensional or separable. This is the vector-valued reproducing kernel Hilbert space $\mc{H} (k)$ of $\mc{H}$-valued functions on the ball $\B ^d := (\C ^d ) _1$ corresponding to the several-variable, operator-valued Szeg\"{o} kernel: $$ k (z,w) := \frac{1}{1-zw^*} I _\mc{H}; \quad \quad z,w \in \B ^d. $$ Here, $zw^* := (w , z) _{\C ^d}$, all inner products are assumed to be conjugate linear in the first argument. The Drury-Arveson space is arguably the canonical several-variable generalization of the classical Hardy space $H^2 _1 = H^2 (\D )$, at least from an operator-theoretic viewpoint \cite{Sha2013,ArvIII,Arv1998curv}.

We will use the notation $H^\infty _d \otimes \scr{L} (\J , \K ) := \mr{Mult} (H^2 _d \otimes \J , H^2 _d \otimes \K)$ (the multiplier
 spaces are the closure of this algebraic tensor product in the weak operator topology). The \emph{Schur class} is the closed unit ball of the multiplier space: $\scr{S} _d (\J , \K ) := [H^\infty _d \otimes \scr{L} (\J , \K )] _1$.  In the single variable ($d=1$) and scalar-valued ($\mc{H}= \C $) setting we recover the classical Hardy space $H^2 (\D )$ and algebra $H^\infty (\D )$ of analytic functions on the disk which embed isometrically into $L^2 , L^\infty $ of the unit circle $\T$, respectively, by taking non-tangential boundary limits \cite{Hoff}.

As in the single-variable case, given any Schur-class function $b \in \scr{S} _d (\J , \K ) $, one can construct a positive kernel function $k^b$ on $\B ^d \times \B ^d$, the de Branges-Rovnyak kernel:
$$ k^b (z,w) := \frac{I - b(z) b (w) ^* }{1-zw^*} \in \scr{L} (\K); \quad \quad z,w \in \B^d, $$ and the corresponding RKHS $\scr{H} (b) := \mc{H} (k^b)$ is called
the de Branges-Rovnyak space associated to $b$.  It is straightforward to show that a bounded analytic (operator-valued) function $b$ on $\B ^d$ belongs to the Schur class if and only if the above formula
defines a positive kernel on $\B ^d$ \cite[Theorem 2.1]{Ball2001-lift}, and we will use this fact frequently in the sequel.

It is easy to check that $k-k^b$ (with $k$ the Szeg\"{o} kernel of vector-valued Drury-Arveson space) is a positive kernel function so that standard vector-valued RKHS theory implies that $\scr{H} (b)$ is contractively contained in $H^2 _d \otimes \K$ \cite{Paulsen-rkhs}. In the single-variable setting, any de Branges-Rovnyak space $\scr{H} (b)$ is co-invariant for the shift \cite{Sarason-dB}. The natural several-variable generalization of the shift is the Arveson \emph{d-shift}, $S : H^2 _d \otimes \C ^d \rightarrow H^2 _d$, a row partial isometry on $H^2 _d$, also denoted by $S = (S_1 , ... , S_d)$ \cite{ArvIII}. The component operators
of $S$ mutually commute ($S$ is a $d$-contraction) and act as multiplication by the independent variables on $H^2 _d$:
$$ (S \mbf{F}) (z) = z_1 F_1 (z) + ... + z_d F_d (z); \quad \quad \mbf{F} = (F_1 , ... , F_d ) \in H^2 _d \otimes \C ^d. $$ In contrast to the classical case, multi-variable de Branges-Rovnyak spaces are generally not co-invariant for the component operators of the Arveson $d-$shift \cite{Ball2008}. As described in the upcoming Subsection \ref{Gleasonsub}, the appropriate several-variable analogue of the adjoint of the restricted backward shift will be a \emph{Gleason solution} for $\scr{H} (b)$, and (extremal) Gleason solutions will play the role of the model operator in our commutative model for CCNC row contractions.
\subsection{Herglotz spaces} \label{HerglotzSection}
It will be convenient to define a second reproducing kernel Hilbert space $\scr{H} ^+ (H_b)$ associated to suitable `square' $b \in \scr{S} _d (\mc{H})$.

In general we will say that a contraction $T \in \scr{L} (\J , \K )$ is \emph{pure} or \emph{purely contractive} if $\| T g \| < \| g\|$ for all $g$ in $\J$, and that $T$ is \emph{strict} or \emph{strictly contractive} if $\| T \| <1$.   A similar argument to \cite[Proposition 2.1, Chapter V]{NF} (combined with a simple argument using automorphisms of the ball $\B ^d$) shows that any $b \in \scr{S} _d (\mc{H} , \J)$ decomposes as $b = b_0 + b_1$ on $\mc{H} = \mc{H} _0 \oplus \mc{H} _1$ where $b_0 (z) := b(z) | _{\mc{H} _0}$ is purely contractive and $b_1$ is a constant isometry on $\B ^d$ from $\mc{H} _1$ onto its range in $\J$.
We say a Schur-class $b \in \scr{S} _d (\J , \K )$ is \emph{purely contractive} or \emph{strictly contractive} if $b(z)$ is a pure or strict contraction, respectively, for all $z \in \B ^d$. As discussed in \cite[Section 1.8]{JM}, $b$ is strictly contractive if and only if $b(0)$ is a strict contraction. (This follows from the Schwarz Lemma for contractive analytic functions on $\B ^d$ combined with an operator-M\"{o}bius transformation argument, see Lemma \ref{ballauto}).

We say that $b \in \scr{S} _d (\mc{H})$ is \emph{non-unital} if $I - b(z)$ is invertible for all $z \in \B ^d$. Any strictly contractive $b \in \scr{S} _d (\mc{H} )$ is certainly non-unital. The \emph{Herglotz-Schur} class, $\scr{S} _d ^+ (\mc{H})$, is the set of all $\scr{L} (\mc{H})$-valued analytic functions on $\B ^d$ such that the \emph{Herglotz kernel},
$$ K ^H (z,w) := \frac{1}{2} \frac{H(z) + H(w) ^*}{1 -zw^*} \in \scr{L} (\mc{H} ) ; \quad \quad z,w \in \B ^d, $$ is an operator-valued positive kernel function. Any Herglotz-Schur function necessarily has positive real part. In particular, if $b \in \scr{S} _d (\mc{H})$ is square and non-unital, and
$$ H _b (z) := (I - b(z) ) ^{-1} (I + b(z)), $$ then,
\ba & & K^b (z,w) :=  K^{H_b} (z,w) \nn \\
 &= & \frac{1}{2} (1-zw^* ) ^{-1} \left(  (I -b(z)) ^{-1} (I + b(z) ) + (I+ b(w) ^* ) (I -b(w)^* ) ^{-1} \right) \nn \\
& = & \frac{1}{2} k (z,w)  (I -b(z) ) ^{-1} \left( (I + b(z) ) (I - b(w) ^* ) + (I - b(z) ) (I + b(w) ^*) \right) (I -b(w) ^* ) ^{-1} \nn \\
& = & \frac{1}{2} k (z,w)  (I -b(z) ) ^{-1} \left( I + b(z) - b(w) ^* - b(z) b(w) ^* \right. \nn \\
& & \left. + I - b(z) + b(w) ^* - b(z) b(w) ^* \right) (I -b(w) ^* ) ^{-1} \nn \\
& = & (I - b(z) ) ^{-1} k^b (z,w) (I - b(w) ^* ) ^{-1}, \nn \ea
 is a positive kernel so that $H_b \in \scr{S} _d ^+ (\mc{H})$.
As described in \cite[Section 1.8]{JM}, the maps
$$ b \mapsto H_b := (I -b) ^{-1} (I+b ); \quad \quad \mbox{and} \quad \quad H \mapsto b_H := (H+I) ^{-1} (H -I ), $$ are compositional inverses
and define bijections between $\scr{S} _d ^+ (\mc{H})$ and non-unital elements of $\scr{S} _d (\mc{H} )$. If $H = H_b \in \scr{S} _d ^+ (\mc{H} )$ we call
$\scr{H} ^+ (H _b) := \mc{H} (K ^b)$, the \emph{Herglotz space} of $b$.

By standard vector-valued RKHS theory, the above relationship between the de Branges-Rovnyak and Herglotz kernels for non-unital $b \in \scr{S} _d (\mc{H} )$ implies
that there is a unitary multiplier $U_b : \scr{H} (b) \rightarrow \scr{H} ^+ (H_b)$.

\begin{lemma} \label{ontoisomult}
The map $U_b : \scr{H} (b) \rightarrow \scr{H} ^+ (H _b) $ defined by multiplication by
\be  U_b (z) := (I-b(z) ) ^{-1}, \ee is an onto isometry. The action of $U_b$ on point evaluation kernels is
$$ U_b k_z ^b = K_z ^b (I -b(z) ^* ). $$
\end{lemma}

It will be useful to consider the natural row partial isometry $V^b : \scr{H} ^+ (H_b) \otimes \C ^d \rightarrow \scr{H} ^+ (H_b)$ on the Herglotz space
of any square $b \in \scr{S} _d (\mc{H} )$ defined by
\be  z^* K_z ^b h:= \bbm \ov{z} _1 K_z ^b h \\ \vdots \\ \ov{z} _d K_z ^b h \ebm \ \stackrel{V^b}{\mapsto} \ (K_z ^b - K_0 ^b)h; \quad \quad h \in \mc{H}. \label{Hergpi} \ee
Verifying that this defines a partial isometry with
$$ \ker{V^b} ^\perp = \bigvee _{z \in \B ^d } z^* K_z ^b \mc{H}, \quad \quad \mbox{and} \quad \quad \ran{V^b} = \bigvee _{z\in \B^d} (K_z ^b - K_0 ^b) \mc{H},$$is a straightforward computation using the Herglotz kernel \cite[Section 2]{JM}.

\subsection{Gleason solutions} \label{Gleasonsub}

As discussed previously, the de Branges-Rovnyak spaces $\scr{H} (b)$ for arbitrary $b \in \scr{S} _d (\J, \K)$
are generally not co-invariant for the component operators $S_j$ of the Arveson $d-$shift \cite{Ball2008}. Instead, the appropriate replacement for the `adjoint of the restricted backward shift' in the several-variable theory is a contractive solution to the \emph{Gleason problem} \cite{Glea1964,Alpay2002,Ball2008,Ball2010,Ball2011dBR}:

\begin{defn} \label{Gledef}
Let $b \in \scr{S} _d (\J, \K)$ be a Schur-class function. A row contraction $X \in \scr{L} ( \scr{H} (b) \otimes \C ^d , \scr{H} (b) )$ solves the \emph{Gleason problem} in $\scr{H} (b)$ if
\be z (X^*f) (z) := z_1 (X_1 ^* f ) (z) + ... + z_d (X_d ^* f ) (z) = f (z) - f(0); \quad \quad \forall f \in \scr{H} (b). \label{Gleasoldef} \ee We say that a \emph{Gleason solution} $X$ is \emph{contractive} if $$ X X ^* \leq I - k_0 ^b (k_0 ^b) ^*,$$ and we say that $X$ is \emph{extremal} if equality holds in the above.
\end{defn}

It is easy to check that for any row contraction $X$ on $\scr{H} (b)$, the Gleason solution condition (\ref{Gleasoldef}) above is equivalent to:
\be (I -Xz^* ) ^{-1} k_0 ^b = k_z ^b ; \quad \quad \forall \ z \in \B ^d, \label{Gleker} \ee and this property will also be used frequently in the sequel.
Given $z \in \B ^d$, and any Hilbert space $\mc{H}$, we will often view $z$ as a strict contraction from $\mc{H} \otimes \C ^d $ into $\mc{H}$:
Define $z^* : \mc{H} \rightarrow \mc{H} \otimes \C ^d$ by $z^* h := (\ov{z} _1 h , ... , \ov{z} _d h ) ^T  \in \mc{H} \otimes \C ^d$ and
for any $\mbf{h} \in \mc{H} \otimes \C ^d$, the adjoint map $z := (z^* ) ^* \in \scr{L} (\mc{H} \otimes \C ^d , \mc{H} )$ obeys:
$$ z \mbf{h} = z \bbm h_1 \\ \vdots \\ h_d \ebm  = z_1 h _1 + ... + z_d h_d, \quad \ \mbox{and} \ \| z \|  ^2 = zz^* := \left( z,z \right)_{\C ^d} <1. $$

\begin{remark} \label{GSrowpi}
    If $X$ is any extremal Gleason solution for $\scr{H} (b)$ where $b \in \scr{S} _d (\J , \K )$ obeys $b(0) = 0$, then
$$XX^* = I - k_0 ^b (k_0 ^b)^*, $$ is a projection so that $X$ is a (row) partial isometry on $\scr{H} (b)$. In general, $P_0 := I -  k_0 ^b k^b (0,0) ^{-1} (k_0 ^b ) ^*$
is the projection onto the subspace of all functions $f \in \scr{H} (b)$ such that $ f(0) = 0$, and
$$ k^b (0, 0) = I - b(0) b(0) ^*, $$ so that an extremal Gleason solution $X$ for $\scr{H} (b)$ is a row partial isometry if and only if $b(0) =0$.
\end{remark}

In the case where $d=1$, the unique solution to equation (\ref{Gleasoldef}) is the adjoint of the restriction of the backward shift $S^*$ to $\scr{H} (b)$, so that adjoints of Gleason solutions are natural analogues of the restricted backward shift in the several-variable setting. Many references define a Gleason solution for $\scr{H} (b)$ as the adjoint of our definition above, we prefer to view it as a row contraction from several copies of a Hilbert space into itself. Contractive solutions $X$ to the Gleason problem in $\scr{H} (b)$ always exist, although they are in general non-unique \cite{Ball2010,JM}. Also note that the component operators of a contractive Gleason solution $X$ for $\scr{H} (b)$ are generally non-commuting \cite{Ball2008}. In fact, the existence of a commuting contractive Gleason solution $X$ for $\scr{H} (b)$ is equivalent to co-invariance of $\scr{H} (b)$ with respect to the component operators of the Arveson $d$-shift, and in this case $X = (S^* | _{\scr{H} (b)} ) ^*$ is a contractive commuting
Gleason solution for $\scr{H} (b)$ \cite[Theorem 3.5]{Ball2008}. This happens, for example, if $b \in \scr{S} _d (\J , \K)$ is an \emph{inner multiplier}, \emph{i.e.} multiplication by
$b$, $M_b : H^2 _d \otimes \J \rightarrow H^2 _d \otimes \K$ is a partial isometry so that $\scr{H} (b) \subset H^2 _d \otimes \K$ is a co-invariant \emph{model subspace} of
vector-valued Drury-Arveson space \cite{ArvIII,BES2005char,GRS2002,McTrent2000} (see Remark \ref{innerGS} below).

Similarly we define contractive Gleason solutions for any $b \in \scr{S} _d (\J  , \K)$:
\begin{defn}
 A linear map $\mbf{b} \in \scr{L} (\J , \scr{H} (b)  \otimes \C ^d )$, $\mbf{b} = \begin{bmatrix} b_1 \\ \vdots \\  b_d \end{bmatrix}$, $b_j \in \scr{L} (\J , \scr{H} (b))$, $1 \leq j \leq d$, is a solution to the Gleason problem for $b \in \scr{S} _d (\J , \K)$ provided that
$$ b(z) - b(0) = z \cdot \mbf{b} (z) := \sum _{j=1} ^d z_j b_j (z).$$ We say that $\mbf{b}$ is a \emph{contractive} Gleason solution for $b$ if
$$ \mbf{b}^* \mbf{b} \leq I - b(0) ^* b(0),$$ and an \emph{extremal} Gleason solution for $b$ if equality holds in the above.
\end{defn}

There is a surjection $\mbf{b} \mapsto X(\mbf{b})$ of contractive Gleason solutions for $b \in \scr{S} _d (\J , \K)$ onto contractive Gleason solutions for $\scr{H} (b)$ given by the formula
 \be X(\mbf{b}) ^* k_z ^b := z^* k_z ^b - \mbf{b} b(z) ^*; \quad \quad z \in \B^d. \label{GSHbb} \ee This surjection is injective if and only if $\bigcap _{z \in \B ^d} \ker{b(z)} = \{ 0 \}$, and it preserves extremal Gleason solutions. This follows as in \cite[Section 5]{JM} (which considers the case $\J = \K$).

If $b \in \scr{S} _d (\mc{H})$ is square and non-unital, then there is a bijection between contractive extensions $D$ of the row partial isometry $V^b$ defined on the Herglotz
space $\scr{H} ^+ (H_b)$ (\emph{i.e.} $D : \scr{H} ^+ (H_b) \otimes \C ^d \rightarrow \scr{H} ^+ (H_b)$, and $D(V^b) ^* V^b = V^b$) and contractive Gleason solutions $\mbf{b} = \mbf{b} [D]$ given by
\be \mbf{b} [D] := (U_b^* \otimes I_d ) D^* K_0 ^b (I -b(0) ), \label{Gleformb} \ee see \cite[Section 5]{JM}. If $X = X (\mbf{b} [D])$, we will simply write $X = X [D]$.

To apply the above parametrizaton of Gleason solutions to arbitrary Schur-class multipliers, it will be useful to define the \emph{square extension} of any $b \in \scr{S} _d (\J ,\K)$: Any such $b$ coincides with an element $b \in \scr{S} _d (\J ' , \K ' )$ where $\J ' \subseteq \K ' $ or $\K ' \subseteq \J '$. Given any $b \in \scr{S} _d (\J , \K)$, and assuming $\J \subseteq \K$ or $\K \subseteq \J$, the \emph{square extension}, $[b]$, of $b$ is
$$  [b] := \left\{  \arraycolsep=2pt\def\arraystretch{1.2} \begin{array}{cc} \bbm b & 0 _{\K \ominus \J , \K } \ebm \in \scr{S} _d (\K ); & \J \subset \K \\
\bbm b \\ 0 _{\J , \J \ominus \K} \ebm \in \scr{S} _d (\J ) ; & \K \subset \J \end{array}. \right. $$

\begin{remark} \label{sqGS}
\bn
\item  If $J \subseteq \K$, then it follows that $\scr{H} (b) = \scr{H} ([b])$. In this case it follows that
$\mbf{b}$ is a contractive Gleason solution for $b$ if and only if there is a contractive Gleason solution $\mbf{[b]}$ for $[b]$ so that $\mbf{b} = \mbf{[b]} | _\J$. Moreover if $\mbf{[b]}$ is extremal, so is $\mbf{b} = \mbf{[b]} | _\J$.
\item If $\K \subseteq \J$ then $\scr{H} ([b]) = \scr{H} (b) \oplus H^2 _d \otimes (\J \ominus \K )$. In this case $\mbf{b}$ is a contractive Gleason solution
for $b$ if and only if $\mbf{[b]} := \mbf{b} \oplus \mbf{0}$ is a contractive Gleason solution for $b$. Here $\mbf{0} : \J \ominus \K \rightarrow H^2 _d \otimes (\J \ominus \K )$ sends every vector to $0 \in H^2 _d \otimes (\J \ominus \K )$. Again this $\mbf{[b]}$ is extremal if and only if $\mbf{b} = \bbm I_{\scr{H} (b)} \otimes I_d, &   0 \ebm  \mbf{[b]}$ is extremal.
\en
\end{remark}

The statement of Theorem \ref{main1} implicitly claims that any extremal Gleason solution $X$ for $\scr{H} (b)$, $b \in \scr{S} _d (\J , \K)$, is a CCNC row contraction.  This is easily established using the extremal condition:

\begin{lemma} \label{GSCCNC}
Let $b \in \scr{S} _d (\J , \K)$ be a multi-variable Schur-class function. Any extremal Gleason solution, $X$, for $\scr{H} (b)$ is a CCNC row contraction.
\end{lemma}

\begin{proof}
Since $X$ is extremal, $$ X X^* = I - k_0 ^b (k_0 ^b) ^*, $$ and since $D_{X^*} ^2 = I -XX^*$,  it follows that $\ran{D_{X^*}}$ is the range of the projection
$$ k_0 ^b k^b (0,0) ^{-1} (k_0 ^b) ^*.$$ Since $X$ is a contractive Gleason solution, it follows that
\ba \bigvee _{z \in \B^d} (I -Xz^* ) ^{-1} \ran{D_{X^*}} & = & \bigvee (I -Xz^* ) ^{-1} k_0 ^b \K  \nn \\
& = & \bigvee _{z\in \B ^d} k_z ^b \K  = \scr{H} (b), \quad \quad \mbox{(by equation (\ref{Gleker}))} \nn \ea so that $X$ is CCNC.
\end{proof}

\subsection{Extremal Gleason solutions} \label{extGSsect}

Since extremal Gleason solutions for $\scr{H} (b)$, $b \in \scr{S} _d (\J , \K )$ will provide a model for arbitrary CCNC row contractions, it is natural to ask whether extremal Gleason solutions always exist.  The results below show that there are large classes of $b \in \scr{S} _d (\J , \K)$, and hence $\scr{H} (b)$, which admit extremal Gleason solutions. (Recall that if $\mbf{b}$ is an extremal Gleason solution for $b \in \scr{S} _d (\J , \K )$, then $X (\mbf{b})$ is an extremal Gleason solution for $\scr{H} (b)$.) We expect that any Schur class $b \in \scr{S} _d (\J , \K)$ admits an extremal Gleason solution, but we do not have a proof of this in full generality.

\begin{lemma} \label{extGSlem}
Given $b \in \scr{S} _d (\mc{H})$, and a row contractive $D \supseteq V^b$, $\mbf{b} [D]$ is extremal if and only if $D$ is a co-isometry.
\end{lemma}
Recall that $\mbf{b} [D]$ is extremal implies that $X[D]$ is extremal. In particular, if $V^b$ is itself a co-isometry, then $\scr{H} (b)$ has a unique contractive Gleason solution, and this solution is extremal. This property (of $V^b$ being co-isometric) is equivalent to $b$ being a quasi-extreme Schur multiplier, as defined in \cite{Jur2014AC,JM}. See \cite{Jur2014AC,JM} and Section \ref{QEsect} for examples of quasi-extreme multipliers, equivalent conditions, and details. 
\begin{proof}
If $D$ is a co-isometric extension of $V^b$, it follows readily from Formula (\ref{Gleformb}), that $\mbf{b} [D]$ is extremal. 

Conversely, extremality of $\mbf{b}[D]$ implies:
$$ (K_0 ^b ) ^* D D ^* K_0 ^b = K^b (0,0) = (K_0 ^b ) ^* K_0 ^b. $$ 
Then, for any $z \in \B ^d$,
\ba (K_z ^b) ^* D D ^* K_0 ^b & = &  (D D^* (K_z ^b - K_0 ^b)) ^* K_0 ^b + (K_0 ^b) ^* D D ^* K_0 ^b  \nn \\
& = & (K_z ^b - K_0 ^b ) ^* K_0 ^b + (K_0 ^b ) ^* K_0 ^b = (K_z ^b) ^* K_0 ^b. \nn \ea This proves that $D D^* K_0 ^b = K_0 ^b$ so that 
\ba DD^* K_z ^b &=& DD ^* (K_z ^b - K_0 ^b) + DD^* K_0 ^b \nn \\
& = & K_z ^b -K_0 ^b + K_0 ^b = K_z ^b, \nn \ea and $DD^* = I$.
\end{proof}

\begin{remark}
By the above lemma, $b \in \scr{S} _d (\mc{H} )$ (and hence $\scr{H} (b)$) has extremal Gleason solutions if and only if $\dim{\ker{V^b}} \geq \dim{\ran{V^b} ^\perp}$ (so that $V^b$ has co-isometric extensions).  The proofs of \cite[Proposition 4.12, Proposition 4.15]{JM} can be adapted to show this is the case for large classes of $b \in \scr{S} _d (\mc{H})$:
\end{remark}
\begin{prop} \label{Extprop1}
If $b \in \scr{S} _d (\mc{H})$ is such that $\dim{\scr{H} (b)} = \infty$ and $\dim{\mc{H} } < \infty$, then $\dim{\ker{V^b}} = \infty$ and $\dim{\ran{V^b} ^\perp} \leq \dim{\mc{H} }$.
\end{prop}
\begin{proof}
Set $V:= V^b$ and calculate,
$$ (V_k ^* V_j ^* - V_j ^* V_k ^*) K_z ^b = (\ov{z} _j V_k ^* - \ov{z} _k V_j ^* + V_k ^* V_j ^* - V_j ^* V_k ^* ) K_0 ^b; \quad \quad z \in \B ^d. $$
Since $\dim{\mc{H}} < \infty$ and $K_0 ^b \in \scr{L} (\mc{H} , \scr{H} ^+ (H_b))$, the above commutators have finite rank. It is also easy to check that $\ran{V} ^\perp$ consists of $\mc{H}$-valued constant functions so that $\dim{\ran{V}^\perp} \leq \dim{\mc{H}} < \infty$, by assumption. Since we assume that $\scr{H} (b)$, and hence $\scr{H} ^+ (H_b)$ are infinite dimensional, the $V_k$ cannot all have finite rank by the definition of $V=V^b$. It follows that the image of $V$ under the quotient $*$-homomorphism, $\pi$, onto the Calkin algebra (quotient by the compact operators) is a non-zero row co-isometry with commuting component operators. If $\ker{V^b}$ is finite dimensional, then $\pi (V)$ is a row isometry (in fact Cuntz unitary) with commuting component operators, and this is readily checked to be impossible.
\end{proof}
Similarly, 
\begin{prop} \label{Extprop2}
Let $\mbf{b} = (b_1 , ... , b_d) ^T$ be any contractive Gleason solution for $b \in \scr{S} _d (\mc{H})$, $d\geq 2$. Then $\dim{\ker{V^b}} \geq \frac{d (d-1)}{2} \cdot \dim{\scr{H} (b) \ominus \bigvee _{j=1} ^d b_j \mc{H}}$.
\end{prop}

\begin{proof}
This is easily established using the proof of \cite[Proposition 4.4]{JM}:
Let $X = X (\mbf{b})$ be a contractive Gleason solution for $\scr{H} (b)$, where $\mbf{b}$ is any contractive Gleason solution for $b$ (see Equation (\ref{GSHbb})). Then $\mbf{b} = (b_1, ...,b_d) ^T \in  \scr{H} (b) \otimes \C ^d$, $d \geq 2$. Choose any $f \in \scr{H} (b)$ orthogonal to the linear span of the $b_j, \ 1 \leq j \leq d$. It follows that
\ba \ip{ k_z ^b}{ (X_j f) } & = & \ip{\ov{z} _j k_z ^b - b_j b(z) ^*  }{f} \nn \\
& = & z_j f  (z)_\mc{H}. \nn \ea This proves that $S_j f  \in \scr{H} (b)$ for any $1 \leq j \leq d$. This in turn implies that $h_j := (I-b) ^{-1} S_j f \in \scr{H} ^+ (H_b)$ and if we define $\mbf{h} := -h_j \otimes e_i + h_i \otimes e_j \in \scr{H} ^+ (H_b) \otimes \C ^d$, where $\{e_k \}$ is the standard basis of $\C ^d$ and $i\neq j$, we then have that
\ba \ip{z^* K_z ^b }{\mbf{h}} & = & - z_i h_j (z) +z_j h_i (z) \nn  \\
& = & (-z_i z_j + z_j z_i ) (I-b(z) ) ^{-1} f(z)  = 0. \nn \ea 
It follows that for any distinct pair $(i,j); \ 1 \leq i,j \leq d$, and any $f \perp \bigvee _{j=1} ^d \ran{b_j} $, we can construct an element $\mbf{h} ^{(i,j)} \in \ker{V^b}$.
\end{proof}

As discussed in Remark \ref{sqGS}, if the square extension $[b]$, of any $b \in \scr{S} _d (\J, \K)$ has extremal Gleason solutions, so does $b$. The above two propositions now imply that there are large classes of $b \in \scr{S} _d (\J , \K)$ (and hence $\scr{H} (b)$) which admit extremal Gleason solutions.  
\begin{remark} \label{innerGS}
    If $b \in \scr{S} _d (\J , \K)$ is an inner multiplier ($M_b : H^2 _d \otimes \J \rightarrow H^2 _d \otimes \K$ is a partial isometry), then 
$\scr{H} (b) = \ran{M _b} ^\perp$ is co-invariant for the component operators of the (vector-valued) Arveson $d$-shift, $S \otimes I _{\mc{K}}$, and 
$$ X_j ^* := S^* _j \otimes I _{\mc{K}} | _{\scr{H} (b)}, $$ defines an extremal Gleason solution (with commuting component operators). Indeed, 
\ba I- X X^* &=& P_{\ran{M_b}} ^\perp (I - S S^*)  P_{\ran{M_b}} ^\perp  \nn \\
& = & (I - M_b M_b ^*) k_0  k_0 ^* (I -M_b M_b ^* ) \nn \\ 
& = & k_0 ^b (k_0 ^b) ^*. \nn \ea 
\end{remark}

\subsection{de Branges-Rovnyak model via transfer-function theory}
\label{BBFsect}

In this section we will show that Theorem \ref{main1} (\emph{i.e.} Theorem \ref{main1intro} from the Introduction) is a straightforward consequence of the multi-variable transfer function realization theory of \cite{Ball2007trans}.

As part of \cite[Theorem 1.3]{Ball2007trans}, J.A. Ball, V. Bolotnikov and Q. Fang proved that a $\scr{L} (\J , \K )$-valued function $b$ on $\B ^d$ belongs to the multi-variable, operator-valued Schur class $\scr{S} _d (\J , \K)$, if and only if there is an auxiliary Hilbert space, $\mc{H}$, and a contractive \emph{colligation}, $U$:
$$ U := \bbm A & B \\ C & D \ebm =: \bbm A_1 & B_1 \\ \vdots & \vdots \\ A_d & B_d \\ C & D \ebm : \bbm \mc{H} \\ \J \ebm \rightarrow \bbm \mc{H} \otimes \C ^d \\ \K \ebm, $$ so that $b$ has the \emph{transfer function realization}:
$$ b(z) = D + C (I-zA) ^{-1} zB. $$
Moreover in \cite[Section 3]{Ball2007trans}, the authors construct a canonical \emph{functional model realization} of $b \in \scr{S} _d (\J , \K )$, providing a natural extension of the classical transfer function theory of de Branges and Rovnyak \cite{dBss,dBmodel}: If $A := X^*$ is the adjoint of any contractive Gleason solution for $\scr{H} (b)$, $C:= (k_0 ^b) ^*$, and $D:= b(0)$, then there exists a (essentially unique \cite[Corollary 2.9]{Ball2007trans}) contractive Gleason solution $\mbf{b}=:B$ for $b$ so that
$$ U_{dBR} := \bbm A & B \\ C & D \ebm = \bbm X^* & \mbf{b} \\ (k_0 ^b ) ^* & b(0) \ebm : \bbm \scr{H} (b) \\ \J \ebm \rightarrow \bbm \scr{H} (b) \otimes \C ^d \\ \K \ebm, $$ is a contractive colligation providing a transfer-function realization for $b$. In fact, \cite[Theorem 7.13]{JMfree} proves that $A^* = X = X(\mbf{b})$ is the contractive Gleason solution for $\scr{H} (b)$ corresponding to the contractive Gleason solution $B = \mbf{b}$ for $b$ as in Subsection \ref{Gleasonsub}.

\begin{remark} \label{isomextr}
In this paper we focus on extremal Gleason solutions $X = X(\mbf{b})$ as these will provide a model for arbitrary CCNC row contractions. It is easy to check that $\mbf{b}$ (and hence $X(\mbf{b})$) are extremal if and only if the corresponding canonical functional model colligation
$$ U_{dBR} = \bbm X (\mbf{b}) ^* & \mbf{b} \\ (k_0 ^b) ^* & b(0) \ebm,  $$ is an isometry \cite[Remark 4.18]{JM}.
\end{remark}

As described in \cite[Section 3]{Ball2007trans}, any canonical functional model realization for $b \in \scr{S} _d (\J , \K )$ is \emph{weakly co-isometric} and \emph{observable}. A colligation $U$, as above is weakly co-isometric if it is co-isometric on a certain subspace \cite{Ball2007trans}. By \cite[Proposition 1.5]{Ball2007trans}, $U$ is weakly co-isometric if and only if:
$$ k^b (z,w) = C^* (I-A^*z) ^{-1} (I-w^* A) ^{-1} C.$$ This characterization and Formula (\ref{Gleker}) show that any canonical functional model colligation $U_{dBR}$ is always weakly-coisometric.

A contractive colligation $U$ is said to be \emph{observable} if
$$ C (I -Az) ^{-1} x =0 \quad \quad \forall \  z \in \B ^d $$ if and only if $x \equiv 0$ \cite[Section 2]{Ball2007trans}. Equivalently,
\be \bigvee _{z \in \B ^d } (I -z^* A^* ) ^{-1} C^* \K = \mc{H}. \label{observable} \ee In the case where $U= U_{dBR}$ is a functional model colligation, $A = X^*$ is a contractive Gleason solution and $C = (k_0 ^b) ^*$, so that, again by Formula (\ref{Gleker}),
$$ \bigvee _{z \in \B ^d} (I -z^* A^* ) ^{-1} C^* \K = \bigvee _{z \in \B ^d} k_z ^b \K = \scr{H} (b).$$ This shows that $U_{dBR}$ is observable.

If $T$ is any CCNC row contraction on $\mc{H}$, consider the Julia operator:
$$ U_T := \bbm T^* & D_T \\ D_{T^*} & -T \ebm : \bbm \mc{H} \\ \ran{D_T} \ebm \rightarrow \bbm \mc{H} \otimes \C ^d \\ \ran{D_{T^*}} \ebm. $$ It is easy to
verify that $U_T$ is an onto isometry, so that in particular, it is a weakly co-isometric colligation. Since $T$ is CCNC,
$$ \mc{H} = \bigvee _{z\in \B ^d} (I -Tz^* ) ^{-1} \ran{D_{T^*}}, $$ proving that $U_T$ is also observable (by Formula \ref{observable}).

Let $b_T \in \scr{S} _d ( \ran{D_T} , \ran{D_{T^*}} )$ be the transfer-function corresponding to $U_T$. Since $U_T$ is observable and weakly co-isometric, \cite[Theorem 3.4]{Ball2007trans} implies that $U_T$ is unitarily equivalent to a canonical functional model colligation, $U_{dBR}$, for $b_T$.
In particular, $T$ is unitarily equivalent to $X$, a contractive Gleason solution for $\scr{H} (b_T)$. Finally, since $U_T$, and hence $U_{dBR}$ is an isometry, it follows as in Remark \ref{isomextr} that $X$ is necessarily an extremal Gleason solution, and this completes a proof of Theorem \ref{main1}.

\section{Completely non-coisometric row contractions} \label{CNCsection}

In this section we characterize completely non-coisometric (CNC) row contractions and motivate the definition of CCNC row contractions (contractions obeying the commutative CNC condition).  As in the classical setting a row contraction $T : \mc{H} \otimes \C ^d \rightarrow \mc{H}$ is CNC if there is no non-trivial co-invariant subspace $\mc{H} ' \subset \mc{H}$ (co-invariant for each component operator $T_k$ of $T$, $1\leq k \leq d$)
so that $T^* | _{\mc{H} '}$ is an isometry of $\mc{H} ' $ into $\mc{H} ' \otimes \C ^d$. As observed in \cite{Ball2011dBR}, a row contraction $T: \mc{H} \otimes \C ^d \rightarrow \mc{H}$ is CNC if and only if $ \bigvee _{\alpha \in \F ^d} T^\alpha \ran{D_{T^*}} = \mc{H}$, where, as before, the defect operator of any contraction $T : \J \rightarrow \K$ is $D_T := \sqrt{ I _\J - T^* T},$ and $\F ^d$ denotes the free monoid on $d$ generators. For convenience, we will provide a proof based on methods ultimately due to Kre\u{\i}n.

In the above, recall that the free semigroup (or monoid), $\F^d$, on $d \in \N$ letters, is the multiplicative unital semigroup of all finite products or \emph{words} in
the $d$ letters $\{1, ... , d \}$. That is, given words $\alpha := i_1 ... i_n$, $\beta := j_1 ... j_m$, $i_k, j_l \in \{1 , ... , d \}; \ 1\leq k \leq n, \ 1 \leq l \leq m$, their product $\alpha \beta $ is defined by concatenation:
$$ \alpha \beta = i_1 ... i_nj_1 ... j_m, $$ and the unit is the empty word, $\emptyset$, containing no letters. Given $\alpha = i_1 \cdots i_n$, we use the standard notation $|\alpha | = n$ for the length of the word $\alpha$. Let $\N ^d$ be the unital additive semigroup or monoid of $d$-tuples of non-negative integers. By the universality property of the free unital semigroup $\F ^d$, there is a unital semigroup epimorphism $\la : (\F ^d , \cdot ) \rightarrow (\N ^d , +)$, the \emph{letter counting map} which sends a given word $\alpha = i_1 \cdots i_n \in \F ^d $ to $\mbf{n} = (n_1, ... ,n_d) \in  \N ^d$ where $n_k$ is the number of times the letter $k$ appears in the word $\alpha$. Given any $\alpha = i _1 \cdots i _k \in \F ^d$, and any row contraction $T = (T_1, ... ,T_d)$, we use the standard notation
$$ T^\alpha := T_{\alpha _1} \cdots T_{\alpha _k}, $$ and for any $\n \in \N ^d$ we define the symmetrized monomial:
$$ T^\n := \sum _{\alpha ; \ \la (\alpha ) = \n} T^\alpha. $$

Initially, we focus on the case of a (row) partial isometry, $V : \mc{H} \otimes \C ^d \rightarrow \mc{H}$, and we define the \emph{restricted range spaces},
$$ \ra{V-z} := \ran{(V-z) V^*V}; \quad \quad z \in \C ^d, $$ as the range of $V -z$ restricted to the initial space of $V$. The orthogonal complement,
$\ra{V-z} ^\perp$, will be called a $z$-deficiency space or the \emph{$z-$defect space}. More generally, as in \cite{KVV}, consider the \emph{non-commutative} (NC) \emph{open unit ball}:
$$ \B ^d _{\N}:= \coprod _{n =1} ^\infty \B ^d _n; \quad \quad \B ^d _n := \left( \C ^{n \times n} \otimes \C ^d \right) _1.  $$
Elements $Z \in \B ^d _n$ are viewed as strict (row) contractions from $\C ^n \otimes \C ^d$ into $\C ^n$:
$$ Z =: (Z _1 , ... , Z _d ) ; \quad \quad Z_k \in \C ^{n \times n}. $$ In particular, $\B ^d _1 \simeq \B ^d = (\C ^d ) _1$ can be identified with the open unit ball of $\C ^d$.

\begin{defn}
    For any $Z \in \B ^d _n$, $n \in \N$, let
\ba \ra{V-Z} & := & \ran{ \left[ (V \otimes I_n ) - (I _\mc{H} \otimes Z) \right] (V ^* V \otimes I_n ) } \nn \\
& = & \left[ I _\mc{H} \otimes I_n - (I _\mc{H} \otimes Z ) (V^* \otimes I_n ) \right] \ran{V} \otimes \C ^n \nn \\
& =: & (I - ZV^* ) \ran{V} \otimes \C ^n, \nn \ea
\be  \mbox{where} \quad \quad ZV^* := V_1 ^* \otimes Z_1 + ... + V_d ^* \otimes Z_d \in \left ( \scr{L} ( \mc{H} \otimes \C ^n ) \right) _1. \label{ZXstar} \ee
\end{defn}

The main results of this section will be:

\begin{thm} \label{Kthm}
    Let $V : \mc{H} \otimes \C ^d \rightarrow \mc{H}$ be a row partial isometry. The subspace
$$ \mc{H} ' :=  \{ h \in \mc{H} | \ h \otimes \C ^n \subseteq \mc{H} ' _n \quad \forall \ n \in \N \} ; \quad \quad \mc{H} ' _n := \bigcap _{Z \in \B ^d _n } \ra{V -Z}, $$
is the largest co-invariant subspace for $V$ on which $V^*$ acts isometrically.

In particular, $V$ is CNC if and only if
\ba \mc{H} & = & (\mc{H} ' ) ^\perp = \bigvee _{\la \in \C ^n; \ Z \in \B ^d _n} \la (I -VZ^* ) ^{-1} \ran{V} ^\perp \otimes \C ^n  \nn \\
& = & \bigvee _{\alpha \in \F ^d} V^\alpha \ran{V} ^\perp. \nn \ea
\end{thm}
In the above, as in Subsection \ref{Gleasonsub}, we view any $\la \in \C^n$ as a row operator, $\la : \mc{H} \otimes \C ^n \rightarrow \mc{H}$. This theorem is a generalization of a characterization of CNU partial isometries due to Kre\u{\i}n \cite[Chapter 1, Theorem 2.1]{Krein}. Kre\u{\i}n's result is proven in the setting of unbounded symmetric operators. This can be restated in terms of partial isometries using the Cayley transform, a fractional linear transformation that implements  a bijection between partial
isometries and symmetric linear transformations \cite{Glazman}.

In the special case where $V$ is commutative, \emph{i.e.} a $d$-contraction, this yields:

\begin{thm} \label{CKthm}
    Let $V : \mc{H} \otimes \C^d \rightarrow \mc{H}$ be a $d$-partial isometry. The subspace
$$\mc{H} ' := \bigcap _{z \in \B ^d} \ra{V-z}, $$ is the largest co-invariant subspace for $V$ on which $V^*$ acts isometrically.

In particular, $V$ is CNC if and only if
\ba \mc{H} & = & (\mc{H} ' ) ^\perp = \bigvee _{z \in \B ^d = \B ^d _1} (I-Vz^* ) ^{-1} \ran{V} ^\perp \nn \\
& = & \bigvee _{\n \in \N ^d} V^\n \ran{V} ^\perp. \nn \ea
\end{thm}

A key corollary of these results is:

\begin{cor} \label{CCNCcor}
    A row contraction $T : \mc{H} \otimes \C ^d \rightarrow \mc{H}$ is CNC if and only if
$$ \mc{H} =  \bigvee _{\alpha \in \F ^d} T^\alpha \ran{D_{T^*}}. $$ A $d$-contraction $T : \mc{H} \otimes \C ^d \rightarrow \mc{H}$ is CNC if and only if
$$ \mc{H} = \bigvee _{\n \in \N ^d } T^\n \ran{D_{T^*}} = \bigvee _{z \in \B^d} (I-Tz^* ) ^{-1} \ran{D_{T^*}}. $$
In particular, if a (not necessarily commutative) row contraction $T$ on $\mc{H}$ obeys
$$\mc{H} = \bigvee _{z \in \B^d} (I -Tz^* ) ^{-1} \ran{D_{T^*}}, \quad \quad \mbox{then $T$ is CNC.} $$
\end{cor}

Motivated by the above corollary, we define:
\begin{defn}
    A row contraction $T : \mc{H} \otimes \C ^d \rightarrow \mc{H}$ obeys the \emph{commutative CNC condition}, and we write: $T$ is CCNC,
if $$\mc{H} = \bigvee _{z \in \B^d} (I -Tz^* ) ^{-1} \ran{D_{T^*}}.$$ Corollary \ref{CCNCcor} implies that any CCNC row contraction is CNC, and that any CNC $d-$contraction is CCNC.
\end{defn}

Note, in particular, that if $T$ is commutative, \emph{i.e.} $T$ is a $d$-contraction,
that $$ \bigvee _{\alpha \in \F ^d} T^\alpha \ran{D_{T^*}} = \bigvee _{\n \in \N ^d} T^\n \ran{D_{T^*}},$$ and this second set can be re-expressed as follows:
\begin{lemma}{ \cite[Lemma 2.2]{Jur2014AC}} \label{JurLemma}
    Given any row contraction $T$ on $\mc{H}$,
$$ \bigvee _{z \in \B ^d}  (I - Tz^* ) ^{-1} = \bigvee _{\n \in \N ^d} T^\n. $$
\end{lemma}

The next example below shows that not every CNC row contraction is CCNC:
\begin{eg}{ (The Arveson $d$-shift and the left free shift)}
The Arveson $d$-shift, $S$, on the Drury-Arveson space, $H^2 _d$, is a commutative row partial isometry. In particular $D_{S^*} = k_0 k_0 ^*$ is a projection
so that $$ \bigvee _{z \in \B ^d} (I-Sz^* ) ^{-1} \ran{D_{S^*}} = \bigvee _{z \in \B ^d} k_z = H^2 _d, $$ and $S$ is CCNC. Observe that $S$ is an extremal Gleason solution for the de Branges-Rovnyak space $\scr{H} (b) = H^2 _d$ with $b \equiv 0 \in \scr{S} _d = \scr{S} _d (\C )$.

For an example of a CNC row contraction which is not CCNC, consider the left free shift $L$ on the full Fock space $F^2 _d$. Recall here that the full Fock space over $\C ^d$, $F^2 _d$, is the direct sum of all tensor powers of $\C ^d$:
\ba F^2 _d & := & \C \oplus \C ^d \oplus  \left( \C ^d \otimes \C ^d \right) \oplus \left( \C ^d \otimes \C ^d \otimes \C ^d \right) \oplus \cdots \nn \\
& =& \bigoplus _{k=0} ^\infty \left( \C ^d \right) ^{k \cdot \otimes }. \nn \ea Fix an orthonormal basis $e_1, ... , e_d$ of $\C ^d$. The left creation
operators $L_1, ..., L_d$ are the operators which act as tensoring on the left by these basis vectors:
$$ L_k f := e_k \otimes f; \quad \quad f \in F^2 _d. $$  The left  free shift is the row operator $L := (L_1 , ... , L_d): F^2 _d \otimes \C ^d \rightarrow F^2 _d$. The left free shift, is, in fact, a row isometry: $L ^* L = I_{F^2} \otimes I_d$.  The orthogonal complement of the range of $L$ is the vacuum vector $1$ which spans the the subspace $\C =: (\C^d) ^{0\cdot \otimes} \subset F^2 _d$. A canonical orthonormal basis for $F^2 _d$ is then $\{ e_\alpha \} _{\alpha \in \F ^d}$ where $e_\alpha = L^\alpha 1$. It follows that $L$ is CNC:
$$ \bigvee _{\alpha \in F^2 _d} L^\alpha \ran{D_{L^*}} = \bigvee L^\alpha 1 = F^2 _d. $$ However, $L$ is not CCNC, since
$$ \bigvee _{\n \in \N ^d } L^\n \ran{D_{L^*}} = \bigvee L^\n 1, $$ is equal to symmetric or Bosonic Fock space, a proper subspace of $F^2 _d$ (which is in fact canonically isomorphic to $H^2 _d$, see \emph{e.g.} \cite[Section 4.5]{Sha2013}).
\end{eg}

It will be convenient to first establish several preliminary facts before proving Theorem \ref{Kthm} and its corollaries. Given a partial isometry $V : \mc{H} _1 \rightarrow \mc{H} _2$, and a contraction $T : \K _1 \rightarrow \K _2$ where $\mc{H} _k \subset \K _k$, we say
that $T$ extends $V$ and write $V \subseteq T$ if $T V^*V = V$.

\begin{lemma}{ (\cite[Lemma 2.3]{JM})} \label{contractext}
Let $V : \mc{H} _1 \rightarrow \mc{H} _2$ be a partial isometry and assume that $\mc{H} _k \subseteq \J _k$ for Hilbert spaces $\mc{H} _k , \J _k$, $k=1,2$. For any contraction $T : \J _1 \rightarrow \J _2$ the following are equivalent:
\bi
    \item[(i)] $T$ is a contractive extension of $V$, $V \subseteq T$.
    \item[(ii)] $T^*$ is a contractive extension of $V^*$, $V^* \subseteq T^*$.
    \item[(iii)] There is a contraction $C : \J _1 \rightarrow \J _2$ such that $\ker{C} ^\perp \subseteq \J _1 \ominus \ker{V} ^\perp, \ \ran{C} \subseteq \J _2 \ominus \ran{V} $ and $T = V-C$.
\ei
\end{lemma}

\begin{lemma} \label{TextVlemma}
    Let $V$ be a row partial isometry on $\mc{H} \otimes \C ^d$ and let $T \supseteq V$ be a row contractive extension to $\K \otimes \C ^d$, $\K \supseteq \mc{H}$. Then for any $Z \in \B ^d _n$,
    $$ (I -TZ^* ) ^{-1} : \K \otimes \C ^n \ominus \left( \ran{V} \otimes \C ^n \right) \rightarrow \K \otimes \C ^n \ominus \ra{V-Z}, \quad \quad \mbox{is an isomorphism.} $$
\end{lemma}

Setting $T =V$, this shows $\ra{V-Z} ^\perp = (I -VZ^* ) ^{-1} \ran{V} ^\perp$. This holds, in particular, for $Z = z \in \B ^d _1 = \B ^d$ and $T \supseteq V$ acting on $\mc{H} \otimes \C ^d$ (\emph{e.g.} $T=V$):
$$ (I-Tz^*) ^{-1} : \ran{V} ^\perp \rightarrow \ra{V-z} ^\perp \quad \quad \mbox{is an isomorphism.}$$

\begin{proof}
    Since $Z$ is strictly contractive, $TZ^*$, as defined above in equation (\ref{ZXstar}), is a strict contraction so that $(I -TZ^* ) ^{-1}$ is well-defined.
Assume that $f \in \K \otimes \C ^n \ominus \ra{V-Z}$, and suppose $g \in \ran{V} \otimes \C ^n$ so that $g = (V \otimes I_n )G $ for some
$G \in (\mc{H} \otimes \C ^n) \otimes \C ^d \ominus \ker{V \otimes I_n }.$ Then,
\ba \ip{ (I -TZ^*) f }{ g} & = & \ip{f}{(I -ZT^* ) (V\otimes I_n) G} \nn \\
& = & \ip{f}{ (I -ZV^* ) (V\otimes I_n )G} = 0. \nn \ea This proves that
$$ (I -TZ^* ) \left( (\K \otimes \C ^n) \ominus \ra{V-Z} \right) \subseteq (\K \otimes \C ^n) \ominus (\ran{V} \otimes \C ^n), $$ so that
$$ (\K \otimes \C ^n) \ominus \ra{V-Z} \subseteq (I -TZ^* ) ^{-1} \left( \K \otimes \C ^n \ominus\ran{V} \otimes \C ^n \right). $$
For the reverse inclusion suppose that $f \in \K \otimes \C ^n \ominus \ran{V \otimes I_n}$. Also
assume that $g \in \ra{V-Z}$, $g = (I - ZV^* ) (V \otimes I_n) G$. Then,
\begin{align*} \ip{ (I -TZ^* ) ^{-1} f}{g} & =  \ip{f}{ (I -ZT^* ) ^{-1} (I -ZV^*) VG }  && (VG := V\otimes I_n G) \\
& =  \ip{f}{ (I -ZT^* ) ^{-1} (I -ZT^*) VG } && \\
& =  \ip{f}{V G} = 0. && \end{align*}
\end{proof}

\begin{lemma} \label{Lcontain}
    Suppose that $\scr{L} \subseteq \mc{H}$ is co-invariant for $V$, and that $V^* | _{\scr{L} }$ is an isometry. Then for any $Z \in \B ^d _n$,
$\scr{L} \otimes \C ^n \subseteq \ra{V-Z}. $
\end{lemma}

\begin{proof}
\begin{align*} \ra{V -Z} & \supseteq  (I -ZV^* ) V \otimes I_n (\scr{L} \otimes \C ^n \otimes \C ^d )  \\
& = (I -ZV^*) \scr{L} \otimes \C ^n; && \mbox{$V$ is co-isometric, hence onto $\scr{L}$}  \\
& =  \scr{L} \otimes \C ^n; && (I -ZV^* ) \ \mbox{is invertible, hence onto.} \end{align*}

\end{proof}

\begin{defn}
    For any $n \in \N$ define $\mc{H} ' _n \subset \mc{H} \otimes \C ^n$ by
$$ \mc{H} ' _n := \bigcap _{Z \in \B ^d _n} \ra{V-Z}. $$
Also define $\mc{H} ' \subseteq \mc{H}$ by
 \ba \mc{H} ' & := & \{ h \in \mc{H} | \ h \otimes \C ^n \subseteq \mc{H} ' _n, \ \forall n \in \N \} \nn \\
 & = & \{ h \in \mc{H} | \ \la ^* h \in \mc{H} ' _n \ \forall \la \in \C ^n \ \forall n \in \N \}. \nn \ea
\end{defn}

In the above, as before, given any $\la \in \C ^n$, we view $\la : \mc{H} \otimes \C ^n \rightarrow \mc{H}$ as a row operator.

\begin{lemma} \label{ortho}
    $\mc{H} ' $ is a closed subspace, and
\ba (\mc{H}') ^\perp & = & \bigvee _{\substack{\la , \tau \in \C ^n \\ Z \in \B ^d _n ; \ n \in \N}} \la (I -VZ^* ) ^{-1}  \tau ^* \ran{V} ^\perp \nn \\
& = & \bigvee _{\substack{\la  \in \C ^n \\ Z \in \B ^d _n ; \ n \in \N}} \la (I -VZ^* ) ^{-1}  \ran{V} ^\perp \otimes \C ^n. \nn \ea
\end{lemma}

\begin{proof}
    Let $h_k$ be Cauchy in $\mc{H} '$ with limit $h \in \mc{H}$. Then for any fixed $n \in \N$ and $\tau \in \C ^n$, $\tau ^* h_k $ is Cauchy in $\mc{H}' \otimes \C ^n$, and hence the limit $\tau ^*h $ belongs to $\mc{H} ' _n$, since $\mc{H} ' _n$ is clearly closed. Since this holds for any $n$, we obtain that $h \in \mc{H} '$.

For the proof of the second statement suppose that $h \in \mc{H}$ is orthogonal  to any vector of the form
$$ \la (I -VZ^* ) ^{-1} \tau ^* \ran{V} ^\perp: \quad \quad \la , \tau \in \C ^n, \ Z \in \B ^d _n. $$
This happens if and only if:
\begin{align*} h & \perp \la (I -VZ^* ) ^{-1} \ran{V} ^\perp \otimes \C ^n ; && \la \in \C ^n, \ Z \in \B ^d _n, \ n\in \N \\
\Leftrightarrow \ h \otimes \C ^n & \perp   \underbrace{(I-VZ^*) ^{-1} \ran{V} ^\perp \otimes \C ^n}_{=\ra{V-Z} ^\perp \ \mbox{by Lemma \ref{TextVlemma}}}; && Z \in \B ^d _n, \ n \in \N. \end{align*}  This proves that $h$ has this property if and only if $h \otimes \C ^n \subseteq \ra{V-Z}$ for any $Z \in \B ^d _n$, \emph {i.e.} if and only if
$h \otimes \C ^n \subseteq \mc{H} ' _n$ for any $n \in \N$ and therefore if and only if $h \in \mc{H} '$.
\end{proof}

\begin{lemma}
    If $\scr{L} \subseteq \mc{H}$ is co-invariant for the row partial isometry $V$ and $V^*$ is isometric on $\scr{L}$ then $\scr{L} \subseteq \mc{H} '$.
\end{lemma}

\begin{proof}
    By Lemma \ref{Lcontain}, $\scr{L} \otimes \C ^n \subseteq \mc{H} ' _n$ for any $n \in \N$. In particular, given any $l \in \scr{L}$, $\tau ^*l$ belongs to $\mc{H}' _n $ for any $\tau \in \C ^n$ and any $n \in \N$.
Therefore $l \in \mc{H} '$ by definition and $\scr{L} \subseteq \mc{H} '$.
\end{proof}

The following technical fact provides a useful description of $(\mc{H} ' ) ^\perp$.

\begin{lemma} \label{NCfact}
$$     \bigvee _{\substack{\la , \tau \in \C ^n \\ Z \in \B ^d _n; \ n \in \N }}  \la (I -ZV^* ) ^{-1} \tau ^* = \bigvee _{\alpha \in \F ^d } (V^* ) ^\alpha. $$
\end{lemma}

\begin{proof}
    First note that anything in the the left hand side (LHS) of the above equation is a linear combination of products of the $V_1^*, ... , V_d ^*$, and so it follows that the left hand side is contained in the right hand side (RHS).

We will prove the converse inductively on the length, $N = | \alpha |$ of a word $\alpha \in \F ^d$.  First, taking $n=1$, we have by Lemma \ref{JurLemma} that $(V^* ) ^\n $ belongs to the LHS for all $\n \in \N ^d$.
In particular $V_k ^* \in LHS$ for all $1 \leq k \leq d$ so that the inductive hypothesis holds for $N=1$. Assume that the inductive hypothesis holds for all $\alpha \in \F ^d$ of length less than or equal to $K \in \N$.
That is, $| \alpha | \leq K$ implies that $(V^*)^\alpha$ belongs to the LHS. To complete the induction step we need to prove that given any $\beta \in \F ^d$ of length $K+1$ that $(V^*)^\beta$ belongs to the LHS.

Any such $\beta$ can be written $\beta = j \alpha$ where $j \in \{ 1 , ... , d \}$ and $|\alpha | = K$. By the hypothesis $(V^*)^\alpha $ is the norm-limit of finite linear combinations of terms of the form
$$ \la (I -ZV^* ) ^{-1} \tau ^*, $$ where $\la, \tau \in \C ^n$ and $Z \in \B ^d _n$ for some $n \in \N$. It therefore suffices to prove that for any such term, and any fixed $1\leq j \leq d$, we can find $W \in \B ^d _m$ and $\La  , \Ga \in \C ^m$
so that $$ \La (I -WV^*) ^{-1} \Ga ^* = V_j ^* \la (I -ZV^* ) ^{-1} \tau ^*. $$ For simplicity fix $j =1$. The other cases will follow from an analogous argument.

Choose $W = (W_1, ..., W_d ) \in \B ^d _{2n}$ as follows:

$$ W_1  := \bbm Z_1 & 0_n \\ r I_n & 0 _n \ebm; \quad \quad W_k := \bbm Z_k &  0 \\ 0 & 0 \ebm; \quad 2 \leq k \leq d, $$ where if $\| Z \| ^2 = s < 1$, choose $0<r<1$ small enough so that $1 > s (1+r^2)$.

Then,
$$ I - W W^* = \bbm I - Z Z ^* &  -r Z_1 \\ -r Z_1 ^* & I \ebm, $$ and by Schur complement theory \cite[Appendix A.5.5]{Boyd}, this is strictly positive if and only if $ I - ZZ ^* - r^2 Z_1 Z_1 ^* > 0$. By our choice of $r$, this is the case,
and it follows that $W \in \B ^d _{2n}$ is strictly contractive.  Observe that
$$ WV^* = \bbm ZV^* & 0 \\ rI_n \otimes V_1 ^* & 0 \ebm; \quad \quad \mbox{and} \quad (WV^*) ^k = \bbm (ZV^*) ^k & 0 \\ r (I_n \otimes V_1 ^*) (ZV^*) ^{k-1} & 0  \ebm;\quad  k \geq 2. $$ It follows that
$$ (I -WV^*) ^{-1} = \bbm (I -ZV^*) ^{-1} & 0 \\ r(I_n \otimes V_1 ^*) (I -ZV^* ) ^{-1} & I_n \otimes I_\mc{H} \ebm. $$  Taking $\La := (0 _n , \la )$ and $\Ga = ( \tau , 0 _n )$ then yields
$$ \La (I -WV^*) ^{-1} \Ga ^* = rV_1 ^* \la (I -ZV^*) ^{-1} \tau ^*, $$ and the inductive step follows.
\end{proof}
We now have all the necessary ingredients to prove the main result of this section:
\begin{proof}{ (of Theorem \ref{Kthm})}
By Lemma \ref{Lcontain}, any co-invariant subspace, $\scr{L}$, on which $V^*$ acts isometrically is contained in $\mc{H} '$. It remains to show
that $\mc{H} '$ is co-invariant for $V$, and that $V^* | _{\mc{H} '}$ is an isometry.

Let $U$ be a co-isometric extension (\emph{e.g.} a Cuntz unitary dilation) of $V$ on $\K \supseteq \mc{H}$. If $h \in \mc{H} '$, then $\tau ^* h \in  \mc{H} ' _n$. Hence,
$$ \tau ^* h \perp \K \otimes \C ^n \ominus \ra{V-Z} = (I -UZ^* ) ^{-1} \left( \left( \K \otimes \C ^n \right) \ominus \ran{V}  \otimes \C ^n \right).$$  This happens if and only if
$$ (I -Z U^*) ^{-1} \tau ^* h \in \ran{V} \otimes \C ^n, $$ and this implies
$$ \bigvee _{\substack{\tau, \la \in \C ^n \\ Z \in \B ^d _n }} \la (I -Z U^*) ^{-1} \tau ^* h \in \ran{V}. $$ By the last lemma this happens if and only if
$$\bigvee _{\alpha \in \F ^d } (U^*) ^{\alpha } h \in \ran{V}. $$
Given any $VF \in \ran{V}$, it is easy to see that by definition
$$ (I -ZU^* ) \tau ^* VF  \in \ra{V-Z}. $$ Hence,
 \begin{align*} \ra{V-Z} & \supseteq   (I -ZU^* )\bigvee _{\substack{h \in \mc{H} ' \\ \la \in \C ^n ; \alpha, \beta \in \F ^d}}  \tau ^* (U^*) ^\beta (U^* ) ^\alpha h && \\ & \supseteq (I -ZU^*) \bigvee _{ \substack{ h \in \mc{H} '; \  \alpha \in \F ^d    \\ \tau, \la , \ka \in \C ^n; W \in \B ^d _n}}  \tau ^* \la (I-WU^*) ^{-1} \ka ^* (U^*) ^\alpha h , && \mbox{(by Lemma \ref{NCfact})} \\
& \supseteq \bigvee _{\substack{ h \in \mc{H}' ; \alpha \in \F ^d \\ \ka \in \C ^n}} \ka ^* (U^*) ^\alpha h, && \end{align*}
so that $(U^*) ^\alpha h \in \mc{H} ' $ for all $\alpha \in \F ^d $ and all $h \in \mc{H} '$.

 Given any $h \in \mc{H} ' \subset \ran{V}$, $h = V H$ where  $V^* V H = H$ so that
 $$ V^* h = V ^* V H = U^* V H = U^* h. $$ This proves that $\mc{H} ' $ is co-invariant for $V$ (\emph{i.e.} co-invariant for each $V_k$), and that $V$ is co-isometric on $\mc{H} '$.
\end{proof}

Although Theorem \ref{CKthm} is an immediate consequence of Theorem \ref{Kthm} under the assumption that $V$ is a $d$-partial isometry, the above proof can be modified to prove Theorem \ref{CKthm} directly.

\begin{lemma}\label{RCdecomp}
    Let $T : \mc{H} \otimes \C ^d \rightarrow \mc{H}$ be a row contraction. Then $\mc{H} \otimes \C ^d$ decomposes as $\mc{H} \otimes \C ^d = \mc{H} _0 \oplus \mc{H} _1$ and $\mc{H} = \mc{H} _0 ' \oplus \mc{H} _1 '$ where
$ V := T  P _{\mc{H} _0}$ is a row partial isometry with initial space $\mc{H} _0 \subseteq \mc{H} \otimes \C ^d$ and final space $\mc{H} _0 ' \subseteq \mc{H}$. The row contraction $C := -T P _{\mc{H} _1}$ is a pure row contraction, \emph{i.e.}, $\| C \mbf{h} \| < \| \mbf{h} \|$ for any $\mbf{h} \in \mc{H} _1 \otimes \C ^d$, with final space
$\ran{C} \subseteq \mc{H} _1 ' = \ran{V} ^\perp$.
\end{lemma}

\begin{proof}
 Let $\mc{H} _0  := \ran{D_T} ^\perp = \ker{D_T}$, where recall $D_T = \sqrt{ I - T ^*T}$. It is clear that $V | _{\mc{H} _0 }$ is an isometry, and can be extended to a row partial isometry on $\mc{H} \otimes \C ^d$ with initial space $\mc{H} _0$. It follows that $C := V -  T$ is a pure contraction and by Lemma \ref{contractext} the initial space of $C$ is contained in $\mc{H} _1 = \mc{H} _0 ^\perp$ and the final space of $C$ is contained in $\mc{H} _1 ' = (\mc{H} _0 ' ) ^\perp$.
\end{proof}

\begin{defn} \label{PIPCdecomp}
    The above decomposition $T = V -C$ of any row contraction on $\mc{H}$ into a row partial isometry, $V$, on $\mc{H} \otimes \C ^d$ and a pure row contraction
$C$ on $\mc{H} \otimes \C ^d$ with $\ker{C} ^\perp \subseteq \ker{V}$ and $\ran{C} \subseteq \ran{V} ^\perp$, will be called the \emph{isometric-pure decomposition} of $T$.
\end{defn}

\begin{proof}{ (of Corollary \ref{CCNCcor})}
By Lemma \ref{RCdecomp}, it follows that $T$ is CNC if and only if its partial isometric part $V$ is CNC. Since $T \supseteq V$ is a row contractive extension of $V$ on $\mc{H}$, Lemma \ref{TextVlemma}, Lemma \ref{NCfact} and Theorem \ref{Kthm} imply that $T$ is CNC if and only if
\ba \mc{H} & = & \bigvee _{\la \in \C ^n; \ Z \in \B ^d _n} \la (I -VZ^* ) ^{-1} \ran{V} ^\perp \otimes \C ^n \nn \\
& = & \bigvee _{\la \in \C ^n; \ Z \in \B ^d _n} \la (I - TZ^* ) ^{-1} \ran{D_{T^*}} \otimes \C ^n \nn \\
& = & \bigvee _{\alpha \in \F ^d} T^\alpha \ran{D _{T^*}}. \nn \ea

If $T$ is CCNC then it is clearly CNC, since in this case,
$$ \bigvee _{\alpha \in \F ^d} T^\alpha \ran{D_{T^*}} \supseteq \bigvee _{\n \in \N ^d} T^\n \ran{D_{T^*}} = \mc{H}. $$
\end{proof}

\section{Model maps for CCNC row partial isometries}
\label{modelmapsect}

Let $V : \mc{H} \otimes \C ^d \rightarrow \mc{H}$ be a CCNC row partial isometry. By definition,
$$ \mc{H} = \bigvee _{z \in \B ^d} \ra{V-z} ^\perp = \bigvee _{z \in \B ^d}  (I -Vz^*)^{-1} \ran{V} ^\perp,$$ and $V$ is CNC by the results of the previous section.

\begin{defn}
A \emph{model triple}, $(\ga , \J _\infty , \J _0)$ for a CCNC row partial isometry $V$ on $\mc{H}$ consists of auxiliary Hilbert spaces $\J _\infty , \J _0$ of dimension $\ker{V}$ and $\ran{V} ^\perp$, respectively, and a \emph{model map}, $\ga$, defined on $\B ^d \cup \{ \infty \}$:
$$ \ga  : \left\{  \arraycolsep=2pt\def\arraystretch{1.2} \begin{array}{c}  \B ^d \rightarrow  \scr{L} ( \J _0 , \ra{V-z} ^\perp ) \\ \{ \infty \} \rightarrow \scr{L} (\J _\infty , \ker{V} ) \end{array}  \right.,$$ such that $\ga (z)$ is a linear isomorphism for any $z \in \B ^d$ and $\ga (0), \ga (\infty)$ are onto isometries.

We say that $(\Ga , \J _\infty , \J _0)$ is an \emph{analytic model triple} and that $\Ga$ is an \emph{analytic model map} if $z \mapsto \Ga (z)$
is anti-analytic in $ \B ^d$.
\end{defn}

Lemma \ref{TextVlemma} shows that any row contractive extension $T \supseteq V$ of a row partial isometry on $\mc{H}$ gives rise to an analytic model map:
Let $\Ga _T (0) : \J _0 \rightarrow \ran{V} ^\perp,  \ \Ga _T (\infty) : \J _\infty \rightarrow \ker{V}$ be any two fixed onto isometries and then
$$ \Ga _T (z) := (I -Tz^* ) ^{-1} \Ga _T (0); \quad \quad z \in \B ^d,$$ defines an analytic model triple. Most simply, we are free to choose $T =V$.

Let $(\Ga , \J _\infty , \J _0)$ be an analytic model triple for a CCNC row partial isometry $V$ on $\mc{H}$.  For any $h \in \mc{H}$, define $$ \hat{h} ^\Ga (z) := \Ga (z) ^* h, $$ an analytic $\J _0$-valued function on $\B ^d$. (When it is clear from context, we will sometimes omit the superscript $\Ga$.) Let $\hat{\mc{H}} ^\Ga$ be the vector space of
all $\hat{h} = \hat{h}^\Ga$ for $h \in \mc{H}$. We endow $\hat{\mc{H} } ^\Ga$ with an inner product:
$$ \ip{\hat{h}}{\hat{g}} _{\Ga} := \ip{h}{g} _\mc{H}. $$

\begin{prop}
    The sesquilinear form $\ip{\cdot}{\cdot} _{\Ga}$ is an inner product on $\hat{\mc{H} } ^\Ga$, and $\hat{\mc{H}} ^\Ga$ is a Hilbert space with respect to this inner product.

    The Hilbert space $\hat{\mc{H}}^\Ga$ is the reproducing kernel Hilbert space of $\J _0$-valued analytic functions on $\B ^d$ with reproducing kernel
    $$ \hat{K} ^\Ga (z,w) := \Ga (z) ^* \Ga (w) \in \scr{L} (\J _0), \quad \quad \hat{\mc{H}} ^\Ga = \mc{H} (\hat{K} ^\Ga ). $$ The map $\hat{U} ^\Ga : \mc{H} \rightarrow \hat{\mc{H} } ^\Ga $ defined by $\hat{U} ^\Ga h = \hat{h} ^\Ga$ is an onto isometry and
    $$ \hat{K} _z ^\Ga g = \hat{U} ^\Ga \Ga (z) g = \widehat{\Ga (z)g}; \quad \quad g \in \J _0. $$
\end{prop}

\begin{proof}
    This is all pretty easy to verify. For simplicity we omit the superscript $\Ga$. In order to show that $\ip{\hat{h}}{\hat{g}} _\Ga := \ip{h}{g} _\mc{H}$ is an inner product, the only non-immediate property to check is that there are no non-zero vectors of zero length with respect to this sesquilinear form, or equivalently that $\hat{h} (z) = 0 $ for all
$z \in \B ^d $ implies that $h  = 0$. This is clear since $\hat{h} (z) = \Ga (z) ^* h$, and since $V$ is CCNC,
$$ \bigvee _{z \in \B ^d } \ran{\Ga (z) } = \bigvee _{z\in \B ^d} \ra{V-z} ^\perp = \mc{H}, $$ so that $\bigcap _{z \in \B ^d } \ker{\Ga (z) ^*} = \{ 0 \}$.

The map $\hat{U} : \mc{H} \rightarrow \hat{\mc{H}}$ is an onto isometry by definition of the inner product in $\hat{\mc{H}}$. For any $g \in \J _0$, compute
\ba \ip{\hat{U} \Ga (z) g }{ \hat{h} }_{\hat{\mc{H}}} & = & \ip{\Ga (z) g}{h} _{\mc{H}} \nn \\
& = & \ip{g}{\Ga (z) ^* h} _{\J _0} \nn \\
& = & \ip{g}{\hat{h} (z)} _{\J _0} \nn \\
& =: & \ip{g}{\hat{K} _z ^* \hat{h}} _{\J _0}. \ea This proves simultaneously that $\hat{\mc{H}}$ is a RKHS of analytic $\J _0$-valued functions on $\B ^d$ with point evaluation maps
$$ \hat{K} _z = \hat{U} \Ga (z) \in \scr{L} (\J _0 , \hat{\mc{H}} ), $$ and reproducing kernel
$$ \hat{K} (z,w) = \hat{K} _z ^* \hat{K} _w = \Ga (z) ^* \Ga (w) \in \scr{L} (\J _0 ). $$
\end{proof}
\begin{prop} \label{multbyz}
Let $V : \mc{H} \otimes \C ^d \rightarrow \mc{H} $ be a CCNC row partial isometry with analytic model map $\Ga$.
The image, $\hat{V} ^\Ga := \hat{U} ^\Ga V (\hat{U} ^\Ga ) ^*$ of $V$ under the unitary map $\hat{U} ^\Ga$ onto the model RKHS $\hat{\mc{H}} ^\Ga$ has initial and final spaces
 \begin{align*} &\ker{\hat{V} ^\Ga } ^\perp \subseteq \{ \hat{\mbf{h}}  \in \hat{\mc{H}} ^\Ga   \otimes \C ^d | \ z \hat{\mbf{h}}  (z) = \hat{g}  (z) \ \mbox{for some} \ \hat{g}   \in \hat{\mc{H}} ^\Ga  \ \mbox{with} \ \|\hat{\mbf{h}}  \|  = \| \hat{g}  \| \}, && \\
 & \ran{\hat{V} ^\Ga }  = \{ \hat{h}  \in \hat{\mc{H}} ^\Ga  | \ \hat{h}  (0) =0 \}, && \end{align*} and $\hat{V} ^\Ga$ acts as multiplication by $z = (z_1 , ... , z_d)$ on its
initial space: For any $\hat{\mbf{h}}  = (\hat{h} _1   , ... , \hat{h} _d   ) ^T \in \ker{\hat{V} ^\Ga} ^\perp$,
$$ (\hat{V} ^\Ga \hat{\mbf{h}})  (z) = \left( (\hat{V} _1 ^\Ga, ..., \hat{V}_d ^\Ga ) \bbm \hat{h} _1  \\ \vdots \\  \hat{h} _d  \ebm \right) (z) = z_1 \hat{h} _1  (z) + ... + z_d \hat{h} _d  (z) = z \hat{\mbf{h}}  (z). $$
\end{prop}

\begin{proof}
Again, we omit the superscript $\Ga$ to simplify notation. First suppose that $\hat{\mbf{h}} \in \ker{\hat{V}} ^\perp$. Then,
\ba \hat{V} \hat{\mbf{h}} & = & ( \hat{U} V \mbf{h} ) (z) \nn \\
& = & \Ga (z) ^* (V \mbf{h} ) \nn \\
& = & \Ga (z) ^* \left( (V -z ) \mbf{h} +z \mbf{h} \right) \nn \\
& = & \Ga (z) ^* z \mbf{h} \quad \quad \quad \left(\mbox{since} \ (V-z) \mbf{h} = (V-z) V^*V \mbf{h} \in \ra{V-z} = \ran{\Ga (z) } ^\perp \right) \nn \\
& = & z (\Ga (z) ^* \otimes I_d ) \mbf{h} \nn \\
& = & z\hat{\mbf{h}} (z). \nn \ea
This proves that $\hat{V}$ acts as multiplication by $z$ on its initial space and that
$$ \ker{\hat{V}} ^\perp \subseteq \{ \hat{\mbf{h}} \in \hat{\mc{H}} \otimes \C ^d | \ z \hat{\mbf{h}} (z) \in \hat{\mc{H}} \ \mbox{and} \ \| z \hat{\mbf{h}} \| = \| \hat{\mbf{h}} \| \}, $$
since $\hat{V}$ is a partial isometry.  The range statement is clear since $ 0 = \hat{h} (0) = \Ga (0) ^* h$ if and only if $h \in \ran{V}$.
\end{proof}

Any model triple $(\ga , \J _\infty , \J _0)$ for a CCNC row partial isometry $V$ can be used to define a \emph{characteristic function}, $b_V ^\ga$, on
$\B ^d$ as follows:
First consider
 $$D ^\ga (z) := \ga (z) ^* \ga (0) = \hat{K} ^\ga (z, 0) \in \scr{L} (\J _0 ), $$
$$ N ^\ga (z) := (\ga (z) ^* \otimes I_d) \ga (\infty ) \in \scr{L} (\J _\infty , \J _0 \otimes \C ^d). $$ The next lemma below shows that $D^{\ga} (z)$
is always a bounded, invertible operator on $\J _0$, so that we can define
\be b_V ^\ga (z) := D ^\ga (z) ^{-1} z N^\ga (z) \in \scr{L} (\J _\infty  , \J _0 ). \ee The function $b_V ^\ga$ is called a (representative) \emph{characteristic
function} of the CCNC row partial isometry $V$. Observe that if $\ga = \Ga$ is an analytic model map then $b_V ^\Ga$ is analytic on $\B ^d$.

\begin{lemma} \label{Dinvert}
    Given any model triple $(\ga , \J _\infty , \J _0)$ for a CCNC row partial isometry $V : \mc{H} \otimes \C ^d \rightarrow \mc{H}$, the operator $D^{\ga } (z) := \ga (z) ^* \ga (0)$ is invertible for any $z \in \B ^d$.
\end{lemma}

\begin{proof}
    Let $D(z) := D^\ga (z) = \ga(z) ^* \ga (0) \in \scr{L} (\J _0 )$, we first show that $D(z) ^*$ has trivial kernel. Suppose that $D(z) ^* g = 0$. Then for any $h \in \J _0$,
\ba 0 & = & \ip{h}{D(z)^* g} _{\J _0}  = \ip{h}{\ga(0) ^* \ga (z) g}_{\J _0} \nn \\
 & = & \ip{\ga (0) h}{\ga (z)g} _\mc{H}, \label{zerogform} \ea since we assume $D(z) ^* g =0$. By definition, $\ga (z) : \J _0 \rightarrow \ra{V-z} ^\perp$ is an isomorphism (bounded, onto, invertible), and $\ga (0)$ is an isometry onto $\ran{V} ^\perp$. By Lemma \ref{TextVlemma}, $\ga (z) g = (I -Vz^*) ^{-1} \ga (0) f$ for some non-zero $\ga (0) f \in \ran{V} ^\perp$. Using that $\ga (0) f \perp \ran{V}$, and that
$Vz^* = V_1 z_1 ^* + ... + V_d z_d ^*$ is a strict contraction for $z \in \B ^d$,
 \ba 0  &= & \ip{\ga (0) f}{\ga (z) g} _\mc{H} \quad \quad \mbox{(Equation \ref{zerogform} with $h = f$)} \nn \\
 & = &  \ip{\ga (0) f }{(I-Vz^* ) ^{-1} \ga (0) f} \nn \\
 &= & \sum _{k=0} ^\infty \ip{\ga (0) f}{ (Vz^*) ^k \ga (0) f } \quad \quad \mbox{(convergent geometric series)} \nn \\
 & = & \ip{\ga (0) f}{\ga (0) f} \quad \quad \mbox{($(Vz^* ) ^k \ga (0) f \in \ran{V}$ for $k\geq 1$)}, \nn \ea a contradiction. In the case where $\J _0$ is finite dimensional, this is enough to prove invertibility.

It remains to prove that $D(z)$ is bounded below. If not, then there is a sequence $(g_n ) \subset \J _0$ such that $\| g_n \| =1$ for all $n$, but
$ \| D (z) g_n \| \rightarrow 0$. Set $G_n := \ga (0) g_n$, $\|G_n \| = 1$ since $\ga (0)$ is an isometry. Again using that $\ga (z) : \J _0 \rightarrow \ra{V-z} ^\perp$, and $(I-Vz^* ) : \ra{V-z} ^\perp \rightarrow \ran{V} ^\perp$ are isomorphisms, if $H_n \in \ra{V-z} ^\perp$ and $h_n \in \J _0$ are chosen so that
$(I -Vz^*) H_n = G_n$ and $\ga (z) h_n = H_n$, then the norms of the sequences $(h_n ), \ (H_n )$ are uniformly bounded above and below by strictly positive constants. Then,
\ba \ip{h_n}{D(z) g_n} & = & \ip{\ga (z) h_n}{\ga (0) g_n} = \ip{H_n}{G_n} \nn \\
& = & \ip{H_n}{(I-Vz^*) H_n} \nn \\
& = & \ip{H_n}{(I-zV^*Vz^*) H_n} - \ip{H_n}{(V- z)V^*Vz^*H_n} \nn \\
& = & \ip{H_n}{(I-zV^*Vz^*) H_n}  \quad \quad \mbox{(Since $H_n \perp \ra{V-z}$.)} \nn \\
& \geq & (1 - |z| ^2 ) \| H_n \| ^2. \nn \ea
Since there are constants $C, c>0$ so that $\| h _n \| \leq C$, and $\| H _n \| \geq c >0$, the limit of the left hand side is zero while the right hand side
is bounded below by a strictly positive constant, a contradiction.
\end{proof}

\begin{defn}
    Let $b_1, b_2$ be two analytic functions on $\B ^d$ taking values in $\scr{L} (\mc{H} _1 , \J _1)$ and $\scr{L} (\mc{H} _2 , \J _2)$, respectively. We say that $b_1, b_2$ \emph{coincide} if there are fixed unitary $R \in \scr{L} (\J _1 , \J _2), Q \in \scr{L} ( \mc{H} _1 , \mc{H} _2 )$ so that $$ R b_1 (z) = b_2 (z) Q ; \quad \quad z \in \B ^d.$$ This clearly defines an equivalence relation on such functions, and we call the corresponding equivalence classes coincidence classes.
\end{defn}

\begin{lemma} \label{modmapinv}
    Let $V$ be a CCNC row partial isometry. Let $b_V ^\ga$ be a characteristic function for $V$ defined using any model triple $(\ga , \J _\infty , \J _0 )$. The coincidence class of $b_V ^\ga$ is invariant under the choice of model triple.
\end{lemma}
    It follows, in particular, that any characteristic function $b^\ga _V$ for $V$ is analytic whether or not $\ga$ is an analytic model map, since $b^\ga _V$ coincides with $b^\Ga _V$ for any analytic model map $\Ga$ (and analytic model maps always exist).
\begin{proof}
Let $(\ga , \J _\infty , \J _0 ), (\delta,  \K _\infty , \K _0 )$ be any two choices of model triples for $V$. Since both $\K _0 , \J _0$ are isomorphic to $\ran{V} ^\perp$ and $\K _\infty , \J _\infty$ are isomorphic to $\ker{V}$, we can define onto isometries $R = \ga (0) ^* \delta (0)   \in  \scr{L} (\K _0 , \J _0 )$ and
$Q :=  \ga (\infty ) ^* \delta (\infty)  \in \scr{L} ( \K _\infty , \J _\infty )$. Moreover $C _z  := \ga (z) ^* (\delta (z)  \delta (z) ^*  ) ^{-1} \delta (z)$ is a linear isomorphism (bounded and invertible) of $\K _0$ onto $\J _0$ for any $z \in \B ^d$. As before,
$$  D ^\ga (z) := \ga (z) ^* \ga (0) ; \quad \quad N ^\ga (z) = (\ga (z) ^* \otimes I_d ) \ga (\infty ), $$
$$ b _V ^\ga (z) := D^\ga (z) ^{-1} z N ^\ga (z), $$ and $b_V ^\delta$ is defined analogously. It follows that:
\ba R b ^\delta _V (z) Q & = & R D ^\delta (z) ^{-1}   C_z ^{-1} C _z z N ^\delta (z) Q \nn \\
& = & \left( C_z D ^\delta (z) R ^* \right) ^{-1} z \left( (C_z \otimes I_d) N^\delta (z) Q \right). \nn \ea
In particular,
\ba C_z D ^\delta (z) R ^* & = & \ga (z) ^* (\delta (z) \delta (z)^* ) ^{-1} \delta (z) \delta (z) ^* \delta (0) \delta (0) ^* \ga (0) \nn \\
& = & \ga (z) ^* \ga (0) \nn \\
& = & D ^\ga (z), \nn \ea and similarly
$$ (C_z \otimes I_d) N^\delta (z) Q  = N ^\ga (z), $$ so that $R b ^\delta _V (z) Q = b ^\ga (z)$, and $b ^\delta _V, b ^\ga _V$ belong to the same equivalence (coincidence) class.
\end{proof}

\begin{thm} \label{Glerep}
Let $\hat{\mc{H}} ^\Ga $ be the abstract model RKHS on $\B ^d$ defined using an analytic model triple $(\Ga , \J _0 , \J _\infty)$ for $V$. The reproducing kernel for $\hat{\mc{H}} ^\Ga$ can be written:
\ba \hat{K} ^\Ga (z,w) & = & \frac{D ^\Ga (z)D ^\Ga (w) ^* - zN ^\Ga (z) N ^\Ga (w) ^* w^*}{1-zw^*} \nn \\
& = &  D ^\Ga (z)  \frac{I - b_V ^\Ga (z) b_V ^\Ga (w) ^*}{1-zw^*} D ^\Ga (w)^* \in \scr{L} (\J _0 ). \nn \ea
It follows that $b_V ^\Ga \in \scr{S} _d (\J _\infty , \J _0 ) $ is Schur-class, and that multiplication by $D ^\Ga (z)$ is an isometry, $M^\Ga := M_{D ^\Ga}$,
of $\scr{H} (b_V ^\Ga )$ onto $\hat{\mc{H}} ^\Ga$.
\end{thm}
\begin{proof}
As before, we omit the superscript $\Ga$ for the proof. Consider the projections:
$$ P_0 := P _{\ran{\hat{V}} ^\perp} =\hat{U} \Ga (0) \Ga (0) ^* \hat{U} ^*, $$  and, $$ P _\infty := P _{\ker{\hat{V}}} = (\hat{U} \otimes I_d) \Ga (\infty ) \Ga (\infty ) ^* ( \hat{U} ^* \otimes I_d ).$$ Then calculate,
\ba \hat{K} _z ^* z \hat{V} ^* \hat{K} _w & = & \left( \hat{K} _w ^* \hat{V} (I - P _\infty ) z^* \hat{K} _z ^* \right) ^* \nn \\
& = & \left( \hat{K} _w ^* w (I - P_\infty ) z^* \hat{K} _z \right) ^*  \nn \\
& = & \hat{K} _z ^* z (I - P _\infty ) w^* \hat{K} _w \nn \\
& = & zw^* \hat{K} (z,w) - \hat{K} _z ^* z P _\infty w^* \hat{K} _w. \nn \ea  In the above we used that $\hat{V}$ acts as multiplication by $z$ on its initial space and $I - P_\infty = P _{\ker{V}} ^\perp$ is the projection onto this initial space.

The same expression can be evaluated differently:
\ba \hat{K} _z ^* z \hat{V} ^* \hat{K} _w & = & \hat{K} _z ^* z \hat{V} ^* (I - P_0) \hat{K} _w \nn \\
& = & \hat{K} _z ^* \hat{V} \hat{V} ^* (I - P_0 ) \hat{K} _w \nn \\
& = & \hat{K} _z ^* (I - P_0 ) \hat{K} _w \quad \quad \quad  (I- P_0) = P _{\ran{\hat{V}}} = \hat{V} \hat{V} ^* \nn \\
& = & \hat{K} (z,w) - \hat{K} _z ^* P_0 \hat{K} _w. \nn \ea Again, in the above we used that $\hat{V}$ acts as multiplication
by $z$ on its initial space, the range of $\hat{V} ^*$.

Equating these two expressions and solving for $\hat{K} (z,w)$ yields:
$$ \hat{K} (z,w) = \frac{ \hat{K} _z ^* P_0 \hat{K} _w - \hat{K} _z ^* z P_\infty w^* \hat{K} _w}{1-zw^* }. $$
Use that $\hat{K} _z = \hat{U} \Ga (z)$, $P_0 = \hat{U} \Ga (0)  \Ga (0) ^* \hat{U} ^*$, $P_\infty = (\hat{U} \otimes I_d) \Ga (\infty ) \Ga (\infty ) ^* (\hat{U} ^* \otimes I_d)$, and that
$$ \hat{K} _z ^* z P_\infty w^* \hat{K} _w = \Ga (z) ^* z \Ga (\infty ) \Ga (\infty ) ^* w^* \Ga (w)$$ to obtain
$$ \hat{K} (z,w) = \frac{ D(z) D(w) ^* - zN(z) N(w) ^* w^*}{1-zw^*}; \quad \quad z,w \in \B^d. $$  In particular
it follows that $b_V = D(z) ^{-1} z N(z) \in  \scr{S} _d ( \J_\infty , \J _0 )$ as claimed (by \cite[Theorem 2.1]{Ball2001-lift}).
\end{proof}

Given an analytic model map $\Ga$ for a CCNC row partial isometry $V$ on $\mc{H}$, let $U ^\Ga : \mc{H} \rightarrow \scr{H} (b_V ^\Ga )$ denote the canonical onto isometry $U ^\Ga := (M^\Ga) ^{-1} \hat{U} ^\Ga$,
where, recall, $(M^\Ga) ^{-1} = M_{D^\Ga } ^{-1} = M_{(D^\Ga ) ^{-1} } =(M^\Ga ) ^*$, since $M^\Ga = M_{D^\Ga}$ is a unitary multiplier of $\scr{H} (b_V ^\Ga )$ onto $\hat{\mc{H}}  ^\Ga$.

\begin{thm} \label{Glerep2}
Let $V : \mc{H} \otimes \C ^d \rightarrow \mc{H}$ be a CCNC row partial isometry with analytic model triple $(\Ga , \J _\infty , \J _0 )$. The image, $X ^\Ga := U ^\Ga V ((U ^\Ga) ^* \otimes I_d)$, of $V$ under the corresponding canonical unitary is an extremal Gleason solution for $\scr{H} (b_V ^\Ga )$.

If $\Ga (z) = \Ga _T (z) := (I - Tz^* ) ^{-1} \Ga (0)$ is an analytic model map corresponding to a contractive extension $T \supseteq V$ on $\mc{H}$, then $X^{\Ga _T} = X (\mbf{b} ^{T} )$  is the unique contractive Gleason solution corresponding to the extremal Gleason solution
$$ \mbf{b} ^T := (U ^{\Ga _T} \otimes I_d) \Ga _T (\infty ), \quad \quad \mbox{for} \ b_V ^{\Ga _T} \in \scr{S} _d (\J _\infty , \J _0 ). $$
\end{thm}

\begin{proof}
(As before we will omit superscripts.) Since multiplication by $D(z) ^{-1}$ is an isometric multiplier of $\hat{\mc{H}}$ onto $\scr{H} (b_V)$, $M_{D^{-1}} = M_D ^*$, Proposition \ref{multbyz} implies that $X   := M_D ^* \hat{V} M_D$ acts as multiplication by $z$ on its initial space, and
$$ \ran{X} = \{ f \in \scr{H} (b_V) | \ f(0) = 0 \}. $$

Hence, $X X ^* = I - k_0 ^b (k_0 ^b) ^*$ and
\ba z (X ^* f ) (z) & = & (X X^* f ) (z) \nn \\
& = & f(z) - k^{b_V} (z,0 ) f(0) \nn \\
& = & f (z) - f(0), \nn \ea since $b _V ^\Ga (0) =0$. This proves that $X$ is a (contractive) extremal Gleason solution for $\scr{H}(b_V ^\Ga)$.

To see that $\mbf{b}  := (M_{D ^{-1} } \hat{U} \otimes I_d)  \Ga (\infty )$ is a contractive extremal Gleason solution for $b_V ^\Ga$, let $b := b_V ^\Ga$ and calculate:
\ba z \mbf{b}  (z) & =& ( z^* k_z ^b ) ^* \mbf{b}  \nn \\
& = & \left( z^* \hat{U}  ^* (M_{D^{-1} } ^* ) k_z ^b \right) ^* \Ga (\infty ) \nn \\
& = & \left( z^* \hat{U}  ^* \hat{K}_z  (D (z) ^{-1})  ^* \right) ^* \Ga (\infty ) \nn \\
&= & \left( z^* \Ga (z) (D  (z) ^{-1} ) ^* \right) ^* \Ga (\infty ) \nn \\
& = & D (z) ^{-1} z (\Ga (z) ^* \otimes I_d ) \Ga (\infty ) \nn \\
& = & D (z) ^{-1} z N (z) = b(z). \nn \ea This proves that $\mbf{b} ^\Ga =: \mbf{b}$ is a Gleason solution (since $b_V ^\Ga (0) = 0$). Furthermore,
$$ \mbf{b}  ^* \mbf{b}  = \Ga (\infty ) ^* \Ga (\infty) = I_{\J _\infty }, $$ so that $\mbf{b} ^\Ga = \mbf{b}$ is contractive and extremal.

To see that $X = X^{\Ga _T} = X (\mbf{b} ^{\Ga _T})$, calculate the action of $X^* -w^*$ on point evaluation maps,
\ba (X^* -w^* ) k_w ^b & = & \left(( M_D ^{-1} \hat{U}  \otimes I_d ) (V^* -w^* ) \hat{U}  ^* M_D \right) k_w ^b \nn \\
& = & (M_D ^{-1} \hat{U}  \otimes I_d ) (V^* -w^*) \Ga (w) (D (w) ^{-1} ) ^*, \nn \ea and compare this to
\ba \mbf{b} ^\Ga b(w) ^* & = & (M_D ^{-1} \hat{U}  \otimes I_d) \Ga (\infty) N  (w) ^* w^* (D  (w) ^{-1} ) ^* \nn \\
& = & (M_D ^{-1} \hat{U}  \otimes I _d) \Ga (\infty) \Ga (\infty) ^* (\Ga (w) \otimes I_d)  w^* (D  (w) ^{-1} ) ^* \nn \\
& = & (M_D ^{-1} \hat{U}  \otimes I _d) P _{\ker{V} } w^* \Ga (w) (D  (w) ^{-1} ) ^* \nn \\
& = & (M_D ^{-1} \hat{U}  \otimes I _d) (I - V^* V) w^* \Ga (w) (D  (w) ^{-1} ) ^*. \nn \ea Under the assumption that
$\Ga = \Ga _T $ for a contractive extension $T \supseteq V$ on $\mc{H}$, recall that by Lemma \ref{RCdecomp}, $T = V -C$, where $C$ is a pure row contraction
with $\ker{C} ^\perp \subseteq \ker{V}$ and $\ran{C} \subseteq \ran{V} ^\perp$. Applying that $\Ga _T (z) = (I -Tz^* ) ^{-1} \Ga _T (0)$,
\ba \mbf{b}  b(w) ^* & = & (M_D ^{-1} \hat{U} \otimes I _d) (I - V^* (T+C)) w^* \Ga _T  (w) (D  (w) ^{-1} ) ^* \nn \\
& = &  (M_D ^{-1} \hat{U} \otimes I _d) w^* \Ga _T  (w) (D  (w) ^{-1} ) ^*   \nn \\
& & - (M_D ^{-1} \hat{U}  \otimes I _d) V^* ( C w^* \Ga _T (w) + \Ga _T (w) - \Ga _T (0) ) (D  (w) ^{-1} ) ^* \nn \\
& = & (M_D ^{-1} \hat{U}  \otimes I _d) (w^* - V^*) \Ga _T  (w) (D  (w) ^{-1} ) ^*, \nn \ea
and it follows that $X ^* k_w ^b = w^* k_w ^b - \mbf{b} ^{\Ga _T} b^{\Ga _T} _V (w) ^*$, proving the claim.
\end{proof}

\begin{remark} \label{qmodelmaps}
Lemma \ref{modmapinv} shows that the coincidence class of any characteristic function $b_V ^\ga$ of a CCNC row partial isometry $V$ is invariant under
the choice of model triple $(\ga , \J _\infty , \J _0)$ and Theorem \ref{Glerep} shows that $b_V ^\ga \in \scr{S} _d (\J _\infty , \J _0)$ belongs to the Schur class. It will also be useful to define \emph{weak coincidence} of Schur-class functions as in \cite[Definition 2.4]{BES2006cnc}:

\begin{defn}
    The \emph{support} of $b \in \scr{S} _d (\mc{H}  , \J )$ is
    $$ \mr{supp} (b) := \bigvee _{z \in \B ^d } \ran{b(z) ^*} = \mc{H}  \ominus \bigcap _{z \in \B ^d} \ker{b(z)}. $$
Schur-class multipliers $b_1 \in \scr{S} _d ( \mc{H}  , \J )$ and $b_2 \in \scr{S} _d (\mc{H}  ' , \J ' )$ \emph{coincide weakly} if $b_1 ' := b_1 | _{\mr{supp} (b _1 ) } $ coincides with $b_2 ' := b_2 | _{\mr{supp} (b_2 )}$.
\end{defn}

By \cite[Lemma 2.5]{BES2006cnc}, $b_1 \in \scr{S} _d (\mc{H}  , \J )$ and $b_2 \in \scr{S} _d (\mc{H}  ' , \J ' )$ coincide weakly if and only if there is an onto isometry $V : \J \rightarrow \J ' $ so that
$$ V b_1 (z) b _1 (w) ^* V ^* = b_2 (z) b_2 (w) ^* ; \quad \quad z,w \in \B ^d ,$$ \emph{i.e.}, if and only if $\scr{H} (b_2 ) = \scr{H} (V b_1 )$.

Weak coincidence also defines an equivalence relation on Schur-class functions, and we define:
\begin{defn} \label{Livsicdef}
    The \emph{characteristic function}, $b_V$, of a CCNC row partial isometry $V$ is the weak coincidence class
of any Schur-class characteristic function $b_V ^\ga \in \scr{S} _d (\J _0 , \J _\infty )$ constructed using any model triple $(\ga , \J _0 , \J _\infty )$ for $V$. We will often abuse terminology and simply say that any $b_V ^\ga$ is the characteristic function, $b_V$, of $V$.
\end{defn}
Note that the characteristic function, $b_V$, of any CCNC row partial isometry always vanishes at $0$, $b_V (0) =0$, and (as discussed in Subsection \ref{HerglotzSection}) this implies that $b_V$ is strictly contractive on $\B ^d$. In the single variable case when $d=1$, the above definition of characteristic function reduces to that of the \emph{Liv\v{s}ic characteristic function} of the partial isometry $V$ \cite{Livsic,AMR,GMR}.
\end{remark}

\begin{remark}
There is an alternative proof of Theorem \ref{Glerep} using the colligation or transfer function theory of \cite{Ball2001-lift,Ball2007trans,Ball2010}.
Any model triple for a CCNC row partial isometry, $V$, provides a unitary colligation and transfer-function realization for $b_V$:
\begin{lemma}{ (\cite[Lemma 3.7]{Manana})}
Let $V$ be a CCNC row partial isometry on $\mc{H}$, and let $(\ga , \J _\infty , \J _0)$ be any model triple for $V$. Then
$$ \Xi ^\ga := \bbm V^* & \ga (\infty ) \\ \ga (0 ) ^* & 0 \ebm : \bbm \mc{H} \\ \J _\infty \ebm \rightarrow \bbm \mc{H} \otimes \C ^d \\ \J _0 \ebm, $$
is a unitary (onto isometry) colligation with transfer function equal to $b_V ^\ga$.
\end{lemma}
In \cite[Theorem 3.14]{Manana}, the above lemma was applied to provide an alternate proof of Theorem \ref{Glerep}.
\end{remark}

\begin{lemma} \label{Gleasoncharfun}
Let $b \in \scr{S} _d (\J , \K )$ be any Schur-class function such that $b(0) = 0$. The characteristic function, $b_X$, of any extremal
Gleason solution, $X$, for $\scr{H} (b)$ coincides weakly with $b$.
\end{lemma}

\begin{proof}
By Remark \ref{GSrowpi} and Lemma \ref{GSCCNC}, any such $X$ is a CCNC row partial isometry.

To calculate the characteristic function of $X$, use the analytic model map $\Ga = \Ga _X$, with $\Ga (0) := k_0 ^b$, an isometry of $\K $ onto $\ran{X} ^\perp$
since $b(0) = 0$. Then,
\be  \Ga (z) = (I -X z^* ) ^{-1} k_0 ^b = k_z ^b, \label{Gaker} \ee and
$$ D(z) = \Ga (z) ^* \Ga (0) = (k_z ^b) ^* k_0 ^b = k^b (z,0) = I _{\K}. $$
It then follows that $zN(z) = b_X (z)$, and by equation (\ref{Gaker}) and Theorem \ref{Glerep}, the reproducing kernel for the model space $\widehat{\scr{H} (b)}$ is
$$ \hat{K} (z,w) = \Ga (z) ^* \Ga (w) = k^b (z,w) = k^{b_X} (z,w); \quad \quad z,w \in \B ^d. $$
This proves that $\scr{H} (b_X ^\Ga) = \scr{H} (b)$ and \cite[Lemma 2.5]{BES2006cnc} implies that $b_X ^\Ga$ coincides weakly with $b$.
\end{proof}

\begin{cor} \label{Gmultbyz}
    If $X$ is an extremal Gleason solution for any Schur-class $b$ such that $b(0) =0$ then $X$ acts as multiplication by $z = (z_1 , ..., z_d)$ on its initial space.
\end{cor}

\begin{proof}
If we take $\Ga = \Ga _X$ as our model map for $X$, the above example shows that $\widehat{\scr{H} (b)} ^\Ga = \scr{H} (b)$ and $X = \hat{X}$, so that
$X$ acts as multiplication by $z$ on its initial space by Proposition \ref{multbyz}.
\end{proof}

\begin{remark} \label{Glemark}
Theorem \ref{Glerep} and Lemma \ref{Gleasoncharfun} show that the (weak coincidence class of any) characteristic function is a unitary invariant: If $X, Y$ are two CCNC row partial isometries which are unitarily equivalent then one can choose model triples for $X, Y$ so that they have the same characteristic function. Conversely if the characteristic functions of $X, Y$ coincide weakly, $\scr{H} (b_X )= \scr{H} (U b_Y )$ for a fixed unitary $U$, then both $X, Y$ are unitarily equivalent to extremal Gleason solutions for $\scr{H} (b_X)$ (it is easy to see that multiplication by $U$ is a constant
        unitary multiplier of $\scr{H} (b_X )$ onto $\scr{H} (U b_X)$ taking extremal Gleason solutions onto extremal Gleason solutions).

 Unless $d=1$, the characteristic function of a CCNC row partial isometry is not a complete unitary invariant.
Indeed, if $b \in \scr{S} _d (\J , \K)$, $b(0) =0$, has two extremal Gleason solutions $X, X'$ which are not unitarily equivalent,
the above results show that $X , X'$ are two non-equivalent row partial isometries with the same characteristic function. The subsection below provides examples of such Schur-class functions.
\end{remark}

\subsection{The characteristic function is not a complete unitary invariant} \label{Gleg}

Let $b \in \scr{S} _d (\mc{H} )$ be any square Schur-class function with $b(0)=0$ (for simplicity) and $\mr{supp} (b) = \mc{H}$. Recall that since $b(0)$ is a strict contraction, so is $b(z)$ for any $z \in \B ^d$.

\begin{lemma} \label{neqcoext}
   Suppose that $b \in \scr{S} _d (\mc{H})$ obeys $b(0) =0$, and that $\scr{H} (b)$ has two extremal Gleason solutions $X, Y $ which are unitarily equivalent, $WX = Y W$, for some unitary $W \in \scr{L} (\scr{H} (b) )$.
Then $W$ is a constant unitary multiplier: There is a constant unitary $R \in \scr{L} (\mc{H})$ so that $W k_z ^b = k_z ^b R$, $(WF) (z) = R^* F(z)$ for all $F \in \scr{H} (b)$ and $z \in \B ^d$. Moreover, given any $z,w \in \B^d$, $R b(z) b(w)^* = b(z) b(w) ^* R$.
\end{lemma}

\begin{proof}
Suppose that $X, Y$ are unitarily equivalent, $WX = Y X$ for some unitary $W : \scr{H} (b) \rightarrow \scr{H} (b)$. Since both $X, Y$ are extremal and $b(0) =0$,
they are both row partial isometries with the same range projection $XX ^* = I - k_0 ^b (k_0 ^b ) ^* = Y Y^*$. It follows that the range of $k_0 ^b$ is a reducing subspace for $W$ so that there
is a unitary $R \in \scr{L} (\mc{H} )$ such that
$$ W k_0 ^b = k_0 ^b R. $$

Moreover, by the Property (\ref{Gleker}), given any $z \in \B ^d$,
\ba W k_z ^b & = & W (I -X z^* ) ^{-1} k _0 ^b \nn \\
& = & (I -Y z^* ) ^{-1} W k_0 ^b \nn \\
& = & (I -Y z^* ) ^{-1} k_0 ^b R = k_z ^b R, \nn \ea so that $W $ is multiplication by the fixed unitary operator $R^* \in \scr{L} (\mc{H} )$. Since $W$ is a unitary multiplier of $\scr{H} (b)$ onto itself, RKHS theory implies that 
$$ k^b (z,w) = R^* k^b (z,w) R; \quad \quad z, w \in \B ^d, $$ and the final claim follows.
\end{proof}

\begin{cor} \label{ncuni}
    If $b \in \scr{S} _d (\C)$ satisfies $b(0) = 0$ and  $\dim{\ker{V^b}} > \dim{\ran{V^b} ^\perp} \neq 0$, then
$\scr{H} (b)$ has two extremal Gleason solutions which are not unitarily equivalent.
\end{cor}

\begin{proof}
In this case $\mc{H} = \C$ the dimension of $\ran{V^b} ^\perp$ is either $0$ or $1$.  By assumption $\dim{\ker{V^b} } > 1 = \dim{\ran{V^b}}$ so that
we can define two co-isometric extensions $D,d$ of $V^b$ on $\scr{H} ^+ (H_b)$ so that
$$ D^* K_0 ^b \perp d^* K_0 ^b. $$ As described in Subsection \ref{extGSsect} and Subsection \ref{Gleasonsub}, the contractive Gleason solutions $X[D]$ and $X[d]$ are extremal.

Suppose that $X[D], X[d]$ are unitarily equivalent, $WX[d] = X[D] W$ for a unitary $W$ on $\scr{H} (b)$. By the previous Lemma, $W$ is a constant unitary multiplier by a unitary operator $R \in \scr{L} (\C )$. That is,
$W$ is simply multiplication by a unimodular constant $R = \alpha \in \bm{T}$, $W = \alpha I _{\scr{H} (b)}$.  In particular, $X[D] = WX[d] W^* = X[d]$ so that $\mbf{b} [D] b(w) ^* = \mbf{b}[d] b(w) ^* $ and $\mbf{b} [D] = \mbf{b} [d]$. Equivalently,
$$ D^* K_0 ^b = d^* K_0 ^b, \quad \quad \mbox{a contradiction.}$$
\end{proof}

\begin{eg} \label{notcompeg}
To construct examples of a Schur-class functions $b \in \scr{S} _d (\C)$ which admit non-unitarily equivalent extremal Gleason solutions, let $b \in \scr{S} _d (\C)$, $d\geq 2$ be any Schur-class function obeying
$b(0) =0$ and such that $V^b$ is not a co-isometry. Such a Schur-class function is called non quasi-extreme (see Section \ref{QEsect}), and it is not difficult to apply Theorem \ref{equivqe} below to show, for example, that if $b \in \scr{S} _d (\C)$ is any Schur-class function and $0<r<1$,  $r b \in \scr{S} _d (\C)$ is non quasi-extreme. Moreover, $\scr{H} (rb)$ is infinite dimensional (this is just $H^2 _d$ with a new norm since $\scr{H} (rb) = \ran{D_{M_{rb}^*}}$ equipped with the norm that makes $D_{M_{rb}^*}$ a co-isometry onto its range as in \cite{Sarason-dB}) so that $rb$ satisfies the assumptions of Propostion \ref{Extprop2}, and hence those of Corollary \ref{ncuni}.
\end{eg}

\begin{cor} \label{notcomplete}
The characteristic function, $b_X$, of a CCNC row partial isometry, $X$, is not a complete unitary invariant: There exist CCNC row partial isometries
$X_1, X_2$ which are not unitarily equivalent and yet satisfy $b_{X_1} = b_{X_2}$.
\end{cor}
This is in contrast to the result of \cite[Theorem 3.6]{BES2006cnc}, which shows that the Sz.-Nagy-Foia\c{s} characteristic function of any CNC $d$-contraction (which is a CCNC row contraction by Corollary \ref{CCNCcor}) is a complete unitary invariant for CNC $d$-contractions. We will also later prove in Proposition \ref{NFprop}, that our characteristic function, $b_T$, of any CCNC row contraction coincides weakly with the Sz.-Nagy-Foia\c{s} characteristic function of $T$.
\begin{proof}
Let $b \in \scr{S} _d = \scr{S} _d (\C )$ be any Schur-class function satisfying the assumptions of Corollary \ref{ncuni} (as in the above example). Then there exist extremal Gleason solutions, $X_1, X_2$ for $\scr{H} (b)$ which are not unitarily equivalent. Lemma \ref{Gleasoncharfun} implies that $b_{X_1} = b_{X_2}$ proving the claim.
\end{proof}

We conclude this section with a calculation that will motivate an approach to extending our commutative de Branges-Rovnyak model for CCNC row partial isometries to arbitrary CCNC row contractions:

\subsection{Frostman Shifts} \label{FSeg}
Suppose that $X$ is any extremal Gleason solution for $\scr{H} (b)$, where $b \in \scr{S} _d (\J , \K)$ is purely contractive on the ball, but we do not assume that $b(0) =0$. Recall that $b$ is purely
contractive on the ball if and only if for any $z \in \B ^d$, $\| b (z) g \| < \| g \|$ for all $g \in \J$, \emph{i.e.} $b(z)$ is a pure contraction.  It is further not difficult to show, as in \cite[Proposition V.2.1]{NF}, that if $b(z)$ is a pure contraction for $z\in \B ^d$, then so is $b(z) ^*$. In particular, both defect operators $D_{b(0)}, D_{b (0) ^* }$ are injective and have dense ranges, \emph{i.e.} they are quasi-affinities. In the case where $b(0) \neq 0$, $X$ is not a row partial isometry, but it is still CCNC by Lemma \ref{GSCCNC} so that we can consider the isometric-purely contractive decomposition $X = V -C$ of $X$, and we calculate the characteristic function, $b_V$, of the CCNC partial isometric part, $V$, of $X$.

By Lemma \ref{RCdecomp}, $V^* = X^* P_{\ker{D_{X^*}}}$, and since $X$ is extremal,
$$ D _{X^* } ^2 = I - X X^* = k_0 ^b (k_0 ^b) ^*. $$ It follows that
$$ \ker{D_{X^*} } = \ker{D_{X^*} ^2 } = \left( \bigvee k_0 ^b \K \right) ^\perp, $$ so that if $b(0)$ is a strict contraction,
$$ V^* = X^* P _{\ker{D_{X^*}}} =  X^* \left( I - k_0 ^b k^b (0, 0 ) ^{-1} (k_0 ^b) ^* \right). $$
If $b(0)$ is a strict contraction, we can then define the onto isometry $\Ga (0) : \K \rightarrow \ran{V} ^\perp$ by $\Ga (0) := k_0 ^b D_{b(0) ^*} ^{-1}$. If $b(0)$ is not a strict contraction, one can define $ \Ga (0) : \ran{D_{b(0) ^*}} \rightarrow \ran{V} ^\perp$ by:
$$ \Ga (0) D_{b(0) ^* } g := k_0 ^b g. $$ It is easy to check that this is a densely defined isometry with dense range in $\ran{V} ^\perp$, and this extends by continuity to an isometry of $\K$ onto $\ran{V} ^\perp$.

Similarly, since $X$ is an extremal Gleason solution for $\scr{H} (b)$, as described in Section \ref{Gleasonsub}, there is an extremal Gleason solution
$\mbf{b}$ for $b$ so that $X^* k_w ^b = w^* k_w ^b - \mbf{b} b(w) ^*$. Since $\mbf{b}$ is extremal, $\mbf{b} ^* \mbf{b} = I - b(0) ^* b(0) = D_{b(0)} ^2$,
so that if $b(0)$ is a strict contraction, $\mbf{b} D_{b(0)} ^{-1}$ also defines an isometry.

\begin{lemma} \label{kerViso}
     Define $\Ga _0 (\infty )$ on $\ran{D_{b(0)}}$ by $\Ga _0 (\infty ) D_{b(0)} h = \mbf{b} h$. Then $\Ga _0 (\infty )$ extends by continuity to an
isometry of $\J$ into $\ker{V}$.
\end{lemma}
If $b(0)$ is a strict contraction then $\Ga _0 (\infty ) = \mbf{b} D_{b(0)} ^{-1}$ and the proof simplifies.
\begin{proof}
Showing that $\Ga _0 (\infty )$ extends to an isometry is straightforward. To show that $\Ga _0 (\infty )$ maps into
$\ker{V}$, use that $V^* = X^* (I - \Ga (0) \Ga (0) ^* )$ since $\Ga (0)$ is an isometry onto $\ran{V} ^\perp$, and observe that
$$ k^b (w, 0) = (k_w ^b) ^* k_0 ^b = (k_w ^b) ^* \Ga (0) \Ga (0) ^* k_0 ^b. $$ It follows that for any $h \in \K$ and $w \in \B ^d$, since $\Ga (0)$ is onto $\bigvee k_0 ^b \K$, there is a $g \in \K$ so that $ \Ga (0) \Ga (0) ^* k_w ^b h = k_0 ^b g$, or, equivalently,
$$ k^b (0,w ) h = (k_0 ^b ) ^* \Ga (0) \Ga (0) ^* k_w ^b h = k^b (0, 0) g. $$ Then, calculating on kernel maps,
\ba D_{b(0)} \Ga _0 (\infty ) ^* V^* k_w ^b h & = & \mbf{b} ^* X^*(I - \Ga (0) \Ga (0) ^* ) k_w ^b h \nn \\
& = & \mbf{b} ^* ( w^* k_w ^b - \mbf{b} b(w) ^* ) h - \mbf{b} ^* X^* k_0 ^b g \nn \\
& = & \left( b(w) ^* - b(0) ^* - (I - b(0) ^* b(0) ) b(w ) ^* \right) h  + D_{b(0)} ^2 b(0) ^* g \nn \\
& = &  (-b(0) ^* + b(0) ^* b(0) b(w) ^* ) h + b(0) ^* k^b (0, 0) g \nn \\
& = & -b(0) ^* h + b(0) ^* b(0) b(w) ^* h + b(0) ^* k^b (0, w) h = 0. \nn \ea
\end{proof}

It follows that we can choose an isometry $\Ga _1 (\infty ) : \J ' \rightarrow \ker{V} \ominus \ran{\Ga _0 (\infty)}$ so that
$$ \Ga (\infty) := \Ga _0 (\infty) \oplus \Ga _1 (\infty) : \J \oplus \J ' \rightarrow \ker{V}, $$ is an onto isometry.

\begin{lemma} \label{vanish}
The range of $\Ga _1 (\infty)$ is in the kernel of $z (k_z ^b ) ^*$  for any $z \in \B ^d$. That is, $z(\Ga (z) ^* \otimes I_d ) \Ga _1 (\infty ) =0$.
\end{lemma}
\begin{proof}
For simplicity assume $b(0)$ is a strict contraction. The argument can be extended to the general case as before.
If $\mbf{h} \in \ran{\Ga _1 (\infty)}$ then $\mbf{h} \in \ker{V}$, where $V = (I - k_0 ^b k^b (0,0) ^{-1} (k_0 ^b) ^*) X$ and
$\mbf{h} \in \ker{\mbf{b} ^* } = \ran{\mbf{b}} ^\perp$.
Then calculate:
\ba z (k_z ^b ) ^* \mbf{h} & = & \left( X^* k_z ^b + \mbf{b} b(z) ^* \right) ^* \mbf{h} \nn \\
&= & \left( X^* k_0 ^b k^b (0,0) ^{-1} (k_0 ^b) ^* k_z ^b + V^* k_z ^b \right) ^* \mbf{h} \nn \\
& = & \left( \mbf{b} b(0) ^* k^b (0,0) ^{-1} k ^b (0, z) \right) ^* \mbf{h}  =  0. \nn \ea
\end{proof}

Since $X$ is a contractive extension of the CCNC row partial isometry $V$, set
$$\Ga (z) := (I -Xz^*) ^{-1} \Ga (0) = k_z ^b D_{b(0)^*} ^{-1}. $$ This defines an analytic model triple, $(\Ga ,  \J \oplus \J ' , \K )$ for $V$, and
we now calculate the characteristic function of $V$ using this model triple: First note that by the previous Lemma,
\ba  z N^\Ga (z) & = &  z(\Ga (z) ^* \otimes I_d)  \left( \Ga _0 (\infty ) \oplus \Ga _1 (\infty ) \right) \nn \\
& = & z (\Ga (z) ^* \otimes I_d ) \Ga _0 (\infty ) + 0 _{\J ' , \K } =: z N_0 ^\Ga (z) + 0 _{\J ' , \K}, \nn \ea
so that $b^\Ga _V (z) = D^\Ga (z) ^{-1} z N^\Ga (z)$ coincides weakly with
\be b^{\ang{0}} (z) := D^\Ga (z) ^{-1} z N_0 ^\Ga (z) \in \scr{L} (\J, \K ). \label{fullFrostdef} \ee Even if $b$ is only purely contractive on the ball,
$D^\Ga (z)$ is a bounded, invertible operator for $z \in \B ^d$ by Lemma \ref{Dinvert}, and since $b^{\ang{0}} (0) =0$, $b^{\ang{0}} \in \scr{S} _d (\J , \K)$ is strictly contractive. Note that the construction of $b^{\ang{0}}$ is independent of the choice of extremal Gleason solution, $X$, for $\scr{H} (b)$.

In the case where $b(0)$ is a strict contraction, the formula for $b^{\ang{0}}$ can be computed more explicitly:
\ba D ^\Ga (z) & = & \Ga (z) ^* \Ga (0) \nn \\
& = & \Ga (0) ^* (I -zX^*) ^{-1} \Ga (0) \nn \\
& = & D_{b(0) ^*} ^{-1} (k_0 ^b) ^* (I -zX^*) ^{-1} k_0 ^b D_{b(0) ^*} ^{-1} \nn \\
& = & D_{b(0) ^*} ^{-1} (k_z ^b) ^* k_0 ^b D_{b(0) ^*} ^{-1} \nn \\
& = & D_{b(0) ^*} ^{-1} (I _{\J _0} - b(z) b(0) ^* ) D_{b(0) ^*} ^{-1}. \nn \ea
Similarly,
\ba z N _0 ^\Ga (z) & = & z (\Ga (z) ^* \otimes I_d) \Ga _0 (\infty) \nn \\
& = &  D_{b(0)^*} ^{-1} z (k_z ^b \otimes I_d ) ^*  \mbf{b} D_{b(0) } ^{-1}  \nn \\
& = & D_{b(0) ^* } ^{-1} ( b(z) - b(0) ) D_{b(0)} ^{-1}, \nn \ea and finally,
\ba b ^{\ang{0}} (z) & = & D^\Ga (z) ^{-1} z N _0 ^\Ga (z)  \nn \\
& = & D _{b(0) ^*  } \left( I _{\K} - b(z) b(0) ^* \right) ^{-1} \left( b(z) - b(0) \right) D _{b(0) } ^{-1}. \label{FSstrictform} \ea
In summary, we have proven:
\begin{cor} \label{FScor}
Let $b \in \scr{S} _d (\J , \K )$ be a strictly contractive Schur-class function and let $X$ be any extremal Gleason solution for $\scr{H} (b)$
with isometric-purely contractive decomposition $X = V -C$. Then the characteristic function $b_V$, of the partial isometric part, $V$, of $X$ coincides weakly with $b^{\ang{0}} \in \scr{S} _d (\J, \K)$.
\end{cor}

Let  $b^{\ang{0}} = b_V ^\Ga |_\J =: \Phi _{b(0) } \circ b$, where for any strict contractions $\alpha, \beta \in \scr{L} (\J , \K)$, we define
\be \Phi _{\alpha } (\beta ) := D _{\alpha ^* } \left( I _{\K} - \beta \alpha ^* \right) ^{-1} \left( \beta - \alpha \right) D _{\alpha } ^{-1} \in \scr{L} (\J , \K ). \ee As we will show in the following section, for any strict contraction $\alpha \in \scr{L} (\J, \K)$, the operator-M\"{o}bius transformation $\Phi _\alpha : [ \scr{L} (\J, \K) ] _1 \rightarrow [\scr{L} (\J, \K ) ] _1$ is an automorphism (\emph{i.e.} a bijection) of the unit ball of $\scr{L} (\J , \K )$,
and $\Phi _\alpha : \scr{S} _d (\J, \K) \rightarrow \scr{S} _d (\J , \K )$ maps the Schur class onto itself. Given any strict contraction $\alpha \in \scr{L} (\J , \K )$ and $b \in \scr{S } _d (\J , \K )$, $\Phi _\alpha \circ b$ is an operator-generalization of what is called a Frostman shift of the Schur-class function $b$ in the classical case where $b$ is a contractive analytic function on the disk \cite{Frostman,Frostman2,GR-model}. (This definition will be extended suitably to pure contractions.) The above example shows that if $X$ is any extremal Gleason solution for $\scr{H} (b)$ with partial isometric part $V$, then $b_V =b ^{\ang{0}} = \Phi _{b(0)} \circ b$ is the Frostman shift of $b$ vanishing at the origin.

\section{Commutative de Branges-Rovnyak model for CCNC row contractions} \label{CCNCmodelsect}

The goal of this section is to show that if $T = V -C :\mc{H} \otimes \C ^d \rightarrow \mc{H}$ is any CCNC row contraction with partial isometric part $V$ and purely contractive part $C$,
then $T$ is unitarily equivalent to an extremal Gleason solution $X$ for $\scr{H} (b)$, where the (purely contractive) Schur-class function $b = b_T$ is a certain Frostman shift of the characteristic function $b_V$ of $V$. As in the case where $T=V$ is a row partial isometry, the characteristic function $b_T$, of $T$, will be a unitary invariant for CCNC row contractions.

\subsection{Automorphisms of the unit ball of $\scr{L} (\mc{H} , \K)$}

Recall that $b \in \scr{S} _d (\mc{H} , \K)$ is said to be strictly contractive if $b(z)$ is a strict contraction for all $z \in \B ^d$ and purely contractive if $b(z)$ is a pure contraction for all $z \in \B ^d$, \emph{i.e.}, if $\| b(z) h \| < \| h \|$ for any $h \in \mc{H} $ and $z \in \B ^d$. Further recall that $b$ is strictly contractive if and only if $b(0)$ is a strict contraction so that, in particular, $b_V$ is strictly contractive for any CCNC row partial isometry $V$ (since $b_V (0) = 0$).

Consider the closed unit ball $\left[ \scr{L} (\mc{H} , \K ) \right] _1$, where $\mc{H}, \K$ are Hilbert spaces. Assume that $\alpha  \in [ \scr{L} (\mc{H} , \K) ]_1$ is a pure contraction
so that the defect operators $D_\alpha, D_{\alpha ^* } $ have closed, densely defined inverses. The $\alpha$-\emph{Frostman transformation}, $\Phi _\alpha$, and the inverse $\alpha-$Frostman transformation, $\Phi _\alpha ^{-1}$ are the maps defined on the open unit ball $\left( \scr{L} (\mc{H} , \K ) \right) _1$ by
\be \Phi _\alpha (\beta ) := D _{\alpha ^* } \left( I _\K - \beta \alpha ^* \right) ^{-1} \left( \beta -\alpha \right) D _{\alpha } ^{-1}, \ee and
\be \Phi _\alpha ^{-1} (\beta ) := \Phi _{-\alpha ^*} (\beta ^* ) ^* = D _{\alpha^*} ^{-1} \left( \beta + \alpha \right) \left( I_\mc{H} + \alpha ^* \beta \right) ^{-1} D_{\alpha}. \nn \ee To be precise, for any pure contraction $\alpha$, $\Phi _\alpha (\beta )$ is initially defined on the dense domain $\ran{D_{\alpha}}$, and $(\Phi _\alpha ^{-1} (\beta ) )^*$ is defined on the dense domain $\ran{D_{\alpha ^*}}$. The next lemma below shows that both of these densely defined transformations are contractions, and hence extend by continuity to the entire space.
If $\alpha$ is actually a strict contraction, we extend these definitions of $\Phi _\alpha , \Phi _{\alpha} ^{-1}$ to the closed unit ball.
\begin{lemma} \label{ballauto}
Let $\alpha \in [ \scr{L} (\mc{H} , \K) ] _1$.
\bi
\item[(i)] If $\alpha$ is a pure contraction, the Frostman transformations $\Phi _\alpha, \Phi _{\alpha} ^{-1}$ map $\left( \scr{L} (\mc{H} , \K) \right) _1$ into pure contractions in $[ \scr{L} (\mc{H} , \K ) ] _1$.

\item[(ii)] If $\alpha$ is a strict contraction, $\Phi _\alpha , \Phi _{\alpha} ^{-1}$ are compositional inverses and define bijections of $\left[ \scr{L} (\mc{H} , \K ) \right] _1$onto itself.

\item[(iii)] Given any two strict contractions $\beta , \ga$, $\Phi _\alpha$ and $\Phi _\alpha ^{-1} $ obey the identities:
$$ I - \Phi _\alpha (\beta ) \Phi _\alpha (\ga ) ^* = D_{\alpha ^*} (I - \beta \alpha ^* ) ^{-1} (I - \beta \ga ^* ) (I -\alpha \ga ^* ) ^{-1} D_{\alpha ^*}, \ \mbox{and},$$
$$ I - \Phi _\alpha ^{-1} (\beta ) \Phi _\alpha  ^{-1} (\ga ) ^* = D_\alpha ^* (I + \beta \alpha ^* ) ^{-1} (I - \beta \ga ^* ) (I + \alpha \ga ^*) ^{-1} D_{\alpha ^*}. $$
\ei
\end{lemma}
\begin{proof}
Assume that $\alpha $ is a pure contraction so that $D_\alpha := \sqrt{ I - \alpha ^* \alpha }$ has dense range. Given any $h = D_\alpha g \in \ran{D_\alpha }$, and any strict contraction $\beta$,
consider:
 $$ \| h \| ^2 - \| \Phi _\alpha (\beta ) h \| ^2  =  \ip{ D_\alpha ^2 g}{g} - \ip{ \underbrace{(\beta ^* -\alpha ^* ) (I _\K - \alpha \beta ^* ) ^{-1} D_{\alpha ^*} ^2 (I_\K - \beta \alpha ^* ) ^{-1} (\beta -\alpha )}_{\mbox{(A)}} g}{g}. $$
Observe that
$$ (I - \beta \alpha ^* )^{-1} (\beta - \alpha ) = \beta (I -\alpha ^* \beta ) ^{-1} D_\alpha ^2 - \alpha. $$ This identity is easily verified by multiplying both sides from the left by $(I - \beta \alpha ^* )$. Substitute this identity into the expression (A) to obtain:
\ba \mbox{(A)} & = & D _\alpha ^2 (I -\beta ^* \alpha ) ^{-1} \beta ^* D_{\alpha ^* } ^2 \beta (I -\alpha ^* \beta ) ^{-1} D_\alpha ^2  + \alpha ^* D_{\alpha ^*} ^2 \alpha \nn \\
& & - D_\alpha ^2 (I -\beta ^* \alpha) ^{-1} \beta ^* D_{\alpha ^* } ^2 \alpha - \alpha ^* D_{\alpha ^* } ^2 \beta (I -\alpha ^* \beta ) ^{-1} D_{\alpha } ^2 \nn \\
& = &  D _\alpha ^2 (I -\beta ^* \alpha ) ^{-1} \beta ^* D_{\alpha ^* } ^2 \beta (I -\alpha ^* \beta ) ^{-1} D_\alpha ^2 + D_{\alpha } ^2 \alpha ^* \alpha \nn \\ & & - D_\alpha ^2 (I -\beta ^* \alpha ) ^{-1} \beta ^* \alpha D_{\alpha } ^2 - D_{\alpha } ^2 \alpha ^* \beta (I -\alpha ^* \beta )^ {-1} D_{\alpha } ^2. \nn \ea
Applying the identity $D_\alpha ^2 - D_\alpha ^2 \alpha ^* \alpha = D_\alpha ^4$,
\ba D_\alpha ^2 - \mbox{(A)} & = & D_{\alpha} ^4 - D _\alpha ^2 (I -\beta ^* \alpha ) ^{-1} \beta ^* D_{\alpha ^* } ^2 \beta (I -\alpha ^* \beta ) ^{-1} D_\alpha ^2 \nn \\
& & + D_\alpha ^2 (I -\beta ^* \alpha ) ^{-1} \beta ^* \alpha D_{\alpha } ^2 + D_{\alpha } ^2 \alpha ^* \beta (I -\alpha ^* \beta ) ^{-1} D_{\alpha } ^2. \nn \\
& = & D_\alpha ^2 (I -\beta ^* \alpha ) ^{-1} \left[ (I -\beta ^* \alpha ) (I -\alpha ^* \beta ) - \beta ^* D_{\alpha ^*} ^2 \beta \right.\nn \\
& & + \left. \beta ^* \alpha (I - \alpha ^* \beta ) + (I - \beta ^* \alpha ) \alpha ^* \beta  \right] (I -\alpha ^* \beta ) ^{-1} D_{\alpha } ^2 \nn \\
& = & D_\alpha ^2 (I -\beta ^* \alpha ) ^{-1} (I - \beta ^* \beta ) (I -\alpha ^* \beta ) ^{-1} D_{\alpha } ^2. \nn \ea
This proves that for any $h \in \ran{D_\alpha}$,
$$ \| h \| ^2 - \| \Phi _\alpha (\beta ) h \| ^2 =  \ip{D_\alpha  (I -\beta ^* \alpha ) ^{-1} (I - \beta ^* \beta ) (I -\alpha ^* \beta ) ^{-1} D_{\alpha } h}{h} > 0, $$ and it follows that $\Phi _\alpha (\beta )$ is a pure contraction (and it will be a strict contraction if $\alpha$ is strict).  If $\alpha$ is a strict contraction then the above formulas make sense with $\beta$ only a pure contraction, and $\Phi _\alpha (\beta )$ will be a pure contraction in this case.

Assuming that $\alpha , \beta$ are both strict contractions, the expression $\Phi _\alpha ^{-1} ( \Phi _\alpha (\beta ) )$ is well-defined. It remains to calculate:
 $$ \Phi _{\alpha } ^{-1} ( \Phi _\alpha (\beta ) )  =  D_{\alpha ^* } ^{-1} \underbrace{\left( \Phi _\alpha (\beta ) + \alpha \right)}_{\mbox{(N)}}\cdot \underbrace{\left( I + \alpha ^* \Phi _\alpha (\beta ) \right) ^{-1} }_{\mbox{(D)} ^{-1} } D_\alpha. $$
The denominator, (D), evaluates to:
\ba \mbox{(D)} & = & D_\alpha (I - \alpha ^* \beta ) ^{-1} \left( I -\alpha ^* \beta +\alpha ^* \beta  -\alpha ^* \alpha \right) D_{\alpha } ^{-1} \nn \\
& = & D_\alpha (I -\alpha ^* \beta ) ^{-1} D_\alpha, \nn \ea while the numerator, (N), evaluates to
\ba \mbox{(N)} & = & D_{\alpha ^*} (I -\beta \alpha ^* ) ^{-1} \left( \beta -\alpha + ( I -\beta \alpha ^*) D_{\alpha ^*} ^{-1} \alpha D_\alpha \right) D_{\alpha } ^{-1} \nn \\
& = & D_{\alpha ^*} (I -\beta \alpha ^* ) ^{-1} \beta D_\alpha. \nn \ea
It follows that the full expression is
\ba \Phi _{\alpha } ^{-1} ( \Phi _\alpha (\beta ) ) & = & (I -\beta \alpha ^* ) ^{-1} \beta (I -\alpha ^* \beta ) = \beta. \nn \ea
The remaining assertions are similarly easy to verify.
\end{proof}

\subsection{Frostman shifts of Schur functions}

For any strictly contractive $b \in \scr{S} _d (\J , \K )$, it was proven in Subsection \ref{FSeg} (see Equation \ref{FSstrictform}) that
$$ b^{\ang{0}} (z) := D_{b(0) ^*}  (I _{\J _0} - b(z) b(0) ^* ) ^{-1}   ( b(z) - b(0) ) D_{b(0)} ^{-1} = (\Phi _{b(0)} \circ b) (z), $$
belongs to $\scr{S} _d (\J , \K)$ and vanishes at $0$. In particular, it follows that $b^{\ang{0}}$ is strictly contractive on the ball (since it vanishes at $0$).

\begin{defn}
Let $b \in \scr{S} _d (\J , \K )$ be a purely contractive Schur-class multiplier. The \emph{$0$-Frostman shift of $b$} is the strictly contractive Schur-class function $b^{\ang{0}} \in \scr{S} _d (\J , \K )$ of Equation (\ref{fullFrostdef}).

For any pure contraction $\alpha \in \scr{L} (\J , \K )$, the \emph{$\alpha$-Frostman shift of $b$} is the purely contractive (strictly contractive if $\alpha$ is a strict contraction) analytic operator-valued function $b^{\ang{\alpha}} := \Phi ^{-1} _{\alpha} \circ b^{\ang{0}}.$
\end{defn}
Theorem \ref{FShiftSchur} below will prove that the $\alpha$-Frostman shift $b^{\ang{\alpha}}$ of any $b \in \scr{S} _d (\J , \K )$ also belongs to the Schur class.
\begin{lemma} \label{pureFSpf}
For any purely contractive $b \in \scr{S} _d (\J , \K )$, $b = b^{\ang{b(0)}} = (b^{\ang{0}} ) ^{\ang{b(0)}} = \Phi ^{-1} _{b(0)} \circ b^{\ang{0}}.$
\end{lemma}
\begin{proof}
    If $b(0)$ is a strict contraction, this follows immediately from Lemma \ref{ballauto}. If $b(0)$ is only a pure contraction, then consider
$$ b^{\ang{b(0)}} = D_{b(0) ^*} ^{-1} \left( b^{\ang{0}} + b(0) \right) \left( I + b(0) ^* b^{\ang{0}} \right) ^{-1} D_{b(0)}. $$
Since $b^{\ang{0}}$ is Schur-class and strictly contractive, Lemma \ref{ballauto} implies that for any pure contraction $\alpha$, $b^{\ang{\alpha}} = \Phi _{\alpha} ^{-1} \circ b^{\ang{0}}$ is purely contractive. The formula above initially defines $b^{\ang{b(0)}} (z) ^*$ on a dense domain, and Lemma \ref{ballauto} shows that this extends by continuity to a pure contraction on the entire space.

Recall the notation of Subsection \ref{FSeg}, and equation (\ref{fullFrostdef}), which defines $b^{\ang{0}} (z):= D(z) ^{-1} z N_0 (z)$ with $D(z) = \Ga (z) ^* \Ga (0)$, $z N_0 (z) = z (\Ga (z) ^* \otimes I_d ) \Ga _0 (\infty )$, and $\Ga (z) , \Ga _0 (\infty)$ are the analytic model maps defined in Subsection \ref{FSeg}. Then consider
\ba D_{b(0) ^* } (D(z) + z N_0 (z)  b(0) ^* ) D_{b(0) ^*} & = & k^b (z,0) + z \left( k_z ^b \otimes I_d \right) ^* \Ga _0 (\infty ) D_{b(0)} b(0) ^* \nn \\
& = & k^b (z,0) + z \left( k_z ^b \otimes I_d \right) ^* \mbf{b} b(0) ^* \nn \\
& = & I - b(0) b(0) ^* = D_{b(0) ^*} ^2. \nn \ea
This proves that $D(z) + z N_0 (z) b(0) ^* = I$.
Next consider the denominator term,
\ba I + b (0) ^* b ^{\ang{0}} &= & I + b(0) ^* D(z) ^{-1} z N_0 (z) \nn \\
& = & I + b(0) ^* (I -z N_0 (z) b(0) ^* ) ^{-1} z N_0 (z) \nn \\
& = & I + (I - b(0) ^* z N_0 (z) ) ^{-1} - I. \nn \ea
The full expression is then

\ba & & D_{b(0) ^*} b^{\ang{b(0)}} (z) =  (D(z) ^{-1} z N_0 (z) + b(0) ) ( I - b(0) ^* z N_0 (z) ) D_{b(0)} \nn \\
& = & \left( D(z) ^{-1} z N_0 (z) - D(z) ^{-1} z N_0 (z) b(0) ^* z N_0 (z) + b(0) - b(0) b(0) ^* z N_0 (z) \right) D_{b(0)} \nn \\
& = & \left( D(z) ^{-1} z N_0 (z) - D(z)  ^{-1} (I - D(z) ) z N_0 (z) + b(0) - b(0) b(0) ^* z N_0 (z) \right) D_{b(0)} \nn \\
& = & b(0)D_{b(0)}  + D_{b(0) ^* } ^2 z N_0 (z) D_{b(0)} \nn \\
& = & D_{b(0) ^*} b(0) + D_{b(0) ^*}  z \mbf{b} (z) \nn \\
&= & D_{b(0) ^*} b(z). \nn \ea

\end{proof}

\begin{thm} \label{FShiftSchur}
    Given any purely contractive Schur function $b \in \scr{S} _d (\mc{H} , \K )$, and any pure contraction $\alpha \in \scr{L} (\mc{H} , \K )$, the $\alpha$-Frostman shift, $b^{\ang{\alpha}} \in \scr{S} _d (\mc{H} , \K)$ is also Schur-class and purely contractive,
$$ b^{\ang{\alpha }} (z) = D_{\alpha ^* } ^{-1} \left(b ^{\ang{0}} (z) + \alpha \right)\left(I + \alpha ^* b^{\ang{0}} (z) \right) ^{-1} D_\alpha. $$

Multiplication by
$$ M ^{\ang{\alpha}} (z) := D_{\alpha ^*} (I + b^{\ang{0}} (z) \alpha ^* ) ^{-1} = (I -b^{\ang{\alpha}} (z) \alpha ^* ) D_{\alpha ^*} ^{-1}, $$ defines
a unitary multipier, $M^{\ang{\alpha}}$, of $\scr{H} (b^{\ang{0}})$ onto $\scr{H} (b^{\ang{\alpha}})$.
\end{thm}
The unitary multiplier $M ^{\ang{\alpha}}$ is a multivariable and operator analogue of a Crofoot multiplier or \emph{Crofoot transform} \cite{GR-model,Crofoot}. Since $M^{\ang{\alpha}}$ is a unitary multiplier, it follows, in particular, that $M^{\ang{\alpha} } (z)$ defines a bounded invertible operator for any $z \in \B ^d$.
\begin{proof}
By the identities of Lemma \ref{ballauto}, 
\ba & &  I - b^{\ang{\alpha}} (z) b^{\ang{\alpha}} (w) ^*  =  I - \Phi _{\alpha } ^{-1} ( b^{\ang{0}} (z) ) \Phi _{\alpha } ^{-1} (b^{\ang{0}} (w) ) ^* \nn \\
& = & D_{\alpha ^*} (I + b^{\ang{0}} (z) \alpha ^*) ^{-1} \left( I -b^{\ang{0}} (z) b^{\ang{0}} (w) ^* \right) (I  + \alpha b^{\ang{0}} (w) ^* ) ^{-1} D_{\alpha ^*}. \nn \ea  Similarly, suppose $\alpha$ is a strict contraction (the argument can be modified to prove the case where $\alpha$ is pure), then
$$ I - b^{\ang{0}} (z) b^{\ang{0}} (w) ^* = D_{\alpha ^* } (I - b^{\ang{\alpha}} (z) ) \alpha ^* ) ^{-1} \left( I - b^{\ang{\alpha}} (z) b^{\ang{\alpha}} (w) ^* \right) (I - \alpha b^{\ang{\alpha}} (w) ^* ) ^{-1} D_{\alpha ^*}, $$
This proves simultaneously that $k^{\ang{\alpha}}:= k^{b^{\ang{\alpha}}}$ is a positive kernel so that $b ^{\ang{\alpha}} \in \scr{S} _d (\mc{H} , \K))$ by \cite[Theorem 2.1]{Ball2001-lift}, and that $M ^{\ang{\alpha}} (z)$ as written above
is a unitary multiplier of $\scr{H} (b^{\ang{0}} )$ onto $\scr{H} (b^{\ang{\alpha}})$.
\end{proof}

\begin{thm} \label{FSGSbij}
The map,
$$ \mbf{b} ^{\ang{\alpha}} \mapsto \mbf{b}  ^{\ang{0}} := (M ^{\ang{\alpha}} \otimes I_d ) ^{-1} \mbf{b} ^{\ang{\alpha}} D_{\alpha} ^{-1}, $$ is a bijection from contractive Gleason solutions for $b ^{\ang{\alpha}}$
onto contractive Gleason solutions for $b ^{\ang{0}}$ which preserves extremal solutions.
\end{thm}

\begin{proof}
First, assume that $\alpha$ is a strict contraction. Let $\mbf{b} ^{\ang{\alpha}}$ be any contractive Gleason solution for $b ^{\ang{\alpha}}$. Then,
\ba z \mbf{b} ^{\ang{0}} (z) & = & z ( M^{\ang{\alpha}} (z) ^{-1} \otimes I_d ) \mbf{b} ^{\ang{\alpha}} (z) D_{\alpha } ^{-1} \nn \\
& = & M ^{\ang{\alpha}} (z) ^{-1}(b ^{\ang{\alpha}}(z) - \alpha ) D_{\alpha } ^{-1} \nn \\
& = & D_{\alpha ^*} (I -b ^{\ang{\alpha}} (z) \alpha ^* ) ^{-1} (b ^{\ang{\alpha}} (z) - \alpha ) D_{\alpha } ^{-1} \nn \\
& = & b ^{\ang{0}} (z), \nn \ea and this equals $b^{\ang{0}} (z) - b ^{\ang{0}} (0)$ since $b^{\ang{0}} (0) = 0$. This shows
that $\mbf{b} ^{\ang{0}}$ is a Gleason solution, and
\ba (\mbf{b} ^{\ang{0}})  ^* \mbf{b} ^{\ang{0}} & = & D _{\alpha } ^{-1} (\mbf{b} ^{\ang{\alpha}} ) ^* \mbf{b} ^{\ang{\alpha}} D_{\alpha} ^{-1} \nn \\
& \leq & D_{\alpha} ^{-1} (I -\alpha ^* \alpha ) D_{\alpha} ^{-1} = I, \nn \ea so that $\mbf{b} ^{\ang{0}}$ is a contractive Gleason solution which is extremal if $\mbf{b} ^{\ang{\alpha}}$ is.
The converse follows similarly.

In the case where $\alpha$ is not a strict contraction, set $b:=b^{\ang{\alpha}}$, so that $\alpha = b(0)$. Recall the notation of Subsection \ref{FSeg}, and equation (\ref{fullFrostdef}), which defines $b^{\ang{0}} (z):= D(z) ^{-1} z N_0 (z)$ with $D(z) = \Ga (z) ^* \Ga (0)$, $z N_0 (z) = z (\Ga (z) ^* \otimes I_d ) \Ga _0 (\infty )$, and $\Ga (z) , \Ga _0 (\infty)$ are the analytic model maps defined in Subsection \ref{FSeg}. Namely, $\Ga (z) D_{\alpha ^*} =k_0 ^b$, and $\Ga _0 (\infty ) D_\alpha = \mbf{b}$. As in the proof of Lemma \ref{pureFSpf}, it follows that $D(z) + zN_0 (z) \alpha ^* =I$, so that 
\ba M^{\ang{\alpha}} (z) & = & D_{\alpha ^*} (I - b^{\ang{0}} (z) \alpha ^* ) ^{-1} \nn \\
& = & D_{\alpha ^*} D(z) (D(z) + z N_0 (z) \alpha ^* ) ^{-1} = D_{\alpha ^*} D(z). \nn \ea
To prove that $\mbf{b} ^{\ang{0}}$ is a contractive Gleason solution for $b^{\ang{0}}$, first observe that
\ba D_{\alpha ^*} z N_0 (z) D_\alpha & =& D_{\alpha ^*} \Ga (z) ^* z \Ga _0 (\infty) D_\alpha \nn \\
& = & (z^* k_z ^b ) ^* \mbf{b} = z \mbf{b} (z) \nn \\
& = & b(z) - b(0) = b(z) - \alpha. \nn \ea It then follows that
\ba & & M^{\ang{\alpha}} (z) z \mbf{b} ^{\ang{0}} (z) D_\alpha  =  z \mbf{b} (z) \nn \\ &\Rightarrow & \quad D_{\alpha ^*} D(z) z \mbf{b} ^{\ang{0}} (z) D_\alpha  = D_{\alpha ^* } z N_0 (z) D_\alpha \nn \\
&\Leftrightarrow & \quad  D(z) z \mbf{b} ^{\ang{0}} (z) = z N_0 (z) \nn \\
&\Leftrightarrow& \quad z \mbf{b} ^{\ang{0}} (z) = b^{\ang{0}} (z). \nn \ea 
This proves that $\mbf{b} ^{\ang{0}}$ is a Gleason solution for $b^{\ang{0}}$ and contractivity follows as before. Again, the fact that $\mbf{b} \mapsto \mbf{b} ^{\ang{0}}$ is surjective follows similarly. 
\end{proof}

\begin{prop} \label{FSGSforHb}
    Let $b \in \scr{S} _d (\mc{H} , \K)$ be a purely contractive Schur-class function, let $X =X (\mbf{b})$ be an extremal Gleason solution for $\scr{H} (b)$, and
let $\mbf{b}$ be the corresponding contractive extremal Gleason solution for $b$: $X^* k_w ^b = w^* k_w ^b - \mbf{b} b(w) ^*.$

If $\mbf{b} ^{\ang{0}} = \left( M ^{\ang{b(0)}} \otimes I_d \right) ^* \mbf{b} D_{b(0)} ^{-1}$, $X^{\ang{0}} := X ( \mbf{b} ^{\ang{0}} )$ are the corresponding
extremal Gleason solutions for $b^{\ang{0}}$ and $\scr{H} (b^{\ang{0}})$, and $X = V -C$ is the isometric-pure decomposition of $X$, then
$$ V = M ^{\ang{b(0)}} X^{\ang{0}} (M^{\ang{b(0)}} \otimes I_d ) ^*. $$
\end{prop}

\begin{proof}
 As in Subsection \ref{FSeg}, we have that $V^* = X^* P _{\ran{X}}$, where
$$ P _{\ran{X} } = I- k_0 ^b k^b (0,0)^{-1} (k_0 ^b )^* = I - k_0 ^b D_{b(0) ^* } ^{-2} (k_0 ^b) ^*. $$
Let $k^{\ang{0}} = k^{b^{\ang{0}}}$, we need to verify that $ M ^{\ang{b(0)}} X^{\ang{0}} (M^{\ang{b(0)}} \otimes I_d ) ^* = V$.
To prove this, check the action on point evaluation maps:
\ba V^* k_w ^b & = & w^* k_w ^b - \mbf{b} b(w) ^* + \mbf{b} b(0) ^* D _{b(0) ^* } ^{-2} (I -b(0) b(w) ^* ) \nn \\
& = & w^* k_w ^b - \mbf{b} D _{b(0)} ^{-2} \left( (I - b(0) ^* b(0) ) b(w) ^* - b(0) ^* (I - b(0) b(w) ^*) \right) \nn \\
& = & w^* k_w ^b -\mbf{b} D_{b(0)} ^{-2} (b(w) ^* - b(0) ^* ). \nn \ea
Compare this to: 
\ba (M ^{\ang{b(0)}} \otimes I_d ) (X^{\ang{0}})^* (M^{\ang{b(0)}}) ^* k_w ^b & = & (M ^{\ang{b(0)}} \otimes I_d ) (X ^{\ang{0}})^* k_w ^{b^{\ang{0}}} M^{\ang{b(0)}} (w) ^* \nn \\
& = & (M^{\ang{b(0)}} \otimes I_d ) \left( w^* k_w ^{\ang{0}} - \mbf{b} ^{\ang{0}} b^{\ang{0}} (w) ^* \right) M^{\ang{b(0)}} (w) ^* \nn \\
& = & w^* k_w ^b - \mbf{b} D _{b(0)} ^{-1} b^{\ang{0}} (w) ^* M^{\ang{b(0)}} (w) ^* \nn \\
& = & w^* k_w ^b -\mbf{b} D _{b(0)} ^{-2} (b(w) ^* - b(0) ^* ). \nn \ea

\end{proof}

\subsection{Gleason solution model for CCNC row contractions}

Let $T$ be an arbitrary CCNC row contraction on $\mc{H}$ with partial isometric-purely contractive decomposition $T = V -C$.

\begin{lemma} \label{zeropoint}
    Let $V$ be a CCNC row partial isometry on $\mc{H}$, and let $(\ga , \J _\infty , \J _0 )$ be a model triple for $V$. Given any pure
contraction $\delta  \in [ \scr{L} (\J _\infty , \J _0) ] _1$, the map
$$ \delta  \mapsto T_\delta := V - \ga (0) \delta \ga (\infty ) ^*, $$ is a bijection from pure contractions onto CCNC row contractions with partial isometric part $V$.
\end{lemma}

\begin{proof}
    Since $\ga (0) : \J _0 \rightarrow \ran{V} ^\perp$ and $\ga (\infty) : \J _\infty \rightarrow \ker{V}$ are onto isometries, it is clear that $\delta \mapsto -\ga (0) \delta \ga (\infty ) ^*$ maps
pure contractions $\delta \in \scr{L} ( \J _\infty , \J _0 )$ onto all pure contractions in $\scr{L} ( \ker{V}  , \ran{V} ^\perp )$. It is also clear that $T_\delta$ is CCNC if and only if $V$ is, and that $V$ is the partial isometric
part of $T_\delta$. Conversely, given any $CCNC$ row contraction $T$ on $\mc{H}$ such that $T = V-C$, we have that $\delta := - \ga (0) ^* T \ga (\infty) =  \ga (0) ^* C \ga (\infty)$. Since $C$ is a pure row contraction,
$\delta$ is a pure contraction and $T = T_\delta$.
\end{proof}

\begin{defn} \label{fullchardef}
 Let $T :\mc{H} \otimes \C ^d \rightarrow \mc{H}$ be a CCNC row contraction with partial isometeric-purely contractive decomposition $T = V-C$. For any fixed model triple $(\ga , \J _\infty , \J _0 )$ of $V$, define
$$ \delta _T ^\ga := - \ga (0) ^* T \ga (\infty ) = \ga (0) ^* C \ga (\infty ), $$ the \emph{zero-point contraction} of $T$.
The \emph{characteristic function}, $b_T$, of $T$, is then any Schur-class function in the weak coincidence class of the $\delta _T ^\ga$-Frostman shift of $b_V ^\ga$,
 \ba b_T ^\ga  &:=&  (b_V ^\ga ) ^{\ang{\delta _T ^\ga}} \in \scr{S} _d (\J _\infty , \J _0 ), \nn \\
 b_T ^\ga (z) & = & D_{(\delta _T ^\ga) ^*} ^{-1} \left( b_V ^\ga (z) + \delta _T ^\ga \right) \left( I + (\delta _T ^\ga )^* b_V ^\ga (z) \right) ^{-1} D_{\delta _T ^\ga}; \quad \quad z \in \B ^d. \nn \ea Since $C$ is a pure row contraction, it follows that $\delta _T ^\ga$ is always a pure contraction, and that $b_T ^\ga$ is purely contractive on the ball, by Lemma \ref{ballauto}.
\end{defn}

\begin{lemma}
    The coincidence class of $b_T ^\ga $ is invariant under the choice of model triple $(\ga , \J_\infty , \J _0)$.
\end{lemma}

\begin{proof}
Let $(\ga , \J _\infty , \J _0)$ and $( \varphi , \K _\infty , \K _0)$ be two model triples for $V$, $T = V-C$. By Lemma \ref{modmapinv} we know that there are onto isometries $R : \K _\infty \rightarrow \J _\infty$
and $Q^*  : \K _0 \rightarrow \J _0$ so that $\varphi (\infty ) = \ga (\infty ) R$, $\varphi (0) = \ga (0) Q^*$, and
$$ b_V ^\varphi = Q b_V ^\ga R. $$ Similarly,
$$ \delta _T ^\varphi = - \varphi (0) ^* T \varphi (\infty ) = Q \delta _T ^\ga R. $$ Finally, we calculate
\ba b_T ^\varphi & = & (Q b_V ^\ga R) ^{\ang{Q \delta _T ^\ga R} } \nn \\
& = & D^{-1} _{ (Q\delta _T ^\ga R) ^* } \left( Q (b_V ^\ga + \delta _T ^\ga ) R \right) \left( I + R^* (\delta _T ^\ga ) ^* Q ^*  Q b_V ^\ga R \right) ^{-1} D _{Q \delta _T ^\ga R } \nn \\
& = & Q D_{(\delta _T ^\ga ) ^* } ^{-1} Q^* Q \left( b_V ^\ga + \delta _T ^\ga \right) R R^* \left( I + (\delta _T ^\ga) ^* b_V ^\ga \right) ^{-1} R R^* D_{\delta _T ^\ga } R \nn \\
& = & Q (b_V ^\ga ) ^{\ang{\delta _T ^\ga}} R = Q b_T ^\ga R. \nn \ea
\end{proof}

\begin{lemma}{ (Extremal Gleason solutions)} \label{extGScharfun}
Let $b \in \scr{S} _d (\mc{H} , \K )$ be a Schur-class multiplier. If $X$ is any extremal Gleason solution for $\scr{H} (b)$, then
the characteristic function, $b_X$, of $X$, coincides weakly with $b$.
\end{lemma}

\begin{proof}
Given any purely contractive $b \in \scr{S} _d (\mc{H} , \K)$, Lemma \ref{GSCCNC} proved that any contractive extremal Gleason solution, $X$, for $\scr{H} (b)$
is a CCNC row contraction. Using the model constructed as in Subsection \ref{FSeg}, we will now show that the characteristic function, $b_X$, of $X$, coincides weakly with $b$.

As discussed in Subsection \ref{Gleasonsub}, $X = X (\mbf{b})$ for an extremal Gleason solution $\mbf{b}$ for $b$, where $X (\mbf{b})$ is given by Formula (\ref{GSHbb}). Let $\mbf{b} ^{\ang{0}} := (M^{\ang{b(0)}} \otimes I_d ) ^* \mbf{b} D_{b(0)} ^{-1}$ be the extremal Gleason solution for $b ^{\ang{0}}$ which corresponds uniquely to $\mbf{b}$ as in Theorem \ref{FSGSbij}, and let $X=V-C$ be the isometric-pure decomposition of $X$. Proposition \ref{FSGSforHb} then implies that if $X ^{\ang{0}} := X (\mbf{b} ^{\ang{0}} )$ is the corresponding extremal Gleason solution for $\scr{H} (b^{\ang{0}})$, that $V = M^{\ang{b(0)}} X ^{\ang{0}} (M ^{\ang{b(0)}} \otimes I_d ) ^*$.

As in Subsection \ref{FSeg}, it then follows that we can define an analytic model triple for $V$ as follows. Let $\Ga (0) := k_0 ^b k^b (0,0) ^{-1/2} = k_0 ^b D_{b(0) ^*} ^{-1}$, $\Ga (z) := (I -X z^* ) ^{-1} \Ga (0) = k_z ^b D_{b(0) ^*} ^{-1}$, and $\Ga (\infty ) := \mbf{b} D_{b(0)} ^{-1} \oplus \Ga (\infty ) '$, where
$\Ga (\infty ) ' : \K ' \rightarrow \ker{V} \ominus \ran{\mbf{b}}$ is an arbitrary onto isometry (that $\mbf{b}$ maps into $\ker{V}$ follows as in Lemma \ref{kerViso}). As in Subsection \ref{FSeg}, $(\Ga , \K \oplus \K ' , \K)$ is an analytic model triple for $V$, and we will use this triple to compute the characteristic function $b_X ^\Ga$. Using the relationship between $V$ and $X^{\ang{0}}$, it is easy to check that $b_V ^\Ga = b ^{\ang{0}} \oplus 0 _{\K ' }$ as in Subsection \ref{FSeg}, and it remains
to check that $\delta _X ^\Ga = b(0) \oplus 0 _{\K '}$:
\ba \delta _X ^\Ga & = & - (k_0 ^b D_{b(0)^*} ^{-1} ) ^* X \Ga (\infty ) \nn \\
& = & - (X^* k_0 ^b D_{b(0) ^* } ^{-1} ) ^* \mbf{b} D_{b(0)} ^{-1} \oplus \Ga (\infty) ' \nn \\
& = & D_{b(0) ^* } ^{-1} b(0) \mbf{b} ^*  \left(  \mbf{b} D_{b(0)} ^{-1} \oplus \Ga (\infty) ' \right) \nn \\
& = & b(0) \oplus 0 _{\K ' }. \nn \ea We conclude that $b_X ^\Ga := (b_V ^\Ga ) ^{\ang{\delta _X ^\Ga }}$ coincides weakly with $b = (b^{\ang{0}}) ^{\ang{b(0)}}$.
\end{proof}

We are now sufficiently prepared to prove one of our main results:
\begin{thm} \label{main1}
    A row contraction $T : \mc{H} \otimes \C ^d \rightarrow \mc{H}$ is CCNC if and only if $T$ is unitarily equivalent to an extremal Gleason solution $X^b$ acting on a multi-variable de Branges-Rovnyak space $\scr{H} (b)$
for a purely contractive Schur-class $b \in \scr{S} _d (\J , \K )$.

If $T \simeq X^b$, the characteristic function, $b_T :=b$ is a unitary invariant: if two CCNC row contractions $T_1, T_2$ are unitarily equivalent, then their characteristic functions $b_{T_1}, b_{T_2}$ coincide weakly. One can choose $b_T \in \scr{S} _d (\ran{D_T} , \ran{D_{T^*}} )$.
\end{thm}
In the above $\simeq$ denotes unitary equivalence. Recall that as shown in Subsection \ref{Gleg}, the characteristic function of a CCNC row contraction is not a complete unitary invariant: there exist CCNC row contractions $T_1, T_2$ which have the same characteristic function but are not unitarily equivalent.
\begin{proof}
We have already proven that any extremal Gleason solution, $X$, for $\scr{H} (b)$, $b \in \scr{S} _d (\J , \K)$ is a CCNC row contraction with characteristic function coinciding weakly with $b$ in Lemma \ref{GSCCNC} and Lemma \ref{extGScharfun} above. Conversely, let $T$ be a CCNC row contraction on $\mc{H}$ with isometric-contractive decomposition $T=V-C$.

Let $(\Ga , \J _\infty , \J _0)$ be the analytic model triple
$\Ga = \Ga _V$ for $V$. Recall that $\J _\infty \simeq \ker{V} \simeq \ran{D_T}$ and $\J _0 \simeq \ran{V} ^\perp \simeq \ran{D_{T^*}}$ ($\simeq$ means isomorphic as Hilbert spaces, \emph{i.e.} they have the same dimension). By Theorem \ref{Glerep2}, the unitary $U  ^\Ga := (M^\Ga ) ^{-1}  \hat{U } ^\Ga : \mc{H} \rightarrow \scr{H} (b_V ^\Ga )$, where recall $M^\Ga = M_{D^\Ga } $, a unitary multiplier, is such that $X := U^\Ga V (U^\Ga)^*  = X (\mbf{b} ^\Ga )$ is an extremal Gleason solution for $\scr{H} (b_V ^\Ga )$ corresponding to
the extremal Gleason solution $\mbf{b} ^\Ga = U ^\Ga \Ga (\infty )$ for $b_V ^\Ga$. It is clear that $(U ^\Ga \Ga , \J _\infty , \J _0 )$ is then an analytic model triple for $X$.

By Theorem \ref{FSGSbij}, since $b_V ^\Ga = \Phi _{b_T ^\Ga (0)} (b_T ^\Ga ) = (b_T ^\Ga) ^{\ang{0}} $ is the $0$-Frostman shift of $b_T$, there is a unique extremal Gleason solution $\mbf{b} ^T$ for $b_T = b_T ^\Ga$ so that
$$ \mbf{b} ^\Ga := (M ^{\ang{b_T(0)}} \otimes I ) ^{-1} \mbf{b} ^T D_{b _T(0)} ^{-1}, $$ where, by definition, $b_T (0) = b_T ^\Ga (0) = \delta _T ^\Ga$. Proposition \ref{FSGSforHb} then implies that if $X ^T := X (\mbf{b} ^T )$ is the corresponding extremal Gleason solution for $b_T$ with isometric-pure decomposition $X^T = V^T -C^T$, then $\wt{X} := (M^{\ang{b_T (0)}}) ^* X^T (M^{\ang{b_T (0)}} \otimes I  _d)  = X - \wt{C}$ where
$(M^{\ang{b_T (0)}}) ^* V^T (M^{\ang{b_T (0)}} \otimes I_d) =X$.

Our goal is to prove that $T' = U ^\Ga T (U ^\Ga) ^* \otimes I_d =: X - C'$ is equal to $\wt{X} = (M^{\ang{b(0)}}) ^* X^T M^{\ang{b_T (0)}} \otimes I  _d  = X - \wt{C}$ so that
$T \simeq X^T$, \emph{i.e.} $T$ is unitarily equivalent to the extremal Gleason solution $X^T$ for $\scr{H} (b_T ^\Ga )$. By Lemma \ref{zeropoint}, it suffices to show that $\delta _{T'} ^{U^\Ga \Ga} = \delta _{\wt{X}} ^{U^\Ga \Ga}$.
Let $b := b_T ^\Ga$, $b_V := b_V ^\Ga$, and calculate
\ba \delta _{T'} ^{U^\Ga \Ga} & = & - \Ga (0) ^* (U^{\Ga }) ^*  T' (U ^\Ga \otimes I_d ) \Ga (\infty )  \nn \\
& = & -\Ga (0) ^* T \Ga (\infty ) = \delta _T ^\Ga = b_T ^\Ga (0). \nn \ea
Similarly, since $X = X (\mbf{b} ^\Ga )$ where $\mbf{b} ^\Ga = U ^\Ga \Ga (\infty)$,
\ba \delta _{\wt{X} } ^{U ^\Ga \Ga } & = & - (U^\Ga \Ga (0) ) ^* \wt{X} \mbf{b} ^\Ga \nn \\
& = & - ( M _{D ^\Ga } ^* \hat{K} _0 ^\Ga ) ^* (M^{\ang{b(0)}}) ^* X ^T (M^{\ang{b_T (0)}} \otimes I_d )  (M^{\ang{b _T (0)}} \otimes I ) ^{-1} \mbf{b} ^T D_{b _T(0)} ^{-1} \nn \\
& = & - \left((M^{\ang{b _T (0) }}) ^{-1} ) ^* k_0 ^{b_V} \right) ^*  X ^T \mbf{b} ^T D _{b_T (0) } ^{-1}  \nn \\
& = & - \left( (X^T) ^* k_0 ^{b_T} (M^{\ang{b_T (0)}} (0) ^{-1}) ^* \right) ^*   \mbf{b} ^T D _{b_T (0) } ^{-1}  \nn \\
& = &  M^{\ang{b_T (0)}} (0) ^{-1} b(0) (\mbf{b} ^T) ^* \mbf{b} ^T  D _{b_T (0) } ^{-1} \nn \\
& = & D _{b_T (0) ^*}  (I -b(0) b(0) ^* ) ^{-1} b(0) (I - b(0) ^* b(0) ) D _{b_T (0) } ^{-1} \nn \\
& = & b_T (0) = \delta _T ^\Ga. \nn \ea
Finally, if $T_1, T_2$ are two unitarily equivalent row contractions, $T_2 = U T_1 U^*$ for a unitary $U:\mc{H} _1 \rightarrow \mc{H} _2$, and $T_k = V_k -C _k$, then given any analytic model triple $(\Ga , \J _\infty , \J _0)$ for $T_1$, $(U \Ga , \J _\infty , \J _0 )$ is an analytic model triple for $T_2$, and $b_{T_1} ^\Ga = b_{T_2} ^{U\Ga}$.
\end{proof}

We conclude this section with the observation that our characteristic function, $b_T$, for any CCNC row contraction, $T$, coincides weakly with the Sz.-Nagy-Foia\c{s}-type characteristic function  of $T$, as defined for CNC $d-$contractions in \cite{BES2006cnc}:
$$ \Theta _T (z) := \left( -T + z D _{T^*} (I -z T^* ) ^{-1} D_T \right) | _{\ran{D_T}} \in \scr{L} (\ran{D_T} , \ran{D_{T^*}});  \quad z \in \B ^d. $$

\begin{prop} \label{NFprop}
    Let $T$ be a CCNC row contraction.  The characteristic function, $b_T$, of $T$, coincides weakly with the Nagy-Foia\c{s} characteristic function,
$\Theta _T$.
\end{prop}

\begin{proof}
It suffices to show that given any purely contractive $b \in \scr{S} _d (\J , \K)$, and any contractive, extremal Gleason solution, $X$, for $\scr{H} (b)$
that $b_X =: b$ coincides weakly with $\Theta _X$. Let $\mbf{b}$ be a contractive extremal Gleason solution for $b$ so that $X = X (\mbf{b} )$.

As in the proof of Lemma \ref{extGScharfun}, consider the analytic model triple $(\Ga , \K \oplus \K ' , \K )$ where
$\Ga (0) = k_0 ^b D_{b(0) ^*} ^{-1}$, $\Ga (z) := (I- X z^* ) ^{-1} \Ga (0)$, and $\Ga (\infty ) = \mbf{b} D_{b(0)} ^{-1} \oplus \Ga (\infty ) '$, where if
$X = V-C$ is the isometric-pure decomposition of $X$, then $\Ga (\infty ) ' : \K ' \rightarrow \ker{V} \ominus \ran{\mbf{b}}$ is an onto isometry. Since $\Ga (0), \Ga (\infty )$ are onto isometries, $\Theta _X$ coincides with
$$ -\Ga (0) ^* X \Ga (\infty ) + \Ga (0) ^* z D_{X^*} (I -zX^* ) ^ {-1} D_X \Ga (\infty), $$ and $\delta _X ^\Ga = - \Ga (0) ^* X \Ga (\infty ) = b(0) \oplus 0 _{\K  '}$ as in the proof of Lemma \ref{extGScharfun}. Since $X$ is extremal, one can verify (by uniqueness of the positive square root) that $$ D_{X^*}  = \sqrt{k_0 ^b (k_0 ^b)^*} = k_0 ^b D_{b(0) ^*} ^{-1} (k_0 ^b )^*, $$ and it follows that
 \ba \Ga (0) ^* z D_{X^*} (I -zX^* ) ^ {-1} D_X \Ga (\infty) & = & D_{b(0) ^*} ^{-1} (k_0 ^b ) ^* k_0 ^b D_{b(0) ^*} ^{-1} (z^* k_z ^b )^* \sqrt{I-X^*X} \Ga (\infty) \nn \\
 & = & (z^* k_z ^b ) ^* \sqrt{I -X^* X} \Ga (\infty ). \nn \ea Since $\Ga (\infty )$ is an isometry onto $\ker{V}$, and $X^*X = V^*V + C ^* C$, it follows that
$$ \sqrt{I -X^* X } \Ga (\infty) = \sqrt{P _{\ker{V}} - P_{\ker{V}} C^*C P_{\ker{V}} } \Ga (\infty).$$ Moreover, in Subsection \ref{FSeg}, we calculated that
$$ V^* = X^* (I - k_0 ^b D_{b(0) ^*} ^{-1} (k_0 ^b) ^* ), $$ and it follows from this that
$$ C^* = \mbf{b} D_{b(0)} ^{-1} b(0) ^* D_{b(0) ^*} ^{-1} (k_0 ^b ) ^*.$$ In particular, it follows that
$$ (C^* C ) ^k = \mbf{b} D_{b(0)} ^{-1} \left( b(0) ^* b(0)  \right) ^k D_{b(0)} ^{-1} \mbf{b} ^*; \quad \quad k \in \N, $$ and the functional calculus then implies that
$$   \sqrt{I -X^* X } \Ga (\infty) = \mbf{b} D_{b(0)} ^{-1} \sqrt{ I - b(0) ^* b(0) } D_{b(0)} ^{-1} \mbf{b} ^* = \mbf{b} D_{b(0)} ^{-1} \mbf{b} ^*. $$
In conclusion we obtain
\ba \Theta _X (z) & \simeq & b(0) \oplus 0 _{\K '} + (z ^* k_z ^b ) ^* \mbf{b} D_{b(0)} ^{-1} \mbf{b} ^* \mbf{b} D_{b(0)} ^{-1} \oplus 0 _{\K '} \nn \\
& = & b(0) \oplus 0 _{\K ' } + (b(z) - b(0)) \oplus 0 _{\K '} = b(z) \oplus 0 _{\K '}, \nn \ea and $\Theta _X$ coincides weakly with $b = b_X$.
\end{proof}

\section{QE row contractions} \label{QErowsect}

In this section we focus on the sub-class of quasi-extreme (QE) row contractions. This is the set of all CCNC row contractions, $T$, whose characteristic function $b_T$ coincides weakly with a quasi-extreme Schur multiplier as defined and studied in \cite{Jur2014AC,JM,JMqe}. We will see that the characteristic function is a complete unitary invariant for QE row contractions.

\subsection{Quasi-extreme Schur multipliers} \label{QEsect}

As discussed in the introduction, the concept of a \emph{quasi-extreme} Schur-class multiplier was introduced in \cite{Jur2014AC,JM}, as a several-variable
analogue of a `Szeg\"{o} approximation property' that is equivalent to being an extreme point of the Schur class in the single-variable, scalar-valued setting (see \emph{e.g.} \cite{JMqe}).

In \cite{JM}, the quasi-extreme property was defined for any non-unital and square $b \in \scr{S} _d (\mc{H})$ (recall from Subsection \ref{HerglotzSection} that the non-unital assumption is needed to ensure that the corresponding Herglotz-Schur function $H_b$ takes values in bounded operators), but we will require the extension of this property to arbitrary purely contractive and `rectangular' $b \in \scr{S} _d (\J , \K)$. For this purpose it will be useful to consider the square extension, $[b]$, of any $b \in \scr{S} _d (\J , \K)$, as defined in Subsection \ref{Gleasonsub}.

\begin{thm} \label{equivqe}
Given any $b \in \scr{S} _d (\J , \K )$ such that $[b]$ is non-unital, the following are equivalent:
\bn
    \item[(i)] $b$ has a unique contractive Gleason solution and this solution is extremal.
    \item[(ii)] $\mr{supp} (b) = \J$, $\scr{H} (b)$ has a unique contractive Gleason solution, and this solution is extremal.
    \item[(iii)] There is no non-zero $g \in \J$ so that $b g \in \scr{H} (b)$.
    \item[(iv)] There is no non-zero $\J -$valued constant function $F \equiv g \in \scr{H} ^+ (H_{[b]} )$, $g \in \J$.
    \item[(v)] $K_0 ^{[b]} (I - b (0) ) \J \subseteq \ran{V^{[b]}}$.
\en
\end{thm}
Any Schur multiplier is said to be \emph{quasi-extreme} if it obeys the assumptions and equivalent conditions of this theorem. If, for example, $b$ is strictly contractive, then $[b]$ will be strictly contractive (and hence non-unital).
For conditions (iv) and (v) of the above theorem we are assuming that either $\J \subseteq \K$ or $\K \subseteq \J$. There is no loss of generality with this assumption, since it is easy to see that $b \in \scr{S} _d (\J , \K )$ is quasi-extreme if and only if every member of its coincidence class is quasi-extreme. In the particular case where $\J = \K$ so that $b = [b]$, items (iv) and (v) reduce to:
\bi
\item[$\mr{(iv)} ^\prime$] $\scr{H} ^+ (H_b)$ contains no constant functions.
\item[$\mr{(v)}^\prime$] $V^b$ is a co-isometry,
\ei
see \cite[Theorem 4.17]{JM}. Since the proof and proof techniques of Theorem \ref{equivqe} are very similar to those of \cite[Section 4]{JM}, we will not include it here. The equivalence of (iii) and (iv), for example, follows as in the proof of \cite[Theorem 3.22]{JM}. An arbitrary purely contractive $b \in \scr{S} _d (\J , \K)$ may still not satisfy the assumptions of Theorem \ref{equivqe}, \emph{i.e.} $[b]$ may not be non-unital, and so we define:
\begin{defn} \label{QEdef}
    A purely contractive $b \in \scr{S} _d (\J , \K )$ is \emph{quasi-extreme} if $b$ has a unique contractive Gleason solution, and this solution is extremal.
\end{defn}
In particular, the bijection between contractive Gleason solutions for $b$ and (the strictly contractive) $b ^{\ang{0}}$ of Lemma \ref{FSGSbij} implies:
\begin{cor} \label{PIQE}
    A purely contractive $b \in \scr{S} _d (\J , \K )$ is quasi-extreme if and only if the $\alpha$-Frostman shift $b^{\ang{\alpha}}$ is quasi-extreme for any
pure contraction $\alpha \in [\scr{L} (\J ,\K) ] _1$.
\end{cor}
In particular, $b$ is quasi-extreme if and only if the strictly contractive $b ^{\ang{0}}$ is quasi-extreme (so that $b^{\ang{0}}$ obeys the equivalent properties of Theorem \ref{equivqe}.)
\begin{lemma}
    If $b \in \scr{S} _d (\J , \K )$ is purely contractive and $\mr{supp} (b) = \J$, then $b$ is quasi-extreme if and only if $\scr{H} (b)$ has a unique contractive Gleason solution, and this solution is extremal.
\end{lemma}
\begin{proof}
This follows from the Formula (\ref{GSHbb}), as in \cite[Theorem 4.9, Theorem 4.4]{JM}.
\end{proof}

The next result will yield an abstract characterization of CCNC row contractions with quasi-extreme characteristic functions.
\begin{thm} \label{newQE}
Let $b \in \scr{S} _d (\J, \K)$ be a purely contractive Schur-class multiplier such that $\mr{supp} (b) = \J$. Then $b$ is quasi-extreme if and only if there
is an extremal Gleason solution, $X$, for $\scr{H} (b)$ so that
$$ \ker{X} ^\perp \subseteq \bigvee _{z \in \B ^d} z^* k_z ^b \K = \bigvee _{z\in B^d} z^* (I-Xz^* ) ^{-1} \ran{D_{X^*}}.  $$
\end{thm}

\begin{proof}
We will first prove that any purely contractive $b$ has this property if and only if $b^{\ang{0}}$ has this property. This will show that we can assume, without loss of generality, that $b$ is strictly contractive so that the equivalent conditions of Theorem \ref{equivqe} apply.

Given a purely contractive $b \in \scr{S} _d (\J , \K )$, with $\mr{supp} (b) = \J$, let $X = X (\mbf{b} )$ be a contractive and extremal Gleason solution for $\scr{H} (b)$. Recall that $X (\mbf{b} )$ is defined as in Formula (\ref{GSHbb}), and that since $\mr{supp} (b) = \J$, $X (\mbf{b} )$ is extremal if and only if $\mbf{b}$ is. Then
$$ \ker{X} ^\perp = \ov{\ran{X^* } } = \bigvee _{z \in \B ^d} \left( z^* k_z ^b - \mbf{b} b(z) ^* \right) \K, $$ and it follows that
$b$ will have the desired property if and only if
$$ \bigvee _{z \in \B ^d} \mbf{b} b(z)^* \K = \bigvee \mbf{b} \J \subseteq \bigvee _{w \in \B ^d} w^* k_w ^b \K. \quad \quad (\mr{supp} (b) = \J)$$
By the bijection between Gleason solutions for $b$ and $b^{\ang{0}}$, Lemma \ref{FSGSbij}, $b$ will have this property if and only if
\ba \bigvee _{z \in \B ^d} z^* k_z ^{\ang{0}} \K & \supseteq &(M ^{\ang{b(0)} } \otimes I_d ) ^* \mbf{b} \J; \quad \quad k^{\ang{0}} := k^{b^{\ang{0}}} \nn \\
& = & \mbf{b} ^{\ang{0}} \J, \nn \ea where $\mbf{b} ^{\ang{0}} := (M ^{\ang{b(0)} } \otimes I_d ) ^* \mbf{b} D_{b(0)} ^{-1}$ is a contractive and extremal Gleason solution for $b ^{\ang{0}}$.  As above it follows
that this happens if and only if $\ker{ X ^{\ang{0}} } ^\perp \subseteq \bigvee z^* k_z ^{\ang{0}} \K$, where $X^{\ang{0}} := X (\mbf{b} ^{\ang{0}})$, and $b$ has the desired property if and only if its Frostman shift $b^{\ang{0}}$ does. We can now assume without loss of generality that $b \in \scr{S} _d (\J , \K)$ is strictly contractive and that $b(0) =0$.

First suppose that $b$ is quasi-extreme (QE) so that $b$ has a unique contractive Gleason solution $\mbf{b}$ which is extremal, by Theorem \ref{equivqe}, and $X := X(\mbf{b} )$ is the unique contractive and extremal Gleason solution for $\scr{H} (b)$.
As in the first part of the proof, it follows that $b$ will have the desired property provided that
$ \mbf{b} \J \subseteq \bigvee _{z\in \B ^d} z^* k_z ^b \K.$
Assume without loss of generality that $\J \subseteq \K$ or $\K \subseteq \J$ and consider $a := [b]$, the (strictly contractive) square extension of $b$. As described in Subsection \ref{Gleasonsub}, any contractive Gleason solution for $a$ is given by
$$ \mbf{a} = (U_a ^* \otimes I_d) D ^* K_0 ^a , \quad \quad (a(0) =0)$$ where $D\supseteq V^a$ is a contractive extension of $V^a$ on $\scr{H} ^+ (H_a)$, the Herglotz space of $a$, and $U_a : \scr{H} (a) \rightarrow \scr{H} ^+ (H_a)$ is the canonical unitary multiplier of Lemma \ref{ontoisomult}. Choose $D = V^a$. In the first case where $\J \subset \K$, uniqueness of $\mbf{b}$ implies that
$ \mbf{b} = \mbf{a} | _\J $ so that
\ba \mbf{b} \J & \subseteq & (U_a ^* \otimes I_d) V_a ^* K_0 ^a \K \nn \\
& \subseteq & \bigvee _{z \in \B ^d } z^* U_a ^* K_z ^a \K \nn \\
& = & \bigvee _{z\in \B ^d} z^* k_z ^b (I-a(z) ^* ) ^{-1} \K = \bigvee z^* k_z ^b \K. \nn \ea
Similarly, in the second case where $\K \subseteq \J$,
\ba \mbf{b} \J & = & \bbm I _{\scr{H} (b)} \otimes  I_d , & 0 \ebm  \mbf{a} \J \nn \\
& \subseteq &  \bbm I, & 0 \ebm \bigvee _{z \in \B ^d} z^* k_z ^a \J \nn \\
& = &   \bigvee _{z \in \B ^d} z^* \bbm I , & 0 \ebm \bbm k_z ^b & 0 \\ 0 & k_z \otimes I _{\J \ominus \K}  \ebm \bbm \K \\ \J \ominus \K \ebm \nn \\
& = & \bigvee _{z \in \B ^d} z^* k_z ^b \K. \nn \ea

Conversely, suppose that $\scr{H} (b)$ has an extremal Gleason solution, $X$, with the desired property. Then it follows, as above, that
$X = X (\mbf{b} )$, where $\mbf{b}$ is an extremal Gleason solution for $b$ obeying
$\mbf{b} \J \subseteq \bigvee _{z \in \B ^d} z^* k_z ^b \K.$ By Remark \ref{sqGS}, setting $a = [b]$, we have that there is a contractive Gleason solution $\mbf{a}$ for $a$
such that either $\mbf{b} = \mbf{a} | _\J$ or $\mbf{b} = \bbm I_{\scr{H} (b)} \otimes I_d, & 0 \ebm \mbf{a}$.
Also, again by Subsection \ref{Gleasonsub}, there is a contractive extension $D \supseteq V^a$ so that $\mbf{a} = \mbf{a} [D]$.

Consider the first case where $\J \subseteq \K$.
It follows that $\mbf{b}$ has the form $\mbf{b} = \mbf{a}[D] | _\J$ so that
$$ D^* K_0 ^a  \J \subseteq \bigvee _{z \in \B ^d} z^* K_z ^a \K = \ker{V^a} ^\perp. \quad \quad \mbox{(Recall $b(0) =0 =a(0)$.)}$$ Since $D ^* = (V^a) ^* + C^* $ where $C ^* : \ran{V^a} ^\perp \rightarrow \ker{V^a}$ (by Lemma \ref{contractext}), it follows that
\ba  D^* K_0 ^a  \J & = &  P_{\ker{V^a} } ^\perp D ^* K_0 ^a  \J  \nn \\
& = & (V^a) ^* K_0 ^a \J, \nn \ea is contained in $\ker{V^a} ^\perp$. Since we assume $X$ and hence $\mbf{b}$ are extremal,
$$ 0 = P _\J  (K_0 ^a) ^* (I -V^a (V^a) ^* ) K_0 ^a  P_\J, $$ and it follows that
$$ K_0 ^a \J \subseteq \ran{V^a},$$ so that $b$ is quasi-extreme by Theorem \ref{equivqe}.

In the second case where $\K \subseteq \J$, we have that $\mbf{a} := \mbf{b} \oplus \mbf{0}$ is a contractive (and extremal) Gleason solution for $\scr{H} (a)$
so that there is a $D \supseteq V^a$ such that $\mbf{a} = \mbf{a} [D]$. As before
$$ \mbf{a} [D] \J \subseteq \bigvee _{z\in \B ^d} z^* k_z ^a \K, \quad \mbox{and} \quad D^* K_0 ^a (I -b(0) ) \J \subseteq \ker{V^a} ^\perp. $$ Again, the same argument as above implies that $b$ is QE.
\end{proof}

\subsection{de Branges-Rovnyak model for quasi-extreme row contractions}

\begin{defn}
    A CCNC row contraction $T : \mc{H} \otimes \C ^d \rightarrow \mc{H}$ with isometric-pure decomposition $T=V-C$ is said to be \emph{quasi-extreme} (QE) if its characteristic function coincides weakly with a QE Schur multiplier.
\end{defn}

We obtain a refined model for QE row contractions:

\begin{thm} \label{main2}
A row contraction $T : \mc{H} \otimes \C ^d \rightarrow \mc{H}$ is QE if and only if $T$ is unitarily equivalent to the (unique) contractive and extremal Gleason solution $X$ in a multi-variable de Branges-Rovnyak space $\scr{H} (b)$ for a quasi-extreme and purely contractive Schur multiplier $b$.

In particular, any QE row contraction, $T$, is unitarily equivalent to $X^{b_T}$ where $b_T$ is any characteristic function for $T$.
The characteristic function, $b_T$, of $T$, is a complete unitary invariant: Any two QE row contractions $T_1, T_2$ are unitarily equivalent if and only if their characteristic functions coincide weakly.
\end{thm}

\begin{proof}
    This follows from Theorem \ref{main1} under the added assumption that the characteristic function of $T$ is quasi-extreme. For the final statement simply note that if $b_1, b_2$ are quasi-extreme Schur functions that coincide weakly so that $\scr{H} (b_1) = \scr{H} (U b_2)$ for some unitary $U$, it is easy to see that $X^{b_1}$ is unitarily equivalent to $X ^{b_2}$ (via a constant unitary multiplier), where $X ^{b_1}, X^{b_2}$ are the unique, contractive, and extremal Gleason solutions for $\scr{H} (b_1 )$ and $\scr{H} (b_2)$, respectively.
\end{proof}

We will conclude with an abstract characterization of the class of QE row contractions:

\begin{thm} \label{QEconditionthm}
    A row contraction $T : \mc{H} \otimes \C ^d \rightarrow \mc{H}$ is QE if and only if
$$ \bigvee _{z \in \B ^d} (I -Tz^* ) ^{-1} \ran{D_{T^*} } = \mc{H}; \quad \quad T \ \mbox{is CCNC}, $$ and
$$ \ker{T} ^\perp  \subseteq \bigvee _{z \in \B ^d } z^* (I - Tz^* ) ^{-1} \ran{D_{T^*} }. \quad \quad \mbox{($T$ obeys the QE condition.)} $$
\end{thm}

\begin{proof}
Let $T$ be a QE row contraction on $\mc{H}$. By Theorem \ref{main2}, $T$ is unitarily equivalent to the unique contractive and extremal Gleason solution, $X^T$ for $\scr{H} (b_T)$. We can assume that $b_T = b_T |_{\mr{supp} (b_T )}$ so that $b_T$ is QE by Theorem \ref{main2}. By Theorem \ref{newQE},
$$ \ker{X^T} ^\perp \subseteq \bigvee _{z \in \B ^d } z^* (I - X^T z^* ) ^{-1} \ran{D_{(X^T)^*} }, $$ and it follows that $T \simeq X^T$ also obeys the QE condition.

Conversely suppose that $T$ is CCNC and $T$ obeys the QE condition. Then $T \simeq X^T$, an extremal Gleason solution in $\scr{H} (b_T)$. Again we can assume that $b_T = b_T |_{\mr{supp} (b_T)}$, and since $T$ obeys the QE condition, so does $X^T$. Theorem \ref{newQE} implies that $b_T$ is quasi-extreme so that
$T$ is QE.
\end{proof}

\begin{prop}
If $T$ is a QE row contraction on $\mc{H}$ with isometric-pure decomposition $T = V-C$, then its partial isometric part, $V$, is a QE row partial isometry.
\end{prop}
\begin{proof}
Since $\ker{T} \subseteq \ker{V}$ and $T \supseteq V$ is a QE contractive extension,
\ba \ker{V} ^\perp & \subseteq & \ker{T} ^\perp \subseteq \bigvee _{z \in \B ^d} z^* (I - Tz^*) ^{-1} \ran{D_{T^*} } \nn \\
& =  & \bigvee z^* (I - Vz^* ) ^{-1} \ran{D_{V^*}}, \nn \ea and $V$ also obeys the QE condition.
\end{proof}

On the other hand,
\begin{prop} \label{QExtprop}
Let $V$ be a QE row partial isometry with model triple $(\ga , \J , \K)$. If 
$\delta \in \scr{L} (\J , \K )$ is any pure contraction with $\ker{\delta } ^\perp \subseteq \mr{supp} (b_V ^\ga)$, then $T _{\delta} = V - \ga (0)  \delta \ga (\infty) ^*$ is a QE row contractive extension of $V$.
\end{prop}

\begin{lemma}
Let $b \in \scr{S} _d (\J, \K)$, $b(0) = 0$ be a Schur multiplier and $\delta \in [ \scr{L} (\J , \K ) ] _1$ be any pure contraction obeying $\ker{\delta } ^\perp \subseteq \mr{supp} (b)$. If $b' := b| _{\mr{supp} (b) } $ and $\delta ' := \delta | _{\mr{supp} (b)}$, then $\mr{supp} (b) = \mr{supp} (b^{\ang{\delta}} )$ and $b ^{\ang{\delta}} | _{\mr{supp} (b)} = (b ' ) ^{\ang{\delta '}}$.
\end{lemma}
\begin{proof}
Let $\mc{H} := \mr{supp} (b)$.  Writing elements of $\J = \mc{H} \oplus (\J \ominus \mc{H} )$ as two-component column vectors, let $\alpha := \delta | _{\mc{H}}$ and $a := b | _{\mc{H}}$. Recall that $b^{\delta} (z) = D_{\delta ^*} ^{-1} (b(z) + \delta ) (I _\J +\delta ^* b(z) ) ^{-1} D _\delta$.
Writing $$ \delta = \bbm \alpha , & 0 _{\J \ominus \mc{H} , \K } \ebm, \quad \mbox{and} \quad b = \bbm a, & 0 _{\J \ominus \mc{H}, \K } \ebm, $$ it is easy to check that $$ b^{\ang{\delta} } (z) P _{\J \ominus \mc{H} } = D_{\delta ^* } ^{-1} \bbm a(z) + \alpha, & 0 _{\J \ominus \mc{H} , \K } \ebm \bbm 0 & 0 \\ 0 & I _{\J \ominus \mc{H} } \ebm =0, $$ proving that $\mr{supp} ( b^{\ang{\delta}} ) \subseteq \mc{H} = \mr{supp} (b)$. The remaining assertions are similarly easy to verify.
\end{proof}
\begin{proof}{ (of Proposition \ref{QExtprop})}
If $V$ is a QE row partial isometry, and $(\ga, \J  , \K )$ is any model triple for $V$, then $b:= (b_V ^\ga ) | _{\mr{supp} (b_V ^\ga ) }$ is quasi-extreme. Given any pure contraction $\delta \in [ \scr{L} (\J , \K ) ]_1$, we can define, as in Lemma \ref{zeropoint}, the CCNC row contraction $T = T _{\delta} := V - \ga (0)  \delta \ga (\infty) ^*$, which, by definition, has the characteristic function
$b_T ^\ga = (b_V ^\ga ) ^{\ang{\delta}}$. Under the assumption that $\ker{\delta } ^\perp \subseteq \mr{supp} (b_V ^\ga)$, the above lemma proves that $b_T ^\ga$ coincides weakly with a Frostman shift of the quasi-extreme Schur-class function $b$, so that $T$ is also QE.
\end{proof}

If, however, $\ker{\delta} ^\perp$ is not contained in $\mr{supp} (b_V ^\ga)$, $T _\delta$ can fail to be QE. That is, as the following simple example shows, there exist CCNC row contractions $T$ with partial isometric part $V$ such that $V$ is QE but $T$ is not.
\begin{eg}
Let $b \in \scr{S} _d (\mc{H})$ be any purely contractive quasi-extreme multiplier and set
$$ B := \bbm b & 0 \\ 0 & 0 \ebm \in \scr{S} _d (\mc{H} \oplus \C ), \quad \quad \delta := \bbm 0 & 0 \\ 0& r \ebm; \quad 0<r<1. $$
Then,
\ba B^{\ang{\delta}} (z) & = & D_{\delta ^* } ^{-1} ( B(z) + \delta ) ( I + \delta ^* B(z) ) ^{-1} D_{\delta} \nn \\
& = & \bbm b(z) & 0 \\ 0 & r \ebm, \nn \ea which cannot be quasi-extreme since $0<r<1$.
\end{eg}

\section{Outlook} \label{Outlook}

Motivated by the characterization of CNC row contractions in Section \ref{CNCsection}, given any CNC row partial isometry, $V$, on $\mc{H}$, it is natural to extend our definition of model triple and model map to the non-commutative setting of non-commutative function theory \cite{KVV,Ball-NC}.

Namely, recall that the non-commutative (NC) open unit ball is the disjoint union $\B ^d _{\N}:= \coprod _{n=1} ^\infty \B ^d _n$, where
$$ \B ^d _n := \left( \C ^{n \times n } \otimes \C _d \right) _1, $$ is viewed as the set of all strict row contractions (with $d$ component operators) on $\C ^n$, and $\B ^d _1 \simeq \B ^d$.
A natural extension of our concept of model triple to the NC unit ball, $\B ^d _{\N}$, would be a triple $(\ga , \J _\infty , \J _0 )$ consisting of two Hilbert spaces $\J _\infty \simeq \ker{V}$, $\J _0 \simeq \ran{V} ^\perp$, and a map $\ga $ on $\B ^d _{\N}\cup \{ \infty \}$,
$$ \ga : \left\{  \arraycolsep=2pt\def\arraystretch{1.2} \begin{array}{ccc} Z \in \B ^d _n & \mapsto & \ga (Z) \in \scr{L} (\J _0 \otimes \C ^n , \ra{V-Z} ^\perp ) \\
\{ \infty \} & \mapsto &\ga (\infty ) \in \scr{L} (\J _\infty , \ker{V} )  \end{array} \right., $$ where $\ga (Z)$ is an isomorphism for each $Z \in \B ^d _{\N}$ and $\ga (0_n) , \ga (\infty )$ are onto isometries. We will call such a model map $\ga$ a non-commutative (NC) model map.
In particular, as in Section \ref{modelmapsect}, if $T \supseteq V$ is any contractive extension of $V$, and $\Ga _T (0) : \J _0 \rightarrow \ran{V} ^\perp$, $\Ga _T (\infty ) : \J _\infty \rightarrow \ker{V}$ are any fixed onto isometries, then
$$ \Ga _T (Z) := (I -TZ^*)^{-1} (\Ga _T (0) \otimes I_n); \quad \quad Z \in \B ^d _n, \quad \quad TZ^* :=T_1 \otimes Z_1 ^* +...+ T_d \otimes Z_d ^*,$$ defines an analytic NC model map for $V$ (we expect $\Ga _T (Z)$ will be anti-analytic in the sense of non-commutative function theory \cite[Chapter 7]{KVV}). Moreover, as in Section \ref{modelmapsect}, for any analytic NC model map $\Ga$, we expect that one can then define an abstract model space, $\hat{\mc{H} } ^\Ga$ with non-commutative reproducing kernel $$ \hat{K} ^\Ga (Z,W) = \Ga (Z) ^* \Ga (W) ; \quad \quad Z, W \in \B ^d _n, $$ and that this will be a non-commutative reproducing kernel Hilbert space (NC-RKHS) in the sense of \cite{Ball-NC,Ball2003rkhs}. If this analogy continues to hold, it would be natural to use $\Ga$ to define a NC \emph{characteristic function}, $B_T (Z)$, on $\B ^d _{\N}$, and one would expect this to be an element of the free (left or right) Schur class of contractive NC multipliers between vector-valued Fock spaces over $\C ^d$ \cite{Ball2006Fock,JMfree}.  Ultimately, it would be interesting to investigate whether such an extended theory will yield an alternate approach to the NC de Branges-Rovnyak model for CNC row contractions as (adjoints of) the restriction of the adjoint of the left or right free shift on (vector-valued) full Fock space over $\C ^d$ to the right or left non-commutative de Branges-Rovnyak spaces, $\scr{H} ^L (B _T)$ or $\scr{H} ^R (B_T)$ \cite{Ball2006Fock,Ball-NC,Ball2003rkhs}.

\small


\end{document}